\documentclass{article}
\usepackage{graphicx} 
\usepackage{amsthm}
\usepackage{amsfonts}
\usepackage{amssymb}
\usepackage{hyperref}
\usepackage{graphicx}
\usepackage{rotating}
\usepackage{wrapfig}
\usepackage{subcaption}
\usepackage{float} 
\usepackage{soul}

\usepackage{tikz}
\usepackage{tkz-euclide}
\usetikzlibrary{calc, patterns, angles, quotes}
\usetikzlibrary{shapes.geometric, arrows}
\usetikzlibrary{calc,perspective}
\usepackage{euscript}
\usepackage{mathrsfs}
\usepackage{tikz-cd}
\usepackage{pgfplots}
\usepackage{venndiagram}
\usepackage{mathtools}
\usepackage{float}
\usepackage{graphicx}
\usepackage[all]{xy} 
\usepackage[margin=1in]{geometry}
\usepackage{xcolor}
\usepackage[none]{hyphenat}
\usepackage{systeme}
\usepackage{hyperref}

\usepackage[all]{xy}
\usepackage{pgf,tikz}
\usetikzlibrary{arrows}
\usetikzlibrary[patterns]

\usepackage{amstext, amsmath, mathtools, latexsym} 

\newcommand{\N}{\mathbb{N}}
\newcommand{\Z}{\mathbb{Z}}

\newtheorem{prop}{Proposition}[section]
\newtheorem{defi}[prop]{Definition}
\newtheorem{cor}[prop]{Corollary}
\newtheorem{teo}[prop]{Theorem}
\newtheorem{lema}[prop]{Lemma}

\newtheorem*{quest}{Question}
\newtheorem*{obs}{Remark}

\newtheoremstyle{bloco}
  {\topsep}   
  {\topsep}   
  {\upshape}  
  {0pt}       
  {\bfseries} 
  {\!\!\!\!}         
  {5pt plus 1pt minus 1pt} 
\makeatother

\theoremstyle{bloco}
\newtheorem{bloco}{}

\newcommand\restr[2]{{
  \left.\kern-\nulldelimiterspace 
  #1 
  \vphantom{\big|} 
  \right|_{#2} 
  }}

\definecolor{mycolor1}{RGB}{74,144,226}
\definecolor{mycolor2}{RGB}{208,2,27}

\title{Entropy for Generalized Parabolic Dynamics}
\author{Frederico A. C. L. Marinho, Hellen de Paula, and Lucas H. R. de Souza}

\begin{document}

\maketitle

\begin{abstract}
In this paper, we extend the concept of generalized entropy to uniform spaces, allowing computations beyond metrizable settings. We apply this to parabolic dynamics - systems with a unique fixed point uniformly attracting all compact subsets in both time directions - and introduce a broader class, called generalized parabolic dynamics. Within this, we identify a significant subclass and prove its linear entropy, offering several equivalent characterizations linking the generalized entropy of the space, the non-wandering set, and families of mutually singular subsets. We also study homeomorphisms of compact surfaces with a singleton non-wandering set but non-parabolic dynamics. The examples that we present have at least quadratic entropy, bounded above by the supremum of polynomial growth rates. For any growth rate with the linear invariant property within these bounds, we construct a homeomorphism whose generalized entropy realizes the prescribed growth, with the non-wandering set reduced to a point.
\end{abstract}

\

\textbf{Keywords} Generalized entropy,  parabolic dynamics, north-south dynamics, dynamics on surfaces \ 

\

\textbf{Mathematics Subject Classification (2020)} Primary: 37B40, 37B02. Secundary: 20F65.

\




\pagebreak

 \tableofcontents
 \addcontentsline{toc}{section}{Introduction}

\vspace{8cm}

\textbf{Declarations of interest:} 
\noindent 

The first author was supported by CNPq and FAPEMIG (project 29847-001/2022).

The second author was supported by the joint FAPEMIG/CNPq program ``Programa de Apoio à Fixação de Jovens Doutores no Brasil", Grant Number BPD-00181-22 (FAPEMIG).

The third author was supported by the joint FAPEMIG/CNPq program ``Programa de Apoio à Fixação de Jovens Doutores no Brasil", Grant Numbers BPD-00721-22 (FAPEMIG), 150784/2023-6 (CNPq).

\pagebreak
\section*{Introduction}

Topological entropy is one of the most classical invariants used to quantify the complexity of a dynamical system. Given $f:X \to X$, a continuous map on a compact metric space, the topological entropy, $h(f)$, captures the exponential growth rate of orbit separation over time. Despite its importance, topological entropy fails to distinguish systems whose complexity grows slower than exponentially. In particular, many systems of interest — such as those with parabolic or north-south dynamics — have zero topological entropy despite exhibiting nontrivial dynamical behavior. In order to address this limitation, the notion of polynomial entropy, $h_{pol}(f)$, introduced by J. P. Marco in \cite{Mar}, has been proposed as a refinement that measures polynomial growth rates. Polynomial entropy has proven useful in distinguishing between zero-entropy systems that exhibit different orders of complexity. However, it still reduces the complexity to a single real number and thus lacks sensitivity to more nuanced distinctions. To overcome this, Correa and Pujals, in \cite{CoPu21}, introduced the concept of generalized entropy, $o(f)$, a more refined invariant that associates to each dynamical system an equivalence class in a complete space of orders of growth, $\overline{\mathbb{O}}$. The space of orders of growth, $\mathbb{O}$, consists of equivalence classes of non-decreasing sequences of non-negative real numbers, where two sequences are considered equivalent if they exhibit the same asymptotic behavior. This space is then completed in the Dedekind–MacNeille sense, resulting in a complete lattice, $\overline{\mathbb{O}}$, which allows the comparison and the classification of some different growth rates. As shown in \cite{CoPu21}, both topological and polynomial entropy can be recovered as projections of generalized entropy onto the exponential and polynomial families.

In the first part of this paper, we extend the definition of generalized entropy to the context of uniform spaces, allowing for entropy, computations in settings that are not necessarily metrizable. We define both generator and separated sets versions of the generalized entropy for continuous maps and prove their equivalence (see \textbf{Section \ref{Section Entropy for uniform spaces}}). We also verify invariance under uniform conjugacies and compatibility with restrictions to invariant subspaces. In addition, we check that some basic properties of generalized entropy for metric spaces also work in our context.

The second part of the paper is devoted to the study of generalized parabolic dynamics. A homeomorphism with generalized parabolic dynamics can be defined as a system $f: X \to X$ for which there exists a closed invariant set $F \subseteq X$, which we call the parabolic set of the map $f$, such that every compact subset of $X-F$ converges uniformly to $F$ under both forward and backward iteration. When $F$ is a single fixed point, we refer to $f$ as having parabolic dynamics (note that our definition of parabolic dynamics is different and is not equivalent to the definition that is often used in Complex Dynamics). When $F$ is a pair of points such that one of them is an attractor point and the other one is a repulsor point, then we say that $f$ has north-south dynamics.

Parabolic and north-south dynamics arise naturally in geometry. For instance, in the Hyperbolic Geometry, they arise in the following way: Every isometry of the hyperbolic space $\mathbb{H}^n$ extends to a homeomorphism of its compactification $\overline{\mathbb{H}}^n = \mathbb{H}^n \cup \partial\mathbb{H}^n$, where $\partial\mathbb{H}^n$ denotes its ideal boundary. Such isometries are classified into three types: elliptic, if they have a fixed point in $\mathbb{H}^n$; parabolic, if they have a unique fixed point in $\partial\mathbb{H}^n$; and loxodromic, if they have exactly two fixed points in $\partial\mathbb{H}^n$. The parabolic isometries have parabolic dynamics on $\overline{\mathbb{H}}^{n}$ and loxodromic isometries have north-south dynamics on $\overline{\mathbb{H}}^{n}$. See \cite{Kap} and  \cite{Rat} for details on Hyperbolic Geometry. \textbf{Figure \ref{fig:parabolicnorthsouthhyperbolic}} illustrates both types of dynamical behavior in $\overline{\mathbb{H}}^2$.


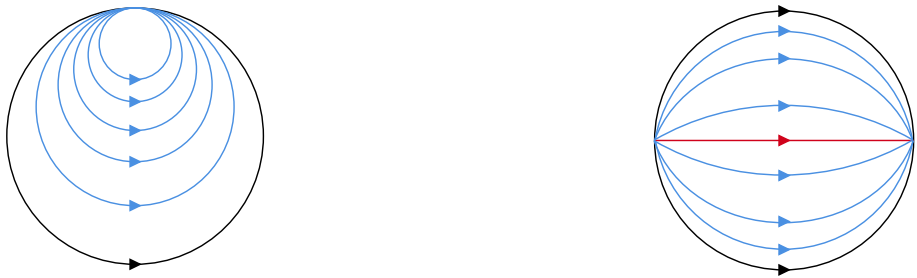
\begin{figure}[H]
\centering
\begin{subfigure}[t]{.48\textwidth}
  \centering
      \resizebox{3.6cm}{!}{%

\tikzset{every picture/.style={line width=0.75pt}} 

\begin{tikzpicture}[x=0.75pt,y=0.75pt,yscale=-1,xscale=1]

\draw   (216,108.5) .. controls (216,59.14) and (256.02,19.12) .. (305.38,19.12) .. controls (354.74,19.12) and (394.76,59.14) .. (394.76,108.5) .. controls (394.76,157.86) and (354.74,197.88) .. (305.38,197.88) .. controls (256.02,197.88) and (216,157.86) .. (216,108.5) -- cycle ;
\draw  [color={rgb, 255:red, 74; green, 144; blue, 226 }  ,draw opacity=1 ] (236.44,88.06) .. controls (236.44,49.99) and (267.31,19.12) .. (305.38,19.12) .. controls (343.45,19.12) and (374.32,49.99) .. (374.32,88.06) .. controls (374.32,126.13) and (343.45,157) .. (305.38,157) .. controls (267.31,157) and (236.44,126.13) .. (236.44,88.06) -- cycle ;
\draw  [color={rgb, 255:red, 74; green, 144; blue, 226 }  ,draw opacity=1 ] (251.75,72.75) .. controls (251.75,43.13) and (275.76,19.12) .. (305.38,19.12) .. controls (335,19.12) and (359.01,43.13) .. (359.01,72.75) .. controls (359.01,102.37) and (335,126.38) .. (305.38,126.38) .. controls (275.76,126.38) and (251.75,102.37) .. (251.75,72.75) -- cycle ;
\draw  [color={rgb, 255:red, 74; green, 144; blue, 226 }  ,draw opacity=1 ] (262.63,61.87) .. controls (262.63,38.26) and (281.77,19.12) .. (305.38,19.12) .. controls (328.99,19.12) and (348.13,38.26) .. (348.13,61.87) .. controls (348.13,85.48) and (328.99,104.62) .. (305.38,104.62) .. controls (281.77,104.62) and (262.63,85.48) .. (262.63,61.87) -- cycle ;
\draw  [color={rgb, 255:red, 74; green, 144; blue, 226 }  ,draw opacity=1 ] (272.69,51.81) .. controls (272.69,33.76) and (287.33,19.12) .. (305.38,19.12) .. controls (323.43,19.12) and (338.07,33.76) .. (338.07,51.81) .. controls (338.07,69.86) and (323.43,84.5) .. (305.38,84.5) .. controls (287.33,84.5) and (272.69,69.86) .. (272.69,51.81) -- cycle ;
\draw  [color={rgb, 255:red, 74; green, 144; blue, 226 }  ,draw opacity=1 ] (280.38,44.12) .. controls (280.38,30.31) and (291.57,19.12) .. (305.38,19.12) .. controls (319.19,19.12) and (330.38,30.31) .. (330.38,44.12) .. controls (330.38,57.93) and (319.19,69.12) .. (305.38,69.12) .. controls (291.57,69.12) and (280.38,57.93) .. (280.38,44.12) -- cycle ;
\draw  [fill={rgb, 255:red, 0; green, 0; blue, 0 }  ,fill opacity=1 ] (308.89,197.88) -- (301.87,201.28) -- (301.87,194.48) -- cycle ;
\draw  [color={rgb, 255:red, 74; green, 144; blue, 226 }  ,draw opacity=1 ][fill={rgb, 255:red, 74; green, 144; blue, 226 }  ,fill opacity=1 ] (308.89,126.38) -- (301.87,129.78) -- (301.87,122.98) -- cycle ;
\draw  [color={rgb, 255:red, 74; green, 144; blue, 226 }  ,draw opacity=1 ][fill={rgb, 255:red, 74; green, 144; blue, 226 }  ,fill opacity=1 ] (308.89,157) -- (301.87,160.4) -- (301.87,153.6) -- cycle ;
\draw  [color={rgb, 255:red, 74; green, 144; blue, 226 }  ,draw opacity=1 ][fill={rgb, 255:red, 74; green, 144; blue, 226 }  ,fill opacity=1 ] (308.89,104.62) -- (301.87,108.02) -- (301.87,101.22) -- cycle ;
\draw  [color={rgb, 255:red, 74; green, 144; blue, 226 }  ,draw opacity=1 ][fill={rgb, 255:red, 74; green, 144; blue, 226 }  ,fill opacity=1 ] (308.89,84.5) -- (301.87,87.9) -- (301.87,81.1) -- cycle ;
\draw  [color={rgb, 255:red, 74; green, 144; blue, 226 }  ,draw opacity=1 ][fill={rgb, 255:red, 74; green, 144; blue, 226 }  ,fill opacity=1 ] (308.89,69.12) -- (301.87,72.52) -- (301.87,65.72) -- cycle ;

\end{tikzpicture}

}

\end{subfigure}\hfill
\begin{subfigure}[t]{.48\textwidth}
  \centering
\resizebox{3.7cm}{!}{

\tikzset{every picture/.style={line width=0.75pt}} 
    
\begin{tikzpicture}[x=0.75pt,y=0.75pt,yscale=-1,xscale=1, line width=0.75pt]
  \useasboundingbox (214,20) rectangle (392.76,199.04);

  \draw (214,109.66) .. controls (214,60.3) and (254.02,20.28) .. (303.38,20.28)
        .. controls (352.74,20.28) and (392.76,60.3) .. (392.76,109.66)
        .. controls (392.76,159.03) and (352.74,199.04) .. (303.38,199.04)
        .. controls (254.02,199.04) and (214,159.03) .. (214,109.66) -- cycle;

  \draw [color={rgb,255:red,208; green,2; blue,27}] (214,109.66) -- (392.76,109.66);

  \foreach \p in {
    {(213.84,109.96) .. controls (225.63,77.09) and (261.1,53.22) .. (303.03,53.22) .. controls (344.96,53.22) and (380.44,77.09) .. (392.23,109.96)},
    {(392.59,109.78) .. controls (380.73,142.61) and (345.15,166.44) .. (303.1,166.44) .. controls (261.05,166.44) and (225.47,142.61) .. (213.61,109.78)},
    {(214.35,109.21) .. controls (221.74,66.72) and (258.68,34.41) .. (303.14,34.41) .. controls (347.61,34.41) and (384.55,66.72) .. (391.94,109.21)},
    {(392.23,109.96) .. controls (384.84,152.46) and (347.9,184.76) .. (303.43,184.76) .. controls (258.97,184.76) and (222.03,152.46) .. (214.64,109.96)},
    {(213.72,109.39) .. controls (239.87,94.21) and (270.23,85.48) .. (302.62,85.4) .. controls (335.01,85.31) and (365.38,93.88) .. (391.55,108.92)},
    {(391.66,109.59) .. controls (365.5,124.75) and (335.13,133.46) .. (302.75,133.53) .. controls (270.36,133.6) and (239.99,125.01) .. (213.84,109.96)}
  } {
    \draw[color={rgb,255:red,74; green,144; blue,226}] \p;
  }

  \foreach \y in {20.28, 199.04} {
    \draw[fill=black, draw=black] (306.89,\y) -- (299.87,\y+3.4) -- (299.87,\y-3.4) -- cycle;
  }

  \foreach \y in {34.04, 53.04, 85.8, 133.86, 165.61, 185.04} {
    \draw[fill={rgb,255:red,74; green,144; blue,226}, draw={rgb,255:red,74; green,144; blue,226}]
      (306.89,\y) -- (299.87,\y+3.4) -- (299.87,\y-3.4) -- cycle;
  }

  \draw[fill={rgb,255:red,208; green,2; blue,27}, draw={rgb,255:red,208; green,2; blue,27}]
    (306.89,109.66) -- (299.87,113.06) -- (299.87,106.26) -- cycle;

\end{tikzpicture}
 }

\end{subfigure}
\caption{Representation of the dynamics of the parabolic and loxodromic isometries, respectively, of the hyperbolic plane. In the first one, the blue lines represent horocycles that are invariant under the isometry. In the second one, the red line represents the fixed geodesic, and the blue lines are equidistant sets of the geodesic, which are also invariant under the isometry.}
\label{fig:parabolicnorthsouthhyperbolic}
\end{figure}


Examples of generalized parabolic dynamics with finite parabolic sets can be constructed easily by gluing a finite set of parabolic and north-south dynamics by their respective parabolic sets.

We show that systems with generalized parabolic dynamics and finite parabolic set have linear generalized entropy:

\begin{bloco}{\textbf{Theorem \ref{parabolicshavelinearentropy}}} \textit{If $(X,\mathcal{U})$ is a Hausdorff compact uniform space and $f: X \rightarrow X$ has generalized parabolic dynamics with finite parabolic set, then $o(f)=[n]$.}
\end{bloco}

This characterizes parabolic dynamics as systems of nonzero dynamical complexity within the generalized entropy framework. We also provide a full characterization of generalized parabolic dynamics with finite parabolic set in terms of generalized entropy:

\begin{bloco}{\textbf{Theorem \ref{characterizationofparabolicsviaentropy}}} \textit{Let $(X,\mathcal{U})$ be a Hausdorff compact space and $f: X \rightarrow X$ a homeomorphism such that the non-wandering set $\Omega(f)$ is finite and each point of $\Omega(f)$ is fixed by $f$. Then, the following statements are equivalent:}

\begin{enumerate}
    \item \textit{$f$ has generalized parabolic dynamics with parabolic set $\Omega(f)$.}
    \item The quotient map to the space of orbits $X-\Omega(f) \rightarrow (X-\Omega(f))/\langle f \rangle$ is a covering map.
    \item \textit{For every compact set $K \subseteq X-\Omega(f)$, $o(f,K) = 0$ and  $o(f^{-1},K) = 0$.}
    \item \textit{$o(f|_{X-\Omega(f)}) = 0$ and $o(f^{-1}|_{X-\Omega(f)}) = 0$, where $X-\Omega(f)$ has the subspace uniform structure from $X$.}
    \item \textit{For every compact set $K \subseteq X-\Omega(f)$, $f$ and $f^{-1}$ are Lyapunov stable in $K$.}
    \item \textit{For every point $x \in X- \Omega(f)$,  $x$ is a regular point of $f$.}
\end{enumerate}

\textit{Moreover, if $X-\Omega(f)$ is connected and locally connected, then there are two more equivalences:}

\begin{enumerate}
    \item[7.] \textit{$o(f) = [n]$}
    \item[8.] \textit{There are no non-trivial families of mutually singular compact sets.}
\end{enumerate}
\end{bloco}

Where the respective definitions of Lyapunov stable maps is given in \textbf{Definition \ref{definitionoflyapunov}}, regular points is given in \textbf{Definition \ref{regularpoints}} and mutually disjoint sets is given in \textbf{Definition \ref{mutuallysingular}}. The implications $1) \Leftrightarrow 5) \Leftrightarrow 6)$ are well known. They are due to Kinoshita (Lemma 1 and Lemma 2 of \cite{Ki}) for the case of $\#\Omega(f) = 1$ and metrizable $X$. Our contribution here is to add the generalized entropy to this equivalence.

In particular, the theorem above shows that a homeomorphism has generalized parabolic dynamics with finite parabolic set if and only if both its generalized entropy and that of its inverse vanish outside the non-wandering set. Moreover, under mild topological conditions, a homeomorphism has generalized parabolic dynamics with finite parabolic set if and only if it has linear entropy. This characterization highlights the strength of generalized entropy in detecting subtle asymptotic behaviors beyond those detected by topological or polynomial entropy. 

There are strong topological obstructions to the existence of generalized parabolic dynamics. For instance, if $M$ is a compact manifold without boundary and dimension bigger than $1$ that admit such dynamics with finite parabolic set, then its parabolic set has at most two points and $M$ is homeomorphic to a $n$-sphere (Theorem A of \cite{HL} and Theorem 6 of \cite{Kau}). In these references, they work in terms of regularity of points instead of generalized parabolic dynamics, but the equivalence is given in \textbf{Theorem \ref{characterizationofparabolicsviaentropy}} for finite non-wandering sets. Such maps and their topological obstructions were extensively studied, and we cite \cite{BK}, \cite{HKi}, and \cite{Ki}.

A bigger class of maps is the class of homeomorphisms such that their non-wandering sets are finite. They behave quite similarly to generalized parabolic dynamics, but in general, they do not need to have uniform convergence on compact sets, as we illustrate using Brouwer’s example on the sphere (\textbf{Section \ref{Brouwer}}). If such maps are not generalized parabolic, then it follows by \textbf{Theorem \ref{characterizationofparabolicsviaentropy}} that their generalized entropy must be strictly bigger than linear. Also, as a consequence of Theorem 2 of \cite{CP2} and Lemma 3.2 of \cite{CP2}, their generalized entropy must be at most the supremum of the polynomial family $\mathbb{P}$. Let $\mathbb{L}$ be the set of elements $[a(n)]$ of $\mathbb{O}$ that satisfies the linearly invariant property (i.e., there exist $m \geqslant 2$ such that $[a(n)] = [a(mn)]$) and $\overline{\mathbb{L}} = \{\sup{\Gamma}: \Gamma \subseteq \mathbb{L}, \# \Gamma \leqslant \aleph_{0}\}$. Then, Theorem 1.2 of \cite{CP} (which generalizes a result of \cite{HR}) says that, given $o \in \overline{\mathbb{L}}$, with $[n^{2}] \leqslant o \leqslant \sup \mathbb{P}$, there is a homeomorphism of the $2$-sphere with a single point as its non-wandering set such that its generalized entropy is $o$.

Despite the obstruction that appeared for parabolic maps on compact surfaces, we show that the same does not happen when we consider homeomorphisms with a single point as its non-wandering set, generalizing the result given in \cite{CP}:

\begin{bloco}{\textbf{Theorem \ref{singlenonwanderingonallsurfaces}}} \textit{Let $S$ be a compact surface without boundary and let $o \in \overline{\mathbb{L}}$, with $[n^{2}] \leqslant o \leqslant \sup \mathbb{P}$. Then there exists a homeomorphism $f: S \rightarrow S$ such that $\#\Omega(f) = 1$ and $o(f) = o$.} 
\end{bloco}

The topological restrictions for generalized parabolic dynamics that we mentioned above make it seem like there are just a few cases of generalized parabolic dynamics with finite parabolic set, but that is not the case. We end this paper presenting several examples of parabolic and north-south dynamics. Some of them are taken in non-metrizable spaces such as the double arrow space, which gives a justification for our choice to extend generalized entropy for uniform spaces. Suspension spaces and one-point compactifications can be sources of examples of north-south dynamics and parabolic dynamics, respectively, and they can be constructed from  metrizable or even non-metrizable spaces.

A source for many interesting examples of parabolic and north-south dynamics is given by convergence actions. These are actions on Hausdorff compact spaces, usually metrizable, with good dynamical properties (see \textbf{Section \ref{convergence}} for the definition). It was defined by Gehring and Martin \cite{GM} and it generalizes the actions of Kleinian groups on its limit sets inside the boundary of the hyperbolic space.  Some examples of such actions are: the action of a finitely generated group on its space of ends (Proposition 9.3.2 of \cite{Ge2}), a hyperbolic group on its Gromov boundary (Proposition 2.13 of \cite{Bo2}) and a relatively hyperbolic group on its Bowditch boundary (see, for example, Lemma 2.11 of \cite{Bo2}). Some examples of topological spaces that admit such actions are the Cantor set, the Sierpi\'nski carpet (page 651 of \cite{KK}), the Menger sponge (Theorem 14 of \cite{KK}), the Apollonian gasket (Lemma 3.4 of \cite{HPW}), and a tree of circles (page 4060 of \cite{Bo5}). So all of such spaces, and many others, are spaces that admit a north-south dynamics or a parabolic dynamics (as a consequence of Lemma 3.1 of \cite{Bo2}). 

The structure of the paper is as follows. In \textbf{Section \ref{2}}, we review the necessary preliminaries on uniform spaces, growth orders, and generalized entropy. In \textbf{Section \ref{Section Entropy for uniform spaces}}, we extend the definition of generalized entropy to uniform spaces, establish its main properties, which are analogous to the metric settings, and review coding techniques for mutually singular sets, extending them also to the context of uniform spaces and adding some new methods which are necessary for our results. \textbf{Section \ref{4}} applies this framework to generalized parabolic dynamics, proving key characterizations and establishing the linearity of entropy in this setting. In \textbf{Section \ref{5}}, we analyze dynamics on compact surfaces, constructing explicit examples of homeomorphisms with prescribed generalized entropy. Finally, in \textbf{Section \ref{6}}, we present several additional examples of parabolic and north-south dynamics, including cases in non-metrizable spaces, to illustrate the breadth of the theory.

\section*{Acknowledgements}
The first author would like to express his deepest gratitude to his advisor, Professor Ezequiel Barbosa, for his guidance, support, and invaluable insights throughout the Masters. He is also sincerely thankful to Professor Csaba Schneider for providing the grant scholarship that made this research possible. Special thanks are also due to Gabriel Chaves for his assistance with the coding involved in this project.

The second and third authors would like to thank Javier Correa, Pablo D. Carrasco, and Rafael da Costa Pereira for the valuable discussions, which were quite helpful to the well understanding of the theory.

\pagebreak

\section{Preliminaries}\label{2}

In this section, we gather the fundamental concepts and definitions needed for the development of the main results of this work. We begin by presenting uniform structures and some of their properties relevant to the definition of generalized entropy in spaces that are not necessarily metrizable. Next, we review the notion of growth orders and the formal definition of generalized entropy as introduced by Correa and Pujals, \cite{CoPu21}, highlighting its compatibility with both topological and polynomial entropy.

\subsection{Uniform spaces}

Here, we present some basic facts about uniform space that are useful throughout this paper. See \cite{Bou} for more details about uniform structures.

Let $X$ be a set. A uniform structure on $X$ is a set $\mathcal{U}$ of subsets of $X \times X$ satisfying:

    \begin{enumerate}
        \item $\mathcal{U} \neq \emptyset$,
        \item If $u \in \mathcal{U}$, then the diagonal $\Delta X = \{(x,x) \in X \times X\}$ is contained in $u$,
        \item If $u \in \mathcal{U}$ and $u \subseteq v \subseteq X \times X$, then $v \in \mathcal{U}$,
        \item If $u,v \in \mathcal{U}$, then $u\cap v \in \mathcal{U}$,
        \item If $u \in \mathcal{U}$, then there exists $v \in \mathcal{U}$ such that $v^{2} = \{(x,y): \exists z \in X: (x,z),(z,y)\in v\} \subseteq u$,
        \item If $u \in \mathcal{U}$, then $u^{-1} = \{(y,x): (x,y) \in u\}\in \mathcal{U}$.   
    \end{enumerate}
    
    In this case, we say that the pair $(X,\mathcal{U})$ is a uniform space.

Let $(X,\mathcal{U})$ and $(Y,\mathcal{V})$ be uniform spaces. A map $f: X \rightarrow Y$ is uniformly continuous if for every $v \in \mathcal{V}$, $f^{-1}(v) \in \mathcal{U}$ (here $f^{-1}$ means $(f\times f)^{-1}$).

If $(X,d)$ is a metric space, then we define the set $\mathcal{U}_{d}$ as the smallest uniform structure containing the set $\{u_{\epsilon}: \epsilon > 0\}$, where $u_{\epsilon} = \{(x,y) \in X\times X: d(x,y) < \epsilon\}$. Between two metric spaces, the notions of uniform continuity from the metrics and from the uniform structures coincide. 

Associated with a uniform space $(X,\mathcal{U})$, we get a topology for $X$ and if a map between uniform spaces is uniformly continuous, then it is continuous with respect to their associated topologies. A uniform space $(X,\mathcal{U})$ is Hausdorff, with respect to its associated topology, if and only if $\bigcap \mathcal{U} = \Delta X$. If $(X,d)$ is a metric space, then the associated topologies from $d$ and $\mathcal{U}_{d}$ coincide. If $X$ is a Hausdorff compact space, then there is a unique uniform structure such that its associated topology coincides with the topology of $X$. Moreover, any continuous map from a Hausdorff compact space to a uniform space is uniformly continuous. Another property for Hausdorff compact uniform $X$ spaces is the Lebesgue covering lemma: if $\{U_{i}\}_{i \in \Gamma}$ is an open covering for $X$, then there exists a Lebesgue entourage $u$ for the covering, i.e., if $S$ is a $u$-small subset of $X$, then there exists $i \in \Gamma$ such that $S \subseteq U_{i}$. 

Let $(X,\mathcal{U})$ be a uniform space and $\mathcal{B}\subseteq \mathcal{U}$. We say that $\mathcal{B}$ is a base of the uniform structure $\mathcal{U}$ (also called a fundamental system of entourages) if for every $u \in \mathcal{U}$, there exists $v \in \mathcal{B}$ such that $v \subseteq u$. If $(X,d)$ is a metric space, then $\{u_{\epsilon}: \epsilon > 0\}$ is a base for the uniform structure of $\mathcal{U}_{d}$.

If $(X,\mathcal{U})$ is a uniform space, $u \in \mathcal{U}$ and $S \subseteq X$, we say that $S$ is $u$-small if $S \times S \subseteq u$. We define $\mathcal{B}(S,u) = \{y \in X: \exists s \in s, (s,y) \in u\}$. Note that, if $u$ is symmetric (i.e., $u = u^{-1}$) and $S = \{x\}$, then $\mathcal{B}(x,u)$ is $u^{2}$-small. We have that $\mathcal{B}(S,u)$ is a neighborhood of $S$, but, in general, it is not an open subset of $X$. However, if $u$ is open in $X \times X$ (in the product topology with respect to the associated topology of $X$), then $\mathcal{B}(S,u)$ is open.

\subsection{Orders of growth and generalized entropy}
Let us start by recalling how the complete set of the orders of growth and the generalized entropy of a map are defined in \cite{CoPu21}. Consider the space of non-decreasing sequences in $[0,\infty)$: $$\mathcal{O}=\{a:\mathbb{N}\rightarrow [0,\infty):a(n)\leqslant a(n+1),\, \forall n\in \mathbb{N}\}.$$
Then, define the equivalence relationship $\approx$ in $\mathcal{O}$ by $a(n)\approx b(n)$ if and only if there exist $c_1,c_2\in (0,\infty)$ such that $c_1 a(n)\leqslant b(n)\leqslant c_2 a(n)$ for all $n\in \mathbb{N}$. The quotient space $\displaystyle \mathbb{O}=\mathcal{O}/_{\approx}$ is called the space of the orders of growth. 

In this space, some of the orders of growth can be compared. We say that $[a(n)] \leqslant [b(n)]$ if there exists $C>0$ such that $a(n) \leqslant Cb(n)$, for all $n\in\mathbb{N}$ (note that this order relation is well defined). Then, we  consider $\overline{\mathbb{O}}$ as the Dedekind-MacNeille completion of $\mathbb{O}$. This object is the smallest complete lattice that contains $\mathbb{O}$; it is uniquely defined, and the space $\overline{\mathbb{O}}$ is named the complete set of the orders of growth. 

Now, let us present how the generalized entropy of a dynamical system is defined in the complete space of the orders of growth. Given $(M,d)$, a compact metric space, and $f: M\rightarrow M$, a continuous map. We consider in $M$ the distance 
\[d^{f}_{n}(x,y)=\max \{d(f^k(x),f^k(y)): 0\leqslant k \leqslant n-1\},\]
 and we denote the dynamical $(n,\epsilon)$-ball centered at the point $x \in M$ as $\mathcal{B}(x,n,\varepsilon)=\{y\in M: d^{f}_{n}(x,y)<\varepsilon\}$. A set $G\subseteq M$ is a $(n,\varepsilon)$-generator if $\displaystyle M=\cup_{x\in G} \mathcal{B}(x,n,\varepsilon)$. Then, we define $g_{f,\varepsilon}(n)$ as the smallest possible cardinality of a finite $(n,\varepsilon)$-generator. If we fix $\varepsilon>0$, then $g_{f,\varepsilon}(n)$ is a non-decreasing sequence of natural numbers, so $g_{f,\varepsilon}(n) \in \mathcal{O}$. For a fixed $n \in \N$, if $\varepsilon_1<\varepsilon_2$, then $g_{f,\varepsilon_1} (n) \geqslant g_{f, \varepsilon_2}(n)$, and $[g_{f,\varepsilon_1}(n)]\geqslant [g_{f,\varepsilon_2}(n)]$ in $\mathbb{O}$. Consider the set $G_f=\{[g_{f,\varepsilon}(n)]\in \mathbb{O}:\varepsilon>0\}$, and the generalized entropy of $f$ is defined as 
$$o(f)=\sup G_f \in \overline{\mathbb{O}}. $$

We can see the generalized entropy as a limit of the orders of growth $[g_{f,\varepsilon}(n)]$ when $\varepsilon\rightarrow 0$ (in the Scott topology of $\overline{\mathbb{O}}$). This object is a dynamical invariant:
\begin{prop}[Correa-Pujals, Theorem 1 of \cite{CoPu21}]
	Let $M$ and $N$ be two compact metric spaces and $f:M\to M$, $g:N\to N$, two continuous maps. Suppose there exists $h: M\to N$, a homeomorphism, such that $h\circ f = g \circ h$. Then, $o(f)=o(g)$.
\end{prop}

The generalized entropy can also be defined through the point of view of $(n,\varepsilon)$-separated sets. We say that a set $E\subseteq M$ is $(n,\varepsilon)$-separated if $\mathcal{B}(x,n,\varepsilon)\cap E = \{x\}$, for all $x\in E$. Then we define $s_{f,\varepsilon}(n)$ as the maximal cardinality of a $(n,\varepsilon)$-separated set. Analogously, as with $g_{f,\varepsilon}(n)$, if we fix $\varepsilon>0$, then $s_{f,\varepsilon}(n)$ is a non-decreasing sequence of natural numbers. Again, for a fixed $n \in \N$, if $\varepsilon_1<\varepsilon_2$, then $s_{f,\varepsilon_1} (n) \geqslant s_{f, \varepsilon_2}(n)$, and therefore, $[s_{f,\varepsilon_1}(n)]\geqslant [s_{f,\varepsilon_2}(n)]$. If we consider the set $S_f=\{[s_{f,\varepsilon}(n)]\in \mathbb{O}:\varepsilon>0\}$, then 
$$o(f)=\sup S_f \in \overline{\mathbb{O}}. $$

If $M$ is not compact, then the generalized entropy of a map can be defined in compact subsets. Given $K\subset M$, a compact subset, the definition of $g_{f,\varepsilon, K}(n)$ as the minimal number of $(n,\varepsilon)$-balls (centered at points in $K$) that are needed to cover $K$ also makes sense. With it, we can define $o(f,K)=\sup\{[g_{f,\varepsilon,K}(n)]:\varepsilon>0\}$ as the generalized entropy of $f$ in $K$ and $o(f) = \sup \{o(f,K): K \text{ is a compact subset of } M\}$ as the generalized entropy of $f$.   

Now, let us explain how generalized entropy is related to the classical notion of topological entropy and polynomial entropy. Given a dynamical system $f$, recall that the topological entropy of $f$ is 
\[h(f) = \lim_{\varepsilon\to 0}\limsup_{n\rightarrow \infty} \frac{log(g_{f,\varepsilon}(n))}{n}.\]

And the polynomial entropy, introduced in \cite{Mar}, of $f$ is
\[h_{pol}(f) = \lim_{\varepsilon\to 0}\limsup_{n\rightarrow \infty} \frac{log(g_{f,\varepsilon}(n))}{log(n)}.\]

The family of exponential orders of growth is the set $\mathbb{E}=\{[\exp(tn)]: t \in(0,\infty) \} \subset \mathbb{O}$ and the family of polynomial orders is the set $\mathbb{P}=\{[n^t]: t \in(0,\infty)\}$. Given $o\in \overline{\mathbb{O}}$, in \cite{CoPu21}, the authors define the intervals $I(o,\mathbb{E})=\{t\in (0,\infty):o\leqslant [\exp(tn)]\}$ and $I(o,\mathbb{P})=\{t\in (0,\infty):o\leqslant [n^t]\}$. With these intervals, the projections  $\pi_{\mathbb{E}}:\overline{\mathbb{O}}\rightarrow [0,\infty]$ and $\pi_{\mathbb{P}}:\overline{\mathbb{O}}\rightarrow [0,\infty]$ are defined as

$$
\pi_{\mathbb{E}}(o)=
\begin{cases} \inf(I(o,\mathbb{E})),&\text{ if }I(o,\mathbb{E})\neq \emptyset \\
\infty, &\text{ if } I(o,\mathbb{E})= \emptyset
\end{cases},
$$ 
and 
$$
\pi_{\mathbb{P}}(o)=
\begin{cases} \inf(I(o,\mathbb{P})),&\text{ if }I(o,\mathbb{P})\neq \emptyset \\
\infty, &\text{ if } I(o,\mathbb{P})= \emptyset
\end{cases}.
$$ 

Generalized entropy, polynomial entropy and classical entropy are related by the following result. 

\begin{prop}\label{projectionmaps}[Correa-Pujals, Theorem 2 of \cite{CoPu21}]
	Let $M$ be a compact metric space and $f: M \rightarrow M$ a continuous map. Then, $\pi_\mathbb{E}(o(f))=h(f)$ and $\pi_\mathbb{P}(o(f))=h_{pol}(f)$.
\end{prop}

\section{Entropy for uniform spaces}\label{Section Entropy for uniform spaces}

In this section, we define generalized entropy for uniform spaces. It generalizes the definition for metric spaces. In \cite{Ho}, Hood defined topological entropy for uniform spaces, which the following definition generalizes in the same way as it is done for metric spaces. Here, we also present some properties that are useful in this paper, with their respective proofs, for the sake of completeness, even when these proofs are analogous to their respective versions of generalized entropy for metric spaces or even when it's already proved for uniform spaces. We conclude this section by presenting the concept of coding and mutually singular sets, which is a translation of the construction presented by Hauseux and Le Roux, in \cite{HR}.

\subsection{Entropy}

Here, we extend the definition of generalized entropy to the setting of uniform spaces, following the same principles used in the metric case. We introduce both the generator-set and separated-set formulations, and we prove that they are equivalent. This extension allows entropy computations in non-metrizable contexts and provides a broader framework for analyzing dynamical complexity in uniform spaces.

\begin{defi}Let $(X, \mathcal{U})$ be a uniform space and $f: X \rightarrow X$ a continuous map. Let $n \in \N$ and $u \in \mathcal{U}$. A subset $E \subseteq X$ is $(n,u)$-separated if for every $x,y \in E$, with $x \neq y$, there exists $i \in \{0,\dots,n-1\}$ such that $(f^{i}(x),f^{i}(y)) \notin u$.

Let $K$ be a compact subspace of $X$, $n \in \N$ and $u \in \mathcal{U}$. We define $s_{f,u,K}(n) = \max\{\# E: E \subseteq K, E$ is $(n,u)-separated\}$ (we denote $s_{f,u,K}(n)$ by $s_{f,u}(n)$, if $X$ is compact and $K = X$). Since $K$ is compact, it is clear that $s_{f,u,K}(n)$ must be finite.  Observe that $s_{f,u,K}(n)\in \mathcal{O}$, so we can define the generalized entropy of $f$ with respect to $K$ as:

$$o(f,K) = \sup \{ [(s_{f,u,K}(n)]: u \in \mathcal{U}\} \in \overline{\mathbb{O}}$$

Finally, we define the generalized entropy of $f$ as:

$$o(f) = \sup \{o(f,K): K \subseteq X, \ K \ is \ compact\}\in \overline{\mathbb{O}}$$ 
\end{defi}

Similar to the metric case, we can also define the entropy through the point of view of $(n,u)$-generator sets. We denote the dynamical $(n,u)$-neighborhood centered at the point $x \in X$ as $$\mathcal{B}(x,n,u)=\{y\in X: (f^i(x),f^i(y))\in u, \text{ for all } 0\leqslant i \leqslant n-1\}.$$

We use the notation $\mathcal{B}_{f}(x,n,u)$ for $\mathcal{B}(x,n,u)$ if it is necessary to clarify which map we are using. Note that we can also describe the dynamical neighborhood as $$\mathcal{B}(x,n,u)= \bigcap_{i = 0}^{n-1} f^{-i}(\mathcal{B}(f^{i}(x),u)).$$

Then $\mathcal{B}(x,n,u)$ is a neighborhood of $x$ since it is a finite intersection of neighborhoods of $x$. By the same reason, if $u$ is open in $X\times X$, then $\mathcal{B}(x,n,u)$ is open.

Let $K\subseteq X$ be a compact subset. A set $G\subseteq X$ is a $(n,u)$-generator set for $K$ if $\displaystyle K\subseteq\cup_{x\in G} \mathcal{B}(x,n,u)$. Then, we define $g_{f,u, K}(n)$ as the smallest possible cardinality of a finite $(n,u)$-generator set for $K$. If we fix $u \in \mathcal{U}$, then $g_{f,u, K}(n)$ is a non-decreasing sequence of natural numbers, so $g_{f,u,K}(n) \in \mathcal{O}$ and the generalized entropy of $f$ with respect to $K$ is  $o_g(f,K)=\sup\{[g_{f,u,K}(n)]\in \mathbb{O}:u\in \mathcal{U}\}\in \overline{\mathbb{O}} $, and the generalized entropy of $f$ is also defined as $o_g(f) = \sup \{o_g(f,K): K \subseteq X, \ K \ is \ compact\}\in \overline{\mathbb{O}}$. 

As in the metric case, the following lemma allows us to compute entropy using either $(n,u)$-separated sets or $(n,u)$-generator sets.

\begin{lema}[Hood, Lemma 1.i of \cite{Ho}]\label{geradoreseseparaveis}Let $(X,\mathcal{U})$ be a uniform space, $f:X\to X$ a continuous map and $K\subseteq X$ a compact. Then, $s_{f,u^2,K}(n) \leqslant g_{f,u,K}(n)\leqslant s_{f,u,K}(n)$, for every $u\in\mathcal{U}$ symmetric.   
\end{lema}

\begin{proof} Consider $E \subseteq K$ a maximal $(n,u)$-separated set. So, given any $y\in K-E$, we have that $E\cup \{y\}$ is not a $(n,u)$-separated set. Then, there exists $x \in E$ such that $(f^i(x), f^i (y))\in u$, for all and $0\leqslant i\leqslant n - 1$, which implies that $E$ is a $(n,u)$-generator set for $K$. Then $g_{f,u,k}(n)\leqslant \#E=s_{f,u,K}(n)$.

For the other inequality, let $E$ be a $(n,u^2)$-separated set for $K$ and $G$ a $(n,u)$-generator set for $K$. For a given $x\in E$, there exists a $y\in G$ such that $(f^i(x),f^i(y))\in u$, for every $0\leqslant i \leqslant n-1$. Define a map $\phi:E\to G$ taking $\phi(x)$ as any $y\in G$ in these conditions. The map $\phi$ is injective. Indeed, suppose that $\phi(x_1)=y=\phi(x_2)$, for $x_1, x_2 \in E$. Then $(f^i(x_1), f^i(y))\in u$,  for every $0\leqslant i \leqslant n-1$ and $(f^i(x_2), f^i(y))\in u$,  for every $0\leqslant i \leqslant n-1$. So $(f^i(x_1),f^i(x_2))\in u^2$, for every $0\leqslant i \leqslant n-1$, since $u$ is symmetric. But $E$ is a $(n,u^2)$-separated, which implies $x_1=x_2$. It follows that $\#E\leqslant \#G$, and since we take $E$ and $G$ arbitrarily, we conclude $s_{f,u^2,K}(n) \leqslant g_{f,u,K}(n)$.
\end{proof}

\begin{lema}[Hood, Lemma 1.ii of \cite{Ho}] Let $(X, \mathcal{U})$ be a uniform space, $\mathcal{B}$ a base for the uniform structure $\mathcal{U}$ and $f: X \rightarrow X$ a continuous map. If $K \subseteq X$, $n \in \N$ and $u,v \in \mathcal{U}$, with $v \subseteq u$, then $s_{f,v,K}(n) \geqslant s_{f,u,K}(n)$.
\end{lema}

\begin{proof}Let $E \subseteq K$ be a $(n,u)$-separated set. Then, for every points $x,y \in E$, with $x\neq y$, there exists $i \in \{0,\dots,n-1\}$ such that $(f^{i}(x),f^{i}(y)) \notin u$. Since $v \subseteq u$, then $(f^{i}(x),f^{i}(y)) \notin v$. So $E$ is a $(n,v)$-separated set. Thus $s_{f,v,K}(n) \geqslant s_{f,u,K}(n)$.
\end{proof}

The definition of topological entropy for uniform spaces given in \cite{Ho} does not take into account the separated sets for all entourages of the uniform structure. Instead of this, they consider only specific entourages that are easier to deal with: open and symmetric. We could choose to do the same thing here for generalized entropy, but the next proposition shows that this choice does not matter, since we would get the same definition.

\begin{prop}\label{basenaomudaentropia}Let $(X, \mathcal{U})$ be a uniform space, $\mathcal{B}$ a base for the uniform structure $\mathcal{U}$ and $f: X \rightarrow X$ a continuous map. If $K \subseteq X$, then $o(f,K) = \sup\limits_{u \in \mathcal{B}}  [s_{f,u,K}(n)]$.
\end{prop}

\begin{proof}Let $u \in \mathcal{U}$. Since $\mathcal{B}$ is a base for the uniform structure $\mathcal{U}$, then there exists $v \in \mathcal{B}$ such that $v \subseteq u$. Then, for every $n \in \N$, $s_{f,v,K}(n) \geqslant s_{f,u,K}(n)$, which implies that $[s_{f,v,K}(n)] \geqslant [s_{f,u,K}(n)]$. Thus $o(f,K) =  \sup\limits_{u \in \mathcal{U}} [s_{f,u,K}(n,u)] = \sup\limits_{u \in \mathcal{B}} [s_{f,u,K}(n)]$.
\end{proof}

\begin{prop}Let $(X, \mathcal{U})$ be a uniform space and $f: X \rightarrow X$ a homeomorphism. Then $o(f) = o_{g}(f)$.
\end{prop}

\begin{proof}Immediate from \textbf{Lemma \ref{geradoreseseparaveis}} an \textbf{Proposition \ref{basenaomudaentropia}}.
\end{proof}

From now on, we call $o(f)$ and $o_{g}(f)$ as $o(f)$, indistinguishably. 

The generalized entropy is, also in the uniform context, an invariant:

\begin{prop}Let $(X,\mathcal{U})$, $(Y, \mathcal{V})$ be uniform spaces and $f: X \rightarrow X$ and $g: Y \rightarrow Y$ be continuous maps. If $f$ and $g$ are conjugated by a uniform homeomorphism $\rho: X \rightarrow Y$, then for every compact subset $K$ of $X$, $o(f,K) = o(g, \rho(K))$. Moreover, $o(f) = o(g)$.
\end{prop}

\begin{proof}Let $K$ be a compact subset of $X$, $u \in \mathcal{U}$ and $n \in \N$. If $E \subseteq K$ is a $(n,u)$-separated set, then there exists $i \in \{0,\dots,n-1\}$ such that $(f^{i}(x),f^{i}(y)) \notin u$. Then $(\rho \circ f^{i}(x), \rho \circ f^{i}(y)) \notin \rho(u)$. Note that $\rho(u) \in  \mathcal{V}$ since $\rho$ is a uniform homeomorphism. We have that $\rho \circ f^{i} = g^{i} \circ \rho$, which implies that $(g^{i} \circ \rho(x), g^{i} \circ \rho(y)) \notin \rho(u)$ and then $\rho(E)$ is a $(n,\rho(u))$-separated set. Analogously, if $\rho(E)$ is a $(n,\rho(u))$-separated set, then we apply the map $\rho^{-1}$ and we get that $E$ is a $(n,u)$-separated set. So $s_{f,u,K}(n) = s_{g,\rho(u),\rho(K)}(n)$.

Then $o(f,K) = \sup \{ [(s_{f,u,K}(n)]: u \in \mathcal{U}\} = \sup \{[s_{g,\rho(u),\rho(K)}(n)]: u \in \mathcal{U}\}$. 
Since $\rho$ is an uniform homeomorphism, then $\mathcal{V} = \{\rho(u): u \in \mathcal{U}\}$, which implies that $$\sup\{[s_{g,\rho(u),\rho(K)}(n)]: u \in \mathcal{U}\} = \sup\{[s_{g,v,\rho(K)}(n)]: v \in \mathcal{V}\} = o(g,\rho(K)).$$ 

So $o(f,K) = o(g, \rho(K))$. Since the map $\rho$ is a homeomorphism, then $\{\rho(K): K \subseteq X \text{ and } K \text{ is compact}\} = \{K' \subseteq Y: K' $ is  compact$\}$. Then 

$\begin{array}{rcl}
     o(f) &=& \sup \{o(f,K): K \subseteq X, \ K\text{ is  compact}\}\\
     & =&\sup \{o(f,\rho(K)): K \subseteq X, \ K\text{ is  compact}\} \\
     &=&\sup \{o(f,K'): K' \subseteq Y, \ K'\text{ is compact}\}\\
     &=&o(g).
\end{array}$  
\end{proof}

If the space is metrizable, then the generalized entropy for the uniform structure induced by the metric coincide with the generalized structure for the metric:

\begin{prop}Let $(X,d)$ be a metric space, $\mathcal{U}_{d}$ the uniform structure induced by the metric $d$, $f: X \rightarrow X$ a continuous map, $o_{d}(f,K)$, $o_{d}(f)$ the entropies of $f$ with respect to the metric $d$ and $o_{u}(f,K)$, $o_{u}(f)$ the entropies of $f$ with respect to the uniform structure $\mathcal{U}_{d}$. Then, for every compact set $K$, $o_{d}(f,K) = o_{u}(f,K)$. Moreover, $o_{d}(f) = o_{u}(f)$.
\end{prop}

\begin{proof}Let $K$ be a compact subset of $X$, $E \subseteq K$, $n \in \N$ and $\epsilon > 0$. Then $E$ is $(n, u_{\epsilon})$-separated if and only if for every $x,y \in E$, with $x \neq y$, there exists $i \in \{0,\dots,n-1\}$ such that $(f^{i}(x),f^{i}(y)) \notin u_{\epsilon}$. But $(f^{i}(x),f^{i}(y)) \notin u_{\epsilon}$ means that $d(f^{i}(x),f^{i}(y)) > \epsilon$. So  $E$ is $(n, u_{\epsilon})$-separated if and only if it is $(n, \epsilon)$-separated. So $s_{f,u_{\epsilon},K}(n) = s_{f,\epsilon,K}(n)$. Then we get $o_{d}(f,K) =  \sup\limits_{\epsilon > 0} [s_{f,\epsilon,K}(n)] = \sup\limits_{\epsilon > 0} [s_{f,u_{\epsilon},K}(n)]$. Since $\{u_{\epsilon}: \epsilon > 0\}$ is a base for the uniform structure $\mathcal{U}_{d}$, then $\sup\limits_{\epsilon > 0} [s_{f,u_{\epsilon},K}(n)] = o_{u}(f,K)$. So $o_{d}(f,K) = o_{u}(f,K)$ and then $o_{d}(f) = o_{u}(f)$.
\end{proof}

\begin{cor}If two metrics induce the same uniform structure, then the entropy of a map with respect to each of them must coincide.
\end{cor}

\begin{proof}Immediate.
\end{proof}

Similarly as for the metric case, we can compare the generalized entropy for uniform spaces with the topological entropy and the polynomial entropy. Let us give a brief definition of polynomial entropy for uniform spaces. Let $K$ be a compact subspace of $X$, $n \in \N$ and $u \in \mathcal{U}$. We consider $s_{f,u,K}(n)$ the maximal cardinality of a $(n,u)$-separated set as before, and define

$$h(f,K)=\lim_{u\in\mathcal{U}}\limsup_{n\rightarrow \infty}\frac{\log s_{f,u,K}(n) }{n}, \ \ \ \ \ \ \ \ h_{pol}(f,K)=\lim_{u\in\mathcal{U}}\limsup_{n\rightarrow \infty}\frac{\log s_{f,u,K}(n) }{\log n} $$

and the topological entropy and polynomial entropy of the map $f$ are defined, respectively, as $h(f)=\sup\{h(f,K): K\subseteq X, K \text{ is compact}\}$ and $h_{pol}(f)=\sup\{h_{pol}(f,K): K\subseteq X, K \text{ is compact}\}$. As mentioned before, the definition of topological entropy is different from that one given in \cite{Ho}, but, by an argument entirely analogous to \textbf{Proposition \ref{basenaomudaentropia}}, both definitions must coincide. 

\begin{lema}Let $S \subseteq \overline{\mathbb{O}}$. Then $\pi_{\mathbb{E}}(\sup S) = \sup \pi_{\mathbb{E}}(S)$ and $\pi_{\mathbb{P}}(\sup S) = \sup \pi_{\mathbb{P}}(S)$.
\end{lema}

\begin{proof}Assume that $\sup \pi_{\mathbb{P}}(S) < \pi_{\mathbb{P}}(\sup S)$. Let $t > 0$, that satisfies $\sup \pi_{\mathbb{P}}(S) < t < \pi_{\mathbb{P}}(\sup S)$. Since  $t < \pi_{\mathbb{P}}(\sup S)$, then $t \notin I(\sup S, \mathbb{P})$, which implies that $\sup S \nleqslant [n^{t}]$. Since $\sup \pi_{\mathbb{P}}(S) < t$, then, for every $s \in S$, we have $\pi_{\mathbb{P}}(s) < t$, which implies that, for every $s \in S$, $s \leqslant [n^{t}]$ and then  $\sup S \leqslant [n^{t}]$, which is a contradiction. Thus $\sup \pi_{\mathbb{P}}(S) \geqslant \pi_{\mathbb{P}}(\sup S)$.

Suppose that $I(\sup S, \mathbb{P}) = \emptyset$. Then $\pi_{\mathbb{P}}(\sup S) = \infty$ and $\pi_{\mathbb{P}}(\sup S) \leqslant \sup \pi_{\mathbb{P}}(S)$, which implies that  $\pi_{\mathbb{P}}(\sup S) = \sup \pi_{\mathbb{P}}(S) = \infty$.

Suppose now that $I((\sup S, \mathbb{P}) \neq \emptyset$ and let $t \in I((\sup S, \mathbb{P})$.
So $\sup S \leqslant [n^{t}]$, which implies that, for every $s \in S$, it holds $s \leqslant [n^{t}]$ and then $t \in I(s, \mathbb{P})$. So $I((\sup S, \mathbb{P}) \subseteq I(s,\mathbb{P})$. Thus, for every $s \in S$, we have $\pi_{\mathbb{P}}(s) = \inf I(s,\mathbb{P}) \leqslant \inf I(\sup S, \mathbb{P}) = \pi_{\mathbb{P}}(\sup S)$ and then $\sup\pi_{\mathbb{P}}(S) \leqslant \pi_{\mathbb{P}}(\sup S)$.

Thus $\pi_{\mathbb{P}}(\sup S) = \sup \pi_{\mathbb{P}}(S)$.

For the projection map with respect to the exponential family, the argument is analogous.
\end{proof}

The next lemma is stated in \cite{CP}. We give a proof of it here, for the sake of completeness.

\begin{lema}(Correa-de Paula, page 5 of \cite{CP}) Let $[a(n)] \in \mathbb{O}$. Then $\pi_\mathbb{E}([a(n)]) = \limsup\limits_{n\rightarrow \infty} \frac{\log (a(n))}{n}$ and $\pi_\mathbb{P}([a(n)]) = \limsup\limits_{n\rightarrow \infty} \frac{\log (a(n))}{\log(n)}$.
\end{lema}

\begin{proof}Suppose that $\pi_{\mathbb{P}}([a(n)]) < \infty$. Then $I([a(n)],\mathbb{P}) \neq \emptyset$. Let $t \in I([a(n)],\mathbb{P})$. Then $[a(n)] \leqslant [n^{t}]$. So there exists $c > 0$ such that for every $n \in \N$, $a(n) \leqslant c n^{t}$, which implies that $\log(a(n)) \leqslant c'+t\log(n)$, where $c' = \log c$. So $\frac{\log(a(n))}{\log(n)}-\frac{c'}{\log(n)} \leqslant t$ and then $\limsup\limits_{n\rightarrow\infty}\frac{\log(a(n))}{\log(n)} \leqslant t$. So $\pi_\mathbb{P}([a(n)]) = \inf I([a(n)],\mathbb{P}) \geqslant \limsup\limits_{n\rightarrow \infty} \frac{\log (a(n))}{\log(n)}$.

Suppose that $\limsup\limits_{n\rightarrow \infty} \frac{\log (a(n))}{\log(n)} = t < \infty$. So, for every $\delta > 0$, there exists $N_{\delta} > 0$ such that for every $n \geqslant N_{\delta}$, $\frac{\log (a(n))}{\log(n)} \leqslant t+\delta$ (this is a limit superior property, see Exercise 4 of Chapter 11, $\S$ 9 of \cite{Lan}, for example). So, for every $\delta > 0$ and $n \geqslant N_{\delta}$, $\log(a(n)) \leqslant \log(n^{t+\delta})$ and then, for every $\delta > 0$ and $n \geqslant N_{\delta}$, $a(n) \leqslant n^{t+\delta}$. This gives a relation between the classes of growth functions: for every $\delta > 0$, $[a(n)] \leqslant [n^{t+\delta}]$. This implies that, for every $\delta > 0$, $t+\delta \in I([a(n)],\mathbb{P})$ and then, for every $\delta > 0$, $\inf I([a(n)],\mathbb{P}) \leqslant t+\delta$. So $\inf I([a(n)],\mathbb{P}) \leqslant t$, which implies that $\pi_{\mathbb{P}}([a(n)]) = \inf I([a(n)],\mathbb{P}) \leqslant t = \limsup\limits_{n\rightarrow \infty} \frac{\log (a(n))}{\log(n)}$.

So $\pi_{\mathbb{P}}([a(n)]) < \infty$ if and only if $\limsup\limits_{n\rightarrow \infty} \frac{\log (a(n))}{\log(n)} < \infty$ and in this case $\pi_\mathbb{P}([a(n)]) = \limsup\limits_{n\rightarrow \infty} \frac{\log (a(n))}{\log(n)}$. This implies that $\pi_{\mathbb{P}}([a(n)]) = \infty$ if and only if $\limsup\limits_{n\rightarrow \infty} \frac{\log (a(n))}{\log(n)} = \infty$ and it follows that $\pi_\mathbb{P}([a(n)]) = \limsup\limits_{n\rightarrow \infty} \frac{\log (a(n))}{\log(n)}$.

For the projection map with respect to the exponential family, the argument is analogous.
\end{proof}

Now, we can relate generalized entropy for uniform spaces to topological entropy and polynomial entropy for uniform spaces, generalizing \textbf{Proposition \ref{projectionmaps}}:

\begin{prop}\label{projectionmaps2} (Correa-Pujals, Theorem 2 of \cite{CoPu21}, for metric spaces) Let $(X,\mathcal{U})$ be a uniform space and $f: X \rightarrow X$ a continuous map. If $K$ is a compact subspace of $X$, then $\pi_\mathbb{E}(o(f,K))=h(f,K)$ and $\pi_\mathbb{P}(o(f,K))=h_{pol}(f,K)$. Moreover, $\pi_\mathbb{E}(o(f))=h(f)$ and $\pi_\mathbb{P}(o(f))=h_{pol}(f)$.
\end{prop}

\begin{proof}Let $K$ be a compact subspace of $X$. Then $\pi_{\mathbb{P}}(o(f,K)) = \pi_{\mathbb{P}}(\sup\limits_{u \in \mathcal{U}}[s_{f,u,K}(n)]) = \sup\limits_{u \in \mathcal{U}}\pi_{\mathbb{P}}([s_{f,u,K}(n)]) = \sup\limits_{u \in \mathcal{U}} \limsup\limits_{n\rightarrow \infty} \frac{\log (s_{f,u,K}(n))}{\log(n)} = h_{pol}(f,K)$. 

Also, $\pi_{\mathbb{P}}(o(f)) = \pi_{\mathbb{P}}(\sup\{o(f,K): K \subseteq X, \ K \ \text{is compact} \}) = \sup\{\pi_{\mathbb{P}}(o(f,K)): K \subseteq X, \ K \ \text{is compact} \} = \sup\{ h_{pol}(f,K):K \subseteq X, \ K \ \text{is compact} \} = h_{pol}(f)$.

For the topological entropy, the argument is analogous.
\end{proof}

\subsection{Properties}

Here, we establish a collection of basic properties of generalized entropy in the context of uniform spaces. Many of these results are natural extensions of their counterparts in the metric setting, while others highlight structural features specific to uniform spaces. These properties will be used throughout the paper and include the behavior under restrictions to invariant subspaces, monotonicity under iteration, and vanishing of entropy under Lyapunov stability, among others.

\begin{prop}\label{entropyandsubspace}Let $(X,\mathcal{U})$ be a uniform space, $f: X \rightarrow X$ a continuous map, $(Y, \mathcal{U}')$ an invariant subset of $X$, with its subspace uniform structure, and $K$ a compact set contained in $Y$. Then $$o(f,K) = o(f|_{Y},K),$$ where we consider on the left hand side the uniform structure $\mathcal{U}$ and on the right hand side we consider the uniform structure $\mathcal{U}'$.
\end{prop}

\begin{proof}Let $u' \in \mathcal{U}'$. Then, there exists $u \in \mathcal{U}$, such that $u' = u \cap (Y\times Y)$. Let $S \subseteq K$. So $S$ is $(n,u')$-separated for the map $f|_{Y}$ if and only if $S$ is $(n,u)$-separated for the map $f$. Then $s_{f|_{Y},u',K}(n) = s_{f,u,K}(n)$. If we start with $u \in \mathcal{U}$, we define  $u' = u \cap (Y\times Y)$ and we also get that $s_{f|_{Y},u',K}(n) = s_{f,u,K}(n)$. Thus  $o(f|_{Y},K) = \sup\{[s_{f|_{Y},u',K}(n)]: u' \in \mathcal{U}'\} =  \sup\{[s_{f,u,K}(n)]: u \in \mathcal{U}\} = o(f,K)$.  
\end{proof}

\begin{prop}[Correa-Pujals, Proposition B.3 of \cite{CoPu21} for  metric spaces]\label{subspaces} If $(X, \mathcal{U})$ is a Hausdorff uniform  space, $f: X \rightarrow X$ is a continuous map, $K_{1},\dots,K_{m}$ compact subspaces of $X$ and $K = \bigcup_{i = 1}^{m} K_{i}$. Then $o(f,K) = \sup\{o(f, K_{i}): i \in \{1,\dots,m\}\}$.
\end{prop}

\begin{proof}Since $K_{i} \subseteq K$, then, for every $u \in \mathcal{U}$ and $n \in \N$, if a subset in $K_{i}$ is $(n,u)$-separated, then it is $(n,u)$-separated as a subset of $K$. So $s_{f,u,K_{i}}(n) \leqslant s_{f,u,K}(n)$, which implies that $o(f,K_{i}) \leqslant o(f,K)$ and then $\sup\{o(f, K_{i}): i \in \{1,\dots,m\}\} \leqslant o(f,K)$.

Let $u \in \mathcal{U}$ and $n \in \N$. Let $S \subseteq K$ that is $(n,u)$-separated and $\# S = s_{f,u,K}(n)$. Then $S\cap K_{i}$ is  $(n,u)$-separated and it is contained in $K_{i}$, which implies that $\# (S\cap K_{i}) \leqslant s_{f,u,K_{i}}(n)$. Then $s_{f,u,K}(n) = \# S \leqslant \sum_{i = 1}^{m} \# (S\cap K_{i}) \leqslant \sum_{i = 1}^{m} s_{f,u,K_{i}}(n) \leqslant m \sup \{s_{f,u,K_{i}}(n): i \in \{1,\dots,m\}\}$, which implies that $[s_{f,u,K}(n)] \leqslant [\sup \{s_{f,u,K_{i}}(n): i \in \{1,\dots,m\}\}]$. But, by Lemma 3.3 of \cite{CP}, $[\sup \{s_{f,u,K_{i}}(n): i \in \{1,\dots,m\}\}] = \sup \{[s_{f,u,K_{i}}(n)]: i \in \{1,\dots,m\}\}$, which implies that $[s_{f,u,K}(n)] \leqslant  \sup \{[s_{f,u,K_{i}}(n)]: i \in \{1,\dots,m\}\}$. So $o(f,K) = \sup\{[s_{f,u,K}(n)]: u \in \mathcal{U}\} \leqslant  \sup \{[s_{f,u,K_{i}}(n)]: i \in \{1,\dots,m\}, u \in \mathcal{U}\} = \sup\{o(f, K_{i}): i \in \{1,\dots,m\}\}$.

Thus $o(f,K) = \sup\{o(f, K_{i}): i \in \{1,\dots,m\}\}$.
\end{proof}

\begin{prop}[Correa-Pujals, Theorem 3 of \cite{CoPu21} for compact metric spaces]\label{alphalimitentropy} Let $(X,\mathcal{U})$ be a Hausdorff compact uniform space and $f: X \rightarrow X$ a homeomorphism. If there exists $x \in X$ that is not contained on its own $\alpha$-limit, then $o(f) \geqslant [n]$.  
\end{prop}

\begin{proof}Let $x \in X$ be a point that is not contained in its own $\alpha$-limit and $n \in \N$. Then, there exists $u \in \mathcal{U}$ and $m_{0} \in \N$ such that for every $m > m_{0}$, $f^{-m}(x) \notin \mathcal{B}(x,u)$. Let $v \subseteq u$ be a symmetric entourage such that $f^{-m_{0}}(x),\dots,f^{-2}(x),f^{-1}(x) \notin \mathcal{B}(x,v)$. Then, for every $m > 1$, $f^{-m}(x) \notin \mathcal{B}(x,v)$. Consider the set $A_{n} = \{f^{-i}(x): i \in \{0,..,n-1\}\}$. Let $i,j \in \{0,\dots,n-1\}$, with $i < j$. Hence, $i \in \{0,\dots,n-1\}$ and $(f^{i}(f^{-i}(x)),f^{i}(f^{-j}(x))) = (x,f^{i-j}(x)) \notin v$, which implies that $f^{-i}(x)$ and $f^{-j}(x)$ are $(n,v)$-separated and then $A_{n}$ is $(n,v)$-separated. So $s_{f,v}(n) \geqslant \# A_{n} = n$, which implies that $[s_{f,v}(n)] \geqslant [n]$ and then $o(f) \geqslant [s_{f,v}(n)] \geqslant [n]$.
\end{proof}

\begin{prop}[Correa-Pujals, Proposition B.1 of \cite{CoPu21} for compact metric spaces]\label{periodichaszeroentropy} If $(X, \mathcal{U})$ is a Hausdorff uniform space and $f: X \rightarrow X$ is a homeomorphism, then $o(f) \leqslant o(f^2) \leqslant o(f^3) \leqslant\dots$. Moreover, if $f$ is a periodic map, then $o(f) = 0$. 
\end{prop}

\begin{obs}We denote by $0$ as the class of eventually constant sequences in $\overline{\mathbb{O}}$.
\end{obs}

\begin{proof}Let $K$ be a compact subset of $X$ and $u \in \mathcal{U}$. If a set $S \subseteq K$ is a $(kn,u)$-generator for $K$, with respect to $f$, then, for every $y \in K$, there exists $x \in S$ such that for every $i \in \{0,..,kn-1\}$, $(f^{i}(x),f^{i}(y)) \in u$. In particular, for every $i \in \{0,\dots,n-1\}$, $(f^{ik}(x),f^{ik}(y)) \in u$, which implies that $S$ is a $(n,u)$-generator for $K$, with respect to $f^{k}$. Thus $g_{f^{k},u,K}(n) \leqslant g_{f,u,K}(kn)$.

Let $v_{u} \in \mathcal{U}$ such that for every $x,y \in \bigcup\limits_{i = 0}^{n-1} f^{ik}(K)$, with $(x,y) \in v_{u}$, we have that for every $i \in \{0,\dots,k-1\}$, $(f^{i}(x),f^{i}(y)) \in u$ (it is possible to assure the existence of $v_{u}$ since $K$ is compact and $f|_{f^{ik}(K)}$ is uniformly continuous, for each $i \in \{0,\dots,n+k-2\}$). Then, for every $x \in \bigcup\limits_{i = 0}^{n-1} f^{ik}(K)$, $\mathcal{B}(x,v_{u}) \subseteq \mathcal{B}_{f}(x,k,u)$. So, for every $x \in K$, we get

\begin{align*}
    \mathcal{B}_{f^{k}}(x,n,v_{u}) &= \bigcap_{i = 0}^{n-1}f^{-ik}(\mathcal{B}(f^{ik}(x),v_{u})) \subseteq  \bigcap_{i = 0}^{n-1}f^{-ik}(\mathcal{B}_{f}(f^{ik}(x),k,u)) = \bigcap_{i = 0}^{n-1}f^{-ik}\left(\bigcap_{j = 0}^{k-1} f^{-j}(\mathcal{B}(f^{ik+j}(x),u))\right)\\ &=\bigcap_{i = 0}^{n-1} \bigcap_{j = 0}^{k-1} f^{-(ik+j)}( \mathcal{B}(f^{ik+j}(x),u))  = \bigcap_{i = 0}^{kn-1} f^{-i}( \mathcal{B}(f^{i}(x),u)) = \mathcal{B}_{f}(x,kn,u).
\end{align*}

So $g_{f,u,K}(kn) \leqslant g_{f^{k},v_{u},K}(n)$. By the inequalities $g_{f^{k},u,K}(n) \leqslant g_{f,u,K}(kn) \leqslant g_{f^{k},v_{u},K}(n)$ and varying $u$, we have that

\begin{align*}
    o(f^{k},K) &= \sup \{[g_{f^{k},u,K}(n)]: u \in \mathcal{U}\} \leqslant \sup\{[g_{f,u,K}(kn)]: u \in \mathcal{U}\} \leqslant \sup\{[g_{f^{k},v_{u},K}(n)]: u \in \mathcal{U}\} \\ & \leqslant \sup \{[g_{f^{k},u,K}(n)]: u \in \mathcal{U}\} = o(f^{k},K).
\end{align*}

Then $o(f^{k},K) = \sup\{[g_{f,u,K}(kn)]: u \in \mathcal{U}\}$. But $g_{f,u,K}(n)$ is a non decreasing map, which implies that

$$[g_{f,u,K}(n)] \leqslant [g_{f,u,K}(2n)] \leqslant \dots \leqslant [g_{f,u,K}(kn)] \leqslant \dots$$

Varying $u \in \mathcal{U}$, we obtain

$$\sup\{[g_{f,u,K}(n)]: u \in \mathcal{U}\} \leqslant \sup\{[g_{f,u,K}(2n)]: u \in \mathcal{U}\} \leqslant \dots \leqslant \sup\{[g_{f,u,K}(kn)]: u \in \mathcal{U}\} \leqslant \dots$$

Since $o(f^{k},K) = \sup\{[g_{f,u,K}(kn)]: u \in \mathcal{U}\}$, then we get

$$o(f,K) \leqslant o(f^{2},K) \leqslant \dots \leqslant o(f^{k},K) \leqslant \dots$$

Varying the compact $K$ in $X$, we get, finally

$$o(f) \leqslant o(f^{2}) \leqslant \dots \leqslant o(f^{k}) \leqslant \dots$$
\end{proof}

\begin{defi}\label{definitionoflyapunov} Let $(X,\mathcal{U})$ be a Hausdorff uniform space, $K$ a compact subset of $X$ and $f: X \rightarrow X$ a continuous map. We say that $f$ is Lyapunov stable in $K$ if, for every $u \in \mathcal{U}$, there exists $v \in \mathcal{U}$, such that for every $v$-small set $S\subseteq K$, $f^{n}(S)$ is $u$-small, for every $n \in \N$. 
\end{defi}

\begin{prop}\label{lyapunovstability}[Correa-Pujals, Theorem 3 of \cite{CoPu21} for compact metric spaces] Let $(X,\mathcal{U})$ be a Hausdorff uniform space, $K$ a compact subset of $X$ and $f: X \rightarrow X$ a continuous map. Then $f$ is Lyapunov stable in $K$ if and only if $o(f,K) = 0$.
\end{prop}

\begin{proof}$(\Rightarrow)$ Let $u \in \mathcal{U}$. Since $f$ is Lyapunov stable in $K$, then there exists $v \in \mathcal{U}$ such that for every $v$-small set $S\subseteq K$, $f^{n}(S)$ is $u$-small, for every $n \in \N$. Let $x \in K$ and $y \in K \cap \mathcal{B}(x,v)$, then, for every $n \in \N$, $(f^{n}(x),f^{n}(y)) \in u$, which implies that for every $n \in \N$, $y \in \mathcal{B}(x,n,u)$. Thus, for every $n \in \N$, $\mathcal{B}(x,v) \subseteq \mathcal{B}(x,n,u)$. Since $K$ is compact, then there exists $x_{1},\dots,x_{m} \in K$ such that $\{\mathcal{B}(x_{i},v): i \in \{1,\dots,m\}\}$ covers $K$, which implies that $\{\mathcal{B}(x_{i},n,u): i \in \{1,\dots,m\}\}$ covers $K$. So, for every $n \in \N$, $g_{f,u,K}(n) \leqslant \#\{\mathcal{B}(x_{i},n,u): i \in \{1,\dots,m\}\} \leqslant \#\{\mathcal{B}(x_{i},v): i \in \{1,\dots,m\}\}$, which is a constant that does not depend on $n$. So $[g_{f,u,K}(n)] = 0$. If we vary $u \in \mathcal{U}$, then we get that $o(f, K) = 0$.

$(\Leftarrow)$ Suppose that $f$ is not Lyapunov stable in $K$. Let $u \in \mathcal{U}$. Then, for every $v \in \mathcal{U}$, there exists $S_{v} \subseteq K$ and $m_{v} \in \N$ such that $S_{v}$ is $v$-small and $f^{m_{v}}(S_{v})$ is not $u$-small. So we can choose $x_{v},y_{v} \in S_{v}$ satisfying $(x_{v},y_{v}) \in v$ and $(f^{m_{v}}(x_{v}),f^{m_{v}}(y_{v})) \notin u$. Hence we get two nets $\{x_{v}\}_{v \in \mathcal{U}}$ and $\{y_{v}\}_{v \in \mathcal{U}}$ in $K$. Since $K$ is a compact space, then $\{x_{v}\}_{v \in \mathcal{U}}$ has a cluster point $z \in K$. Since the points of the two nets are arbitrarily close, then $z$ is also a cluster point of $\{y_{v}\}_{v \in \mathcal{U}}$.

Let $w \in \mathcal{U}$ such that $w$ is symmetric and $\overline{w}^{4} \subseteq u$, where $\overline{w}$ is the closure of $w$ in $X \times X$. Since $o(f,K) = 0$, then $s_{f,w,K}(n)$ is eventually constant. Let $n_{0} \in \N$ such that for every $n \geqslant n_{0}$, $s_{f,w,K}(n) = s_{f,w,K}(n_{0})$. Then we take a maximal $(n_{0},w)$-separated set $E$ in $K$. For any $n \geqslant n_{0}$, $E$ is also a $(n,w)$-separated set, which is maximal since $s_{f,w,K}(n) = s_{f,w,K}(n_{0})$. By the maximality of $E$, we have that $E$ is also a $(n,w)$-generator, for every $n \geqslant n_{0}$. So, for every $a \in K$, there exists $x_{n} \in E$ such that for every $i \in \{0,\dots,n-1\}$, $(f^{i}(a),f^{i}(x_{n})) \in w$. Since $E$ is finite, then  for every $a \in K$, there exists $x \in E$ such that for every $n \in \N$, $(f^{n}(a),f^{n}(x)) \in w$.

Since $z$ is a cluster point of $\{x_{v}\}_{v \in \mathcal{U}}$ and $\{y_{v}\}_{v \in \mathcal{U}}$, then there exists a cofinal subset $\Gamma \subseteq \mathcal{U}$, such that $z$ is a cluster point of $\{x_{v}\}_{v \in \Gamma}$ and $\{y_{v}\}_{v \in \Gamma}$ and, for every $v \in \Gamma$, $(z,x_{v}) \in w$ and $(z,y_{v}) \in w$. For every $x \in E$, take $\Gamma_{x} = \{v \in \Gamma: \forall n \in \N, (f^{n}(x),f^{n}(x_{v})) \in w\}$. We have that $\Gamma = \bigcup_{x \in E} \Gamma_{x}$. Since $E$ is finite, then there exists $x \in E$ such that $\Gamma_{x}$ is cofinal in $\Gamma$ and $z$ is a cluster point of the subnet $\{x_{v}\}_{v \in \Gamma_{x}}$. Then $z$ is also a cluster point of $\{y_{v}\}_{v \in \Gamma_{x}}$. Analogously, there exists $y \in E$ and a cofinal subset $\Gamma_{x,y}$ of $\Gamma_{x}$ such that  $z$ is a cluster point of $\{y_{v}\}_{v \in \Gamma_{x,y}}$ and for every $v \in \Gamma_{x,y}$ and for every $n \in \N$, $(f^{n}(y),f^{n}(y_{v})) \in w$. We have that  $z$ is also a cluster point of $\{x_{v}\}_{v \in \Gamma_{x,y}}$. 

Let $n \in \N$. Since $z$ is a cluster point of the nets $\{x_{v}\}_{v \in \Gamma_{x,y}}$ and $\{y_{v}\}_{v \in \Gamma_{x,y}}$ and for every $v \in \Gamma_{x,y}$, $(f^{n}(x),f^{n}(x_{v})) \in w$ and $(f^{n}(y),f^{n}(y_{v})) \in w$, then $(f^{n}(x),f^{n}(z)) \in \overline{w}$ and $(f^{n}(y),f^{n}(z)) \in \overline{w}$, which implies that $(f^{n}(x),f^{n}(y)) \in \overline{w}^{2}$ (note that $\overline{w}$ is symmetric since $w$ is symmetric). We have that, for every $n \in \N$ and every $v \in \Gamma_{x,y}$,  $(f^{n}(x),f^{n}(x_{v})) \in w$, $(f^{n}(x),f^{n}(y)) \in \overline{w}^{2}$ and $(f^{n}(y),f^{n}(y_{v})) \in w$, which implies that $(f^{n}(x_{v}),f^{n}(y_{v})) \in \overline{w}^{4} \subseteq u$, contradicting the fact that $(f^{m_{v}}(x_{v}),f^{m_{v}}(y_{v})) \notin u$. So $f$ is Lyapunov stable.
\end{proof}

\begin{defi}Let $X$, $Y$ be topological spaces and $f: X \rightarrow Y$ a continuous map. We say that $f$ is a compact-covering if for every compact $K \subseteq Y$, there exists a compact $K' \subseteq X$ such that $f(K') = K$.
\end{defi}

\begin{obs}Note that compact-covering maps are surjective. Examples of such maps are surjective proper maps (i.e. preimage of every compact set is compact) and, as a special case, surjective continuous maps between Hausdorff compact spaces. If $X$ is locally compact, then surjective open continuous maps also have this property. If $X$ is a complete metric space and $f$ is an open surjective continuous map, then $f$ is a compact-covering (Proposition 18 of Chapter IX, $\S$ 2.10 of \cite{Bou2}).
\end{obs}

\begin{prop}\label{semiconjporrecobrimentocompacto}Let $(X, \mathcal{U})$, $(Y,\mathcal{V})$ be uniform spaces, $f: X \rightarrow X$, $g: Y \rightarrow Y$ two continuous maps that are semi-conjugated by a uniformly continuous map $h: X \rightarrow Y$. If $C$ is a compact subset of $X$, then $o(g,h(C)) \leqslant o(f,C)$. Moreover, if $h$ is a compact-covering, then $o(g) \leqslant o(f)$.
\end{prop}

\begin{obs}This proposition generalizes Theorem 5 of \cite{Ho}, where it is done for topological entropy of uniform spaces.
\end{obs}

\begin{proof}Let $v \in \mathcal{V}$, $n \in \N$ and $E \subseteq f(C)$ a $(n,v)$-separated set with $\# E = s_{g,v,h(C)}(n)$. For every $y \in E$, take $x_{y} \in C$ such that $h(x_{y}) = y$ and take $F = \{x_{y}: y \in E\}$. We have that $F \subseteq C$ and $\# F = \# E$. Let $y,y' \in E$, with $y \neq y'$. Since $E \subseteq h(C)$ is $(n,v)$-separated, then there exists $i \in \{0,\dots,n-1\}$ such that $(g^{i}(y),g^{i}(y')) \notin v$. If $(f^{i}(x_{y}),f^{i}(x_{y'})) \in h^{-1}(v)$, then $(h \circ f^{i}(x_{y}),h \circ f^{i}(x_{y'})) \in v$. However,  $(h \circ f^{i}(x_{y}),h \circ f^{i}(x_{y'})) = (g^{i} \circ h(x_{y}), g^{i} \circ h(x_{y'})) = (g^{i}(y),g^{i}(y'))$, which implies that $(g^{i}(y),g^{i}(y')) \in v$, a contradiction. Then  $(f^{i}(x_{y}),f^{i}(x_{y'})) \notin h^{-1}(v)$, which implies that $F$ is $(n,h^{-1}(v))$-separated. So, for every $n \in \N$, $s_{g,v,h(C)}(n) \leqslant s_{f,h^{-1}(v),C}(n)$, which implies that $[s_{g,v,h(C)}(n)] \leqslant [s_{f,h^{-1}(v),C}(n)]$.

If we vary $v \in \mathcal{V}$, we get that $o(g,h(C)) = \sup\limits_{v \in \mathcal{V}} [s_{g,v,h(C)}(n)] \leqslant \sup\limits_{v \in \mathcal{V}} [s_{f,h^{-1}(v),C}(n)] \leqslant \sup\limits_{u \in \mathcal{U}} [s_{f,u,C}(n)] = o(f, C)$.

Suppose that the map $h$ is a compact-covering. Let $K$ be a compact subset of $Y$. Then there exists $C_{K}$, a compact subset of $X$, such that $h(C_{K}) = K$. We showed that $o(g,K) \leqslant o(f,C_{K})$. If we vary the compact set $K$ in $Y$, we get that 
$o(g) = \sup \{o(g,K): K \subseteq Y$ is compact$\} \leqslant \sup \{o(f,C_{K}): K \subseteq Y$ is compact$\} \leqslant \sup \{o(f,K'): K' \subseteq X$ is compact$\} = o(f)$.
\end{proof}

\subsection{Coding and mutually singular sets}

Let $(X,\mathcal{U})$ be a Hausdorff compact uniform space, let $f:X \to X$ be a homeomorphism, and let $\Omega(f)$ be the non-wandering set. Let $\mathcal{F}$ be a finite family of non-empty subsets of $X - \Omega(f)$. We denote by $\cup \mathcal{F}$ the union of all the elements of $\mathcal{F}$ and by $\infty_{\mathcal{F}}$ the complement of $\cup \mathcal{F}$. Let us fix $n\in \N$ and consider $\underline{x}=(x_0,\dots,x_{n-1})$ a finite sequence of points in $X$ and $\underline{w}=(w_0,\dots,w_{n-1})$ a finite sequence of elements of $\mathcal{F}\cup \{\infty_{\mathcal{F}}\}$. We say that $\underline{w}$ is a coding of $\underline{x}$, relative to $\mathcal{F}$, if for every $i=0,\dots, n-1$, we have $x_i\in w_i$. Each $w_i$ is called a letter of the coding $\underline{w}$. Whenever the family $\mathcal{F}$ is fixed, we simplify the notation by using $\infty$ instead of $\infty_{\mathcal{F}}$. Note that $\mathcal{F}$ is a disjoint family (i.e., the sets in the family $\mathcal{F}$ are pairwise disjoint) if and only if we can have only one coding for a given sequence. We denote the set of all the codings of all orbits $(x,f(x),\dots, f^{n-1}(x))$ of length $n$ by $\mathcal{A}_{n}(f,\mathcal{F})$. We define the sequence $c_{f,\mathcal{F}}(n)=\# \mathcal{A}_{n}(f,\mathcal{F})$, and it is easy to see that $c_{f,\mathcal{F}}(n)\in \mathcal{O}$.

We say that a set $Y$ is wandering if $f^n(Y)\cap Y =\emptyset$ for every $n\geq 1$.

\begin{lema}(Correa-de Paula, Lemma 4.1 of \cite{CP2} for compact metric spaces) Let $(X,\mathcal{U})$ be a Hausdorff compact uniform space and $f: X \rightarrow X$ a homeomorphism. For every finite disjoint family $\mathcal{F}$ of compact subsets that are wandering, there exists $u \in \mathcal{U}$ such that for every $n \in \N$, $c_{f,\mathcal{F}}(n) \leqslant 2^{\# \mathcal{F}} s_{f,u}(n)$. 
\end{lema}

\begin{proof}
Let $\mathcal{F}=\{Y_1,\dots, Y_k\}$ be a finite disjoint family of compact subsets that are wandering and fix $n \in \N$. Choose some compact disjoint respective neighborhoods $U_1,\dots, U_k$ of the sets $Y_1,\dots, Y_k$ such that for every $i \in \{-n+1,\dots,n-1\}$ and every $r \in \{1,\dots,k\}$, $f^{i}(U_{r}) \cap U_{r} = \emptyset$. Let $u\in\mathcal{U}$ such that $\mathcal{B}(Y_i,u)\cap (X-U_{i}) = \emptyset$, for every $i=1,\dots, k$. 
	
For every $\mathcal{G}\subseteq \mathcal{F}$, let $\mathcal{A}_n(\mathcal{F},\mathcal{G})$ denote the set of elements of $\mathcal{A}_n(f,\mathcal{F})$ whose set of the letters is exactly $\mathcal{G}\cup \{\infty_{\mathcal{F}}\}$. We fix some $\mathcal{G}\subseteq \mathcal{F}$ and we consider two points $x,y$ in $X$ and two words $\underline{w}=(w_0,\dots, w_{n-1})$, $\underline{z}=(z_0,\dots, z_{n-1})$ in $\mathcal{A}_n(\mathcal{F},\mathcal{G})$ which are codings for the orbits $(x,\dots, f^{n-1}(x))$ and $(y,\dots, f^{n-1}(y))$, respectively. Then, we claim that if the symbols $\underline{w}$ and $\underline{z}$ are distinct, then the points $x$ and $y$ are $(n,u)$-separated. Indeed, let $i\in\{0,\dots, n-1\}$ be such that $w_i\neq z_i$. If both $w_i\neq \infty$, $z_i\neq \infty$, then $f^i(x)$ and $f^i(y)$ belongs to distinct sets $Y_i$'s, and $(f^i(x),f^i(y))\notin u$. If, say, $w_i=\infty$, then $f^i(y)\in z_{i} = Y_{r}$, for some $r \in \{1,\dots,k\}$, and $f^i(x)\notin Y_{r}$. By definition of $\mathcal{A}_n(\mathcal{F},\mathcal{G})$ there exists some $j\neq i$ in $\{0,\dots, n-1\}$ such that $f^j(x)\in Y_{r}\subseteq U_{r}$. Since $f^{i-j}(U_{r})\cap U_{r} = \emptyset$, we see that $f^i(x)\notin U_{r}$, thus again $(f^i(x),f^i(y))\notin u$, and the claim is proved. 
	
As an immediate consequence we have that $\#\mathcal{A}_n(\mathcal{F},\mathcal{G})\leqslant s_{f,u}(n)$. Since the $\mathcal{A}_n(\mathcal{F},\mathcal{G})$'s form a partition of $\mathcal{A}_n(f,\mathcal{F})$ into $2^{\# \mathcal{F}}$ elements we have	
	$$c_{f,\mathcal{F}}(n)\leqslant 2^{\# \mathcal{F}}\cdot s_{f,u}(n),$$
as we wanted.	    
\end{proof}

\begin{prop}\label{prop2.21}(Correa-de Paula, Theorem 2 of \cite{CP2} for compact metric spaces) Let $(X,\mathcal{U})$ be a Hausdorff compact uniform space and $f: X \rightarrow X$ a homeomorphism. If $\mathcal{F}$ is a finite disjoint family of compact subsets that are wandering of $f$, then $[c_{f,\mathcal{F}}(n)] \leqslant o(f)$.
\end{prop}

\begin{proof}By the last lemma, there exists $u \in \mathcal{U}$ such that for every $n \in \N$, $c_{f,\mathcal{F}}(n) \leqslant 2^{\# \mathcal{F}} s_{f,u}(n)$, which implies that $[c_{f,\mathcal{F}}(n)] \leqslant [s_{f,u}(n)]$. Since  $[s_{f,u}(n)] \leqslant o(f)$, then we get that $[c_{f,\mathcal{F}}(n)] \leqslant o(f)$.
\end{proof}

Now, we want to discuss some properties of $[c_{f,\mathcal{F}}(n)]$. For every subset $Y$ of $X- \Omega(f)$, let $M(Y)$ denote the maximum number of terms of an orbit that belongs to $Y$, $M(Y)=\sup_{x\in X} \#\{n\in \Z: f^n(x)\in Y\}$. Observe that if $Y\subseteq (X-\Omega(f))$ is compact, it may be covered by a finite number of wandering open sets, and every orbit intersects a wandering set at most once; thus, in this case $M(Y)$ is finite. If $\mathcal{F}$ is a finite family of subsets of $X$ such that $M(\cup \mathcal{F})<\infty$. Then it satisfies:
	\begin{itemize}
		\item[1.] (Monotonicity) Let $\mathcal{F}'$ be another finite family of subsets of $X$. If each element of $\mathcal{F}'$ is included in an element of $\mathcal{F}$, that we will denote $\mathcal{F}'\subseteq \mathcal{F}$, then $$[c_{f,\mathcal{F}'}(n)]\leqslant [c_{f,\mathcal{F}}(n)].$$
		\item[2.] (Additivity) $$[c_{f,\cup \mathcal{F}}(n)]=[c_{f,\mathcal{F}}(n)].$$
		\item[3.] (Disjoint representatives) Let $\mathcal{F}=\{Y_1,\dots, Y_L\}$ be a family of compact subsets of $X-\Omega(f)$ such that for every $i,j \in \{1,\dots,L\}$, $Y_{i} \cap Y_{j} = \emptyset$ or $Y_{i} \cup Y_{j}$ is wandering. Then $[c_{f,\mathcal{F}}(n)] = \sup\{[c_{f,\mathcal{F}'}(n)]: \mathcal{F}'\subseteq \mathcal{F} \,\text{is disjoint}\}$.
	\end{itemize} 

The monotonicity and the additivity properties work in the same way as the metric case, for the polynomial entropy, as shown by Hauseux and Le Roux (Lemme 2.2, from \cite{HR}). The disjoint representatives property is shown in \cite{CP2}, Lemma 4.3, and also works in the same way as the metric case. So, we omit the proofs here.

The following lemma says that we can estimate the order of the growth of the cardinality of the set of codings relative to a compact subset $Y$ of $ X-\Omega (f)$ in terms of a family of subsets of $Y$ that are $u$-small. It is a generalization of the construction made for the polynomial entropy, by Hauseux and Le Roux (Sous-lemme 2.5 from \cite{HR}).

\begin{lema}\label{lema222}
   Let $(X,\mathcal{U})$ be a Hausdorff compact uniform space and $f: X\rightarrow X$ be a homeomorphism. For every compact subset $Y$ of $X-\Omega(f)$, and every $u\in \mathcal{U}$ there exists a finite family $\mathcal{F}= \{Y_1,\dots,Y_L\}$ of wandering compact subsets of $Y$ that are $u$-small, such that
	$$[c_{f,Y}(n)]\leqslant \sup\{[c_{f,\mathcal{F}'}(n)]: \mathcal{F}'\subseteq \mathcal{F} \,\text{is disjoint}\}.$$
\end{lema}

\begin{proof} Since $Y\subseteq X-\Omega(f)$, every point of $Y$ admits a wandering compact neighborhood, and by compactness there exists an open cover $\{U_{1},...,U_{m}\}$ of $Y$ by $u$-small wandering sets. Take a Lebesgue entourage $v$ for this cover and take $w\in \mathcal{U}$ that satisfies $w \subseteq u$ and $w^{2} \subseteq v$. So, for every $w$-small sets $W,W' \subseteq X-\Omega(f)$, with $W\cap W' \neq \emptyset$, $W\cup W'$ is $w^{2}$-small, which implies that it is $v$-small and then it is wandering, since it is contained in some $U_{i}$.  Again by compactness, there is a finite cover $\mathcal{F}=\{Y_1,\dots, Y_k\}$ of $Y$ by wandering compact subsets of $Y$ that are $w$-small. Then, for every $i,j \in\{1,\dots,k\}$, if $Y_i$ meets $Y_j$, then $Y_i\cup Y_j$ is a wandering set. So, by the disjoint representatives property, $[c_{f,Y}(n)]\leqslant [c_{f,\mathcal{F}}(n)] = \sup\{[c_{f,\mathcal{F}'}(n)]: \mathcal{F}'\subseteq \mathcal{F} \,\text{is disjoint}\}$.
\end{proof}

The following lemma says that we can estimate the generalized entropy from below by the codings relative to a compact subset $Y$ of $ X-\Omega (f)$. It is also a generalization of the construction made for the polynomial entropy, by Hauseux and Le Roux (Lemme 2.4 from \cite{HR}).

 \begin{lema}\label{lema223}
Let $(X,\mathcal{U})$ be a Hausdorff compact uniform space and $f: X\rightarrow X$ be a homeomorphism.	For every compact subset $Y$ of $X-\Omega(f)$ we have $$[c_{f,Y}(n)]\leqslant o(f).$$
\end{lema}

\begin{proof} By the {\bf Lemma \ref{lema222}} there exists a finite family $\mathcal{F}$ of compact subsets of $Y$ whose elements are wandering and such that $[c_{f,Y}(n)]\leqslant\sup\{[c_{f,\mathcal{F}'}(n)]: \mathcal{F}'\subseteq \mathcal{F} \text{ is disjoint}\}$. And, by {\bf Proposition \ref{prop2.21}}, if $\mathcal{F}'$ is a finite disjoint family of compact subsets that are wandering by $f$, then $[c_{f,\mathcal{F}'}(n)] \leqslant o(f)$, which implies $\sup\{[c_{f,\mathcal{F}'}(n)]: \mathcal{F}'\subseteq \mathcal{F} \text{ is disjoint}\}\leqslant o(f)$, and $[c_{f, Y}(n)]\leqslant o(f)$, as we wanted.
\end{proof}

\begin{defi} \label{mutuallysingular} Let $(X,\mathcal{U})$ be a Hausdorff compact uniform space and $f: X \rightarrow X$ a homeomorphism. A disjoint family $Y_{1},\dots,Y_{k}$ of subsets of $X-\Omega(f)$ is mutually singular if for every $n_{0} \in \N$, there exists $x \in X$ and $n_{1},\dots,n_{k} \in \N$, such that $f^{n_{i}}(x) \in Y_{i}$, for every $i\in\{1,\dots,k\}$ and if $i \neq j$, then $|n_{i}-n_{j}| > n_{0}$. A disjoint family $Y_{1},\dots,Y_{k}$ of subsets of $X-\Omega(f)$ is almost mutually singular if for every choice of pairwise disjoint open neighborhoods $U_{i}$ of $Y_{i}$ that are contained in $\Omega(f)$, the sets  $U_{1},\dots,U_{k}$ are mutually singular.    
\end{defi}

\begin{lema}\label{morethanlinear} Let $(X,\mathcal{U})$ be a Hausdorff compact uniform space and $f: X \rightarrow X$ a homeomorphism. If $\mathcal{F} = \{Y_{1},Y_{2}\}$ is a mutually singular family of compact wandering sets, then $[c_{f,\mathcal{F}}(n)] > [n]$.   
\end{lema}

\begin{proof}By the monotonicity property, $[c_{f,\mathcal{F}}(n)] \geqslant [c_{f,Y_{1}}(n)]$ and $[c_{f,Y_{1}}(n)] = [n]$, since $Y_{1}$ is wandering (so all codings with size $n$ for the family $\{Y_{1}\}$ have the form $(\infty, \dots, \infty, Y_{1}, \infty \dots, \infty)$, which are $n$ possibilities). Thus $[c_{f,\mathcal{F}}(n)] \geqslant [n]$. 

Consider the sets $R(Y_{i},Y_{j}) = \{m \in \N: f^{m}(Y_{i})\cap Y_{j} \neq \emptyset\}$. Since $\mathcal{F}$ is mutually singular, then $R(Y_{1},Y_{2}) \cup R(Y_{2},Y_{1})$ is an infinite set, which implies that one of the sets $R(Y_{1},Y_{2})$ or $R(Y_{2},Y_{1})$ must be infinite. Suppose, without loss of generality, that $R(Y_{1},Y_{2})$ is infinite. We define also, for $n \in \N$, $R_{n}(Y_{1},Y_{2}) = \{m \in R(Y_{i},Y_{j}): m \leqslant n\}$.

Consider $\mathcal{D}_{n}(f,\mathcal{F}) = \{\underline{w} \in \mathcal{A}_{n}(f,\mathcal{F}): \underline{w} = (\infty, \dots, \infty, Y_{1}, \infty \dots, \infty,Y_{2},\infty \dots, \infty) \}$. In other words, $\mathcal{D}_{n}(f,\mathcal{F})$ have all codings of $\mathcal{A}_{n}(f,\mathcal{F})$ such that the first letter different from $\infty$ is $Y_{1}$ and $Y_{2}$ appears as a letter of the coding. Note that since $Y_{1}$ and $Y_{2}$ are wandering, both letters cannot appear more than once in an element of $\mathcal{A}_{n}(f,\mathcal{F})$.

Let $d_{f,\mathcal{F}}(n) = \# \mathcal{D}_{n}(f,\mathcal{F})$. Since $\mathcal{D}_{n}(f,\mathcal{F}) \subseteq \mathcal{A}_{n}(f,\mathcal{F})$, then, for every $n \in \N$,  $d_{f,\mathcal{F}}(n) \leqslant c_{f,\mathcal{F}}(n)$.

Let $\underline{w} \in \mathcal{D}_{n}(f,\mathcal{F})$. Then 

$$\underline{w} = \underbrace{(\infty,\dots,\infty, Y_{1}}_{n'},\underbrace{\overbrace{\infty, \dots, \infty,Y_{2}}^{m},\infty, \dots, \infty}_{n-n'}).$$

For some $n' < n$ and $m \leqslant n-n'$. These choices of $n'$ and $m$ are possible if and only if $m \leqslant n-n'$ and $f^{m}(Y_{1}) \cap Y_{2} \neq \emptyset$ (or equivalently, if $m \in R_{n-n'}(Y_{1},Y_{2})$). So the set of elements of $\mathcal{D}_{n}(f,\mathcal{F})$ where the $Y_{1}$ is in the $n'$-th position has cardinality $\# R_{n-n'}(Y_{1},Y_{2})$ and then $d_{f,\mathcal{F}}(n) = \sum_{n' = 1}^{n-1} \# R_{n-n'}(Y_{1},Y_{2})$.

Suppose that $[d_{f,\mathcal{F}}(n)] \leqslant [n]$. Then there exists $c > 0$ such that, for every $n \in \N$, $d_{f,\mathcal{F}}(n) \leqslant cn$. Since $R(Y_{1},Y_{2})$ is infinite and $R(Y_{1},Y_{2}) = \bigcup_{i \in \N} R_{i}(Y_{1},Y_{2})$, then there exists $n_{0} \in 2\N$ such that $\#R_{\frac{n_{0}}{2}}(Y_{1},Y_{2}) > 2c$. We have that, for every $n' \leqslant \frac{n_{0}}{2}$, $R_{n_{0}-n'}(Y_{1},Y_{2}) \supseteq R_{\frac{n_{0}}{2}}(Y_{1},Y_{2})$, which implies that $\#R_{n_{0}-n'}(Y_{1},Y_{2}) > 2c$. So  
$$d_{f,\mathcal{F}}(n_{0}) = \sum_{n' = 1}^{n_{0}-1} \# R_{n_{0}-n'}(Y_{1},Y_{2}) \geqslant \sum_{n' = 1}^{\frac{n_{0}}{2}} \# R_{n_{0}-n'}(Y_{1},Y_{2}) > \sum_{n' = 1}^{\frac{n_{0}}{2}}2c = \frac{2cn_{0}}{2} = cn_{0}.$$

This contradicts the fact that, for every $n \in \N$, $d_{f,\mathcal{F}}(n) \leqslant cn$. Then $[d_{f,\mathcal{F}}(n)] \nleqslant [n]$. Since  $[d_{f,\mathcal{F}}(n)] \leqslant [c_{f,\mathcal{F}}(n)]$, then $[c_{f,\mathcal{F}}(n)] \nleqslant [n]$. Thus $[c_{f,\mathcal{F}}(n)] > [n]$.
\end{proof}

The lemma above is enough to state a condition to have a lower bound of the generalized entropy. However, with the next few lemmas, we are able to remove the hypothesis that the mutually singular family is wandering.

\begin{lema}\label{almostmutualsingsubdivisions}Let $(X,\mathcal{U})$ be a Hausdorff compact uniform space and $f: X \rightarrow X$ a homeomorphism. Consider the sets $Y_{1},\dots,Y_{k}$ that are almost mutually singular. Suppose that for each $i \in \{1,\dots k\}$, $Y_{i} = \bigcup_{j = 1}^{r_{i}} Y_{i,j}$. Then there exists $s_{i} \in \{1,\dots,r_{i}\}$ such that the sets $Y_{1,s_{1}},...,Y_{k,s_{k}}$ are almost mutually singular.
\end{lema}

\begin{proof}Suppose that is not the case. Then, for every $s = (s_{1},\dots,s_{k})$ with $s_{i} \in \{1,\dots, r_{i}\}$, the sets $Y_{1,s_{1}},..., Y_{k,s_{k}}$ are not almost mutually singular. So 
there exists an open neighborhood $U_{i,s}$ of $Y_{i,s_{i}}$ contained in $X-\Omega(f)$ and such that $U_{1,s},..., U_{k,s}$ are pairwise disjoint but they are not mutually singular. Let $\Pi = \{1,\dots, r_{1}\}\times \dots \times \{1,\dots, r_{k}\}$,  $\Pi_{i,j} = \{(s_{1},\dots,s_{k}) \in \Pi: s_{i} = j\}$ and take $U_{i,j} = \bigcap_{s \in \Pi_{i,j}} U_{i,s}$. We have that $U_{i,j}$ is an open neighborhood of $Y_{i,j}$. Let $(s_{1},\dots,s_{k}) \in \Pi$. Since, for every $s \in \Pi_{i,s_{i}}$, $U_{i,s_{i}} \subseteq U_{i,s}$, then, $U_{1,s_{1}},..., U_{k,s_{k}}$ are pairwise disjoint and they are not mutually singular.

Let $U_{i} = \bigcup_{j = 1}^{r_{i}} U_{i,j}$. We have that $U_{i}$ is an open neighborhood of $Y_{i}$. Since $Y_{1},\dots,Y_{k}$ are almost mutually singular, then $U_{1},\dots,U_{k}$ are mutually singular (they are pairwise disjoint since each sets $U_{1,s_{1}},..., U_{k,s_{k}}$ are pairwise disjoint). Then, for every $m \in \N$, there exists $x_{m} \in X$ and $n_{1,m},\dots,n_{k,m} \in \N$, such that $f^{n_{i,m}}(x_{m}) \in U_{i}$, for every $i\in\{1,\dots,k\}$ and if $i \neq j$, then $|n_{i,m}-n_{j,m}| > m$. Since $\Pi$ is finite, then there exists an infinite set $\{m_{l}\}_{l \in \N}$ and $(s_{1},...,s_{k})\in \Pi$ such that for every $l \in \N$, $f^{n_{i,m_{l}}}(x_{m_{l}}) \in U_{i,s_{i}}$. So $U_{1,s_{1}},..., U_{k,s_{k}}$ are mutually singular, which is a contradiction.

Thus there exists $s_{i} \in \{1,\dots,r_{i}\}$ such that the sets $Y_{1,s_{1}},...,Y_{k,s_{k}}$ are almost mutually singular.
\end{proof}

\begin{lema}Let $(X,\mathcal{U})$ be a Hausdorff compact uniform space and $f: X \rightarrow X$ a homeomorphism. Consider $\Gamma$ a directed set and a collections of compact sets $\{Y_{i,\alpha}\}_{\alpha \in \Gamma}$, with $i \in \{1,\dots,k\}$, such that for every $\beta < \alpha$, $Y_{i,\alpha} \subseteq Y_{i,\beta}$ and, for every $\alpha \in \Gamma$, the sets $Y_{1,\alpha},\dots,Y_{k,\alpha}$ are almost mutually singular. If $Y_{i} = \bigcap_{\alpha \in \Gamma} Y_{i,\alpha}$, then the sets $Y_{1},\dots,Y_{k}$ are almost mutually singular.
\end{lema}

\begin{proof}Since  $\{Y_{i,\alpha}\}_{\alpha \in \Gamma}$, with $i \in \{1,\dots,k\}$, satisfies the finite intersection property, then $Y_{i} \neq \emptyset$. Let $U_{1},...,U_{k}$ be pairwise disjoint open neighborhoods of $Y_{1},...,Y_{k}$, respectively. Since $U_{i}$ is an open neighborhood of $Y_{i}$ and $Y_{i} = \bigcap_{\alpha \in \Gamma} Y_{i,\alpha}$, with compact $Y_{i,\alpha}$, then there exists $\alpha_{i} \in \Gamma$ such that $Y_{i,\alpha_{i}} \subseteq U_{i}$ (Corollary 3.1.5 of \cite{Eng}). Take $\alpha_{0} > \alpha_{i}$, for every $i\in\{1,\dots,k\}$. Then $Y_{i,\alpha_{0}} \subseteq Y_{i,\alpha_{i}} \subseteq U_{i}$. Therefore, $U_{1},...,U_{k}$ are pairwise disjoint open neighborhoods of $Y_{1,\alpha_{0}},...,Y_{k,\alpha_{0}}$, respectively, and $Y_{1,\alpha_{0}},...,Y_{k,\alpha_{0}}$ are almost mutually singular, which implies that $U_{1},...,U_{k}$ are mutually singular. Thus $Y_{1},\dots,Y_{k}$ are almost mutually singular.
\end{proof}

\begin{lema}Let $(X,\mathcal{U})$ be a Hausdorff compact uniform space and $f: X \rightarrow X$ a homeomorphism. Consider the compact sets $Y_{1},\dots,Y_{k}$ that are almost mutually singular. Then there exists $y_{i} \in Y_{i}$ such that the points $y_{1},...,y_{k}$ are almost mutually singular.
\end{lema}

\begin{obs}This is stated by Hauseux and Le Roux for compact mutually singular sets in compact metric spaces (page 9 of \cite{HR}).
\end{obs}

\begin{proof}By Zermelo's Theorem, we can give to the set $\mathcal{U}$ a well-order $\leqslant$. We can suppose that $X\times X = \min \mathcal{U}$.

We want to construct a collection of families of closed sets $\{Y_{i,u}\}_{u \in \mathcal{U}}$, for each $i \in \{1,\dots,k\}$, satisfying the following properties:

\begin{enumerate}
    \item For every $i\in \{1,\dots,k\}$, $Y_{i,X\times X} = Y_{i}$.
    \item For every $i\in \{1,\dots,k\}$ and $u,v \in \mathcal{U}$, with $v < u$, $Y_{i,v} \supseteq Y_{i,u}$.  
    \item For every $i\in \{1,\dots,k\}$ and $u \in \mathcal{U}$, $Y_{i,u}$ is $u$-small.
    \item For every $u \in \mathcal{U}$, the family $Y_{1,u},\dots, Y_{k,u}$ is almost mutually singular.
\end{enumerate}

We do this construction by transfinite induction on $(\mathcal{U}, \leqslant)$:

Let $u \in \mathcal{U}$. We suppose that, for every $v < u$, there exists the families of closed sets $\{Y_{i,v}\}_{v < u}$ satisfying the conditions above. We need to consider two cases:

Suppose that $u = X\times X$. Clearly $Y_{1,X\times X},...,Y_{k,X\times X}$ satisfies the conditions above.

Suppose now that $u \neq X\times X$.

Let $Z_{i,u} = \bigcap_{v < u}Y_{i,v}$. Since each $Y_{i,v}$ is compact, then, by the finite intersection property, we get that $Z_{i,u} \neq \emptyset$. By the last lemma, we get that the family $Z_{1,u},\dots,Z_{k,u}$ is almost mutually singular. For each $i \in \{1,\dots,k\}$, we can take compact $u$-small sets $Z_{i,u,1},..., Z_{i,u,r_{i}}$ such that $Z_{i,u} = \bigcup_{j = 1}^{r_{i}} Z_{i,u,j}$. By \textbf{Lemma \ref{almostmutualsingsubdivisions}}, there exists $s_{i} \in \{1,\dots,r_{i}\}$ such that $Z_{1,u,s_{1}},..., Z_{k,u,s_{k}}$ are almost mutually singular. Then we can take $Y_{i,u}$ as $Z_{i,u,s_{i}}$.

So this concludes the construction.

Consider the sets $Z_{i} = \bigcap_{u \in \mathcal{U}} Y_{i,u}$. Since each $Y_{i,u}$ is compact, then, by the finite intersection property, we get that $Z_{i} \neq \emptyset$.  Since $Y_{i,u}$ is $u$-small, then $Z_{i}$ is $u$-small, for each $u \in \mathcal{U}$. Since $X$ is a Hausdorff space, then $Z_{i}$ is a singleton. So we take $y_{i}$ as the unique point in $Z_{i}$. Then, by the last lemma, the points $y_{1},...,y_{k}$ are almost mutually singular.
\end{proof}

\begin{cor}\label{wedontneedwandering}Let $(X,\mathcal{U})$ be a Hausdorff compact uniform space and $f: X \rightarrow X$ a homeomorphism. Consider the compact sets $Y_{1},\dots,Y_{k}$ that are almost mutually singular. Then there exists compact sets $Y'_{i} \subseteq X-\Omega(f)$ such that the  $Y'_{1},...,Y'_{k}$ are a collection of wandering mutually singular sets.
\end{cor}

\begin{proof}By the last lemma, there exists  $y_{i} \in Y_{i}$ such that the points $y_{1},...,y_{k}$ are almost mutually singular. Since $y_{1},...,y_{k} \in X-\Omega(f)$, there exists $U_{1},...,U_{k}$ pairwise disjoint open sets such that $y_{i} \in U_{i}$ and $U_{i}$ is wandering. Take an open set $V_{i}$ such that $y_{i} \in V_{i} \subseteq \overline{V}_{i} \subseteq U_{i}$. Then, take $Y'_{i} =  \overline{V}_{i}$. So the sets $Y'_{i}$ are wandering sets and $Y'_{1},...,Y'_{k}$ are a collection of wandering mutually singular sets, since $V_{1},...,V_{k}$ are a collection of wandering mutually singular sets.  
\end{proof}

Now we are able to conclude a condition to have a lower bound for the generalized entropy.

\begin{prop}\label{Entropyviamutuallydisjointsets} Let $(X,\mathcal{U})$ be a Hausdorff compact uniform space and $f: X \rightarrow X$ a homeomorphism. If there exists a finite disjoint family $\mathcal{F}$ of subsets of $X-\Omega(f)$ that is mutually singular, with $\# \mathcal{F} \geqslant 2$, then $o(f) > [n]$.
\end{prop}

\begin{proof}The family $\mathcal{F}$ is mutually singular. By \textbf{Corollary \ref{wedontneedwandering}}, there exists a wandering mutually singular family $\mathcal{F}'$ such that $\#\mathcal{F}' = \# \mathcal{F}$. Let $\mathcal{F}''$ be any subfamily of $\mathcal{F}'$ such that $\#\mathcal{F}'' = 2$. Then, by \textbf{Lemma \ref{morethanlinear}}, $[c_{f,\mathcal{F}}(n)] > [n]$. Since $[c_{f,\mathcal{F}}] = [c_{f,\cup\mathcal{F}}]$, then it follows from \textbf{Lemma \ref{lema223}} that $o(f) \geqslant [c_{f,\mathcal{F}}]$. Thus $o(f) > [n]$.
\end{proof}

\section{Entropy for Parabolic Dynamics}\label{4}

In this section, we apply the framework of generalized entropy to the study of parabolic and generalized parabolic dynamics. These systems, characterized by convergence of orbits toward a fixed set both forward and backward in time, often exhibit zero topological entropy despite having nontrivial dynamical behavior. 

\subsection{Parabolic Dynamics}

Here, we introduce the definition of generalized parabolic dynamics and discuss their equivalent formulations and dynamical properties. 

\begin{defi}Let $X$ be a Hausdorff compact space and $f: X \rightarrow X$ a homeomorphism. We say that $f$ has generalized parabolic dynamics if there exists a non empty invariant closed set $F \subseteq X$ such that for every open neighborhood $U$ of $F$, and every compact set $K \subseteq X - F$, there exists $n_{0} \in \N$ such that for every $n \geqslant n_{0}$, $f^{n}(K) \subseteq U$ and $f^{-n}(K) \subseteq U$ (we say in this case that $\{f^{n}(K)\}_{n \in \N}$ and $\{f^{n}(K)\}_{n \in \N}$ converge uniformly to a subset of $F$). We also say that $F$ is a parabolic set associated to $f$. If $F$ is a single point, then we say that $f$ has parabolic dynamics.

We say that $f$ has North-South dynamics if there exists fixed points $p$ and $q$ in $X$ satisfying:

\begin{enumerate}
    \item For every neighborhood $V$ of $q$ and every compact set $K \subseteq X - \{p\}$, there exists $n_{0} \in \N$ such that for every $n \geqslant n_{0}$, $f^{n}(K) \subseteq V$.
    \item For every neighborhood $U$ of $p$ and every compact set $K \subseteq X - \{q\}$, there exists $n_{0} \in \N$ such that for every $n \geqslant n_{0}$, $f^{-n}(K) \subseteq U$.
\end{enumerate}

In this case, we say that $p$ is the repulsor point and $q$ is the attractor point for $f$. 
\end{defi}

\begin{obs}Note that the north-south dynamics are particular cases of generalized parabolic dynamics. If $f$ has generalized parabolic dynamics with parabolic set $F$, then the non-wandering set $\Omega(f)$ is contained in $F$. In the special case of parabolic dynamics with parabolic point $p$, then  $\Omega(f) = \{p\}$. If $f$ has north-south dynamics with fixed points $p$ and $q$, then $\Omega(f) = \{p,q\}$.
\end{obs}

A homeomorphism $f: X \rightarrow X$ has generalized parabolic dynamics with parabolic set $F$ if and only if the action of the cyclic group $\langle f \rangle$ on $X-F$ is properly discontinuous (i.e., for every compact sets $K,K' \subseteq X-F$, the set $\{g \in \langle f \rangle: gK\cap K' \neq \emptyset\}$ is finite).

\begin{prop}\label{properlydiscontinuousisparabolic}Let $X$ be a Hausdorff compact space, $f: X \rightarrow X$ a homeomorphism and $F$ a closed invariant set of $X$ such that $\Omega(f) \subseteq F$ and $f' = f|_{X-F}$. Then $\langle f' \rangle$ acts properly discontinuously on $X-F$ if and only if $f$ has generalized parabolic dynamics with parabolic set $F$.
\end{prop}

\begin{proof}Let $U$ be an open neighborhood of $F$ and $K$ a compact subset of $X-F$. Then the set $K' = X-U$ is closed in $X$, and therefore, it is compact. If the group $\langle f'\rangle$ acts properly discontinuously on $X-F$, the set $\{n \in \Z: f'^{n}(K)\cap K' \neq \emptyset\}$ is finite. But $\{n \in \Z: f'^{n}(K)\cap K' \neq \emptyset\} = \{n \in \Z: f'^{n}(K) \subseteq U\} = \{n \in \Z: f^{n}(K) \subseteq U\}$, which implies that there exists $n_{0} \in \N$ such that for every $n \geqslant n_{0}$, $f^{n}(K) \subseteq U$ and $f^{-n}(K) \subseteq U$. Thus, $f$ has generalized parabolic dynamics with parabolic set $F$. The converse is analogous.
\end{proof}

Since $\langle f \rangle$ is torsion-free and $X-F$ is locally compact, then a properly discontinuous action is equivalent to be a covering action (i.e. every point has a neighborhood $U$ such that for every $n \neq 0$, $f^{n}(U) \cap U = \emptyset$), which is also equivalent to ask that the quotient map to the space of orbits of $X-F$ to be a covering map.

\subsection{Entropy}

Here, we analyze the generalized entropy of systems with parabolic or generalized parabolic dynamics.

\begin{lema}\label{compactsaredecomposedingeneralizedentropy}Let $(X,\mathcal{U})$ be a Hausdorff compact uniform space, a homeomorphism $f: X \rightarrow X$ which has generalized parabolic dynamics with parabolic set $F=\{y_{1},\dots,y_{m}\}$, whose points are all fixed by $f$, and $K$ a compact set of $X-F$. Then there exists compact sets $K^{1,1},\dots,K^{m,m}$, some of these possibly empty, such that $K = \bigcup_{i,j = 1}^{m} K^{i,j}$ and for every $V_{1},\dots,V_{m}$ neighborhoods of $y_{1},\dots,y_{m}$, respectively, there exists $n_{0} \in \N$ such that for every $n \in \N$, with $n \geqslant  n_{0}$, $f^{n}(K^{i,j}) \subseteq V_{i}$ and $f^{-n}(K^{i,j}) \subseteq V_{j}$, for every $i,j \in \{1,\dots,m\}$.
\end{lema}

\begin{proof}Let $u \in \mathcal{U}$ such that for every $i \neq j$, $(y_{i},y_{j}) \notin u$ and let $u'$ a symmetric element of $\mathcal{U}$ such that $u'^{2} \subseteq u$. Since $f$ is uniformly continuous, then there exists $v \in \mathcal{U}$ such that $v$ is closed in $X \times X$, $v \subseteq u'$ and for every $(a,b) \in v$, $(f(a),f(b)) \in u'$. Since $f$ is generalized parabolic, then there exists $n_{0}$ such that for every $n \geqslant n_{0}$, $f^{n}(K) \subseteq \bigcup_{i = 1}^{m} \mathcal{B}(y_{i},v)$. Let $K^{i} = \{k \in K: f^{n_{0}}(k) \in \mathcal{B}(y_{i},v)\}$. If $k \in K^{i}$, then $(y_{i},f^{n_{0}}(k)) \in v$, which implies that $(f(y_{i}),f^{n_{0}+1}(k)) = (y_{i},f^{n_{0}+1}(k)) \in u'$. If there exists $j \neq i$ such that  $(y_{j},f^{n_{0}+1}(k)) \in u'$, then $(y_{j},y_{i}) \in u'^{2} \subseteq u$, which is a contradiction. Then, for every $j \neq i$, $(y_{j},f^{n_{0}+1}(k)) \notin u'$, which implies that $(y_{j},f^{n_{0}+1}(k)) \notin v$. Since $f^{n_{0}+1}(k) \in \bigcup_{i = 1}^{m} \mathcal{B}(y_{i},v)$, then $(y_{i},f^{n_{0}+1}(k)) \in v$. Inductively, for every $n \geqslant n_{0}$, $(y_{i},f^{n}(k)) \in v$. Let $k' \in \overline{K}^{i}$. Since, for every $k \in  K^{i}$, $(y_{i},f^{n_{0}}(k)) \in v$ and $v$ is closed in $X\times X$, then $(y_{i},f^{n_{0}}(k')) \in v$, which implies that $k' \in K^{i}$. So $K^{i}$ is closed, and hence, it is compact.

Let $V_{1},\dots,V_{m}$ be neighborhoods of $y_{1},\dots,y_{m}$, respectively. Then there exists $w \in \mathcal{U}$ such that $w \subseteq v$ and for every $i \in \{1,\dots,m\}$, $\mathcal{B}(y_{i},w) \subseteq V_{i}$. Since $f$ has generalized parabolic dynamics, then there exists $n_{1} \in \N$ such that for every $n > \max\{n_{0},n_{1}\}$, $f^{n}(K) \subseteq \bigcup_{i = 1}^{m} \mathcal{B}(y_{i},w)$. Since $f^{n}(K^{i}) \cap \mathcal{B}(y_{j},w) = \emptyset$, for $i \neq j$, then $f^{n}(K^{i}) \subseteq \mathcal{B}(y_{i},w) \subseteq V_{i}$.

For each $i \in \{1,\dots,m\}$, we can decompose the compact set $K^{i}$ in compact sets $K^{i,1},\dots,K^{i,m}$ such that for every $V_{1},\dots,V_{m}$ neighborhoods of $y_{1},\dots,y_{m}$, respectively, there exists $n_{2,i} \in \N$ such that for every $n \in \N$, with $n > n_{2,i}$, $f^{-n}(K^{i,j}) \subseteq V_{j}$, for every $j \in \{1,\dots,m\}$.

Then the collection $\{K^{i,j}: i,j \in \{1,\dots,m\}\}$ satisfies the properties that we want.
\end{proof}

\begin{teo}\label{parabolicshavelinearentropy}If $(X,\mathcal{U})$ is a Hausdorff compact uniform space and $f: X \rightarrow X$ has generalized parabolic dynamics with finite parabolic set, then $o(f)=[n]$.
\end{teo}

\begin{proof}Note that, by the \textbf{Proposition \ref{alphalimitentropy}}, $o(f) \geqslant [n]$. Let $F = \{y_{1},\dots,y_{m}\}$ be the parabolic set of $f$. Suppose that each element of $F$ is fixed by $f$. Now, given $u \in \mathcal{U}$ such that for every $k \neq l$, $\mathcal{B}(y_k,u) \cap \mathcal{B}(y_l,u) = \emptyset$, we are going to construct a $(n,u)$-generator. Let $w \in \mathcal{U}$ such that $w \subseteq u$ and if $(x,y) \in w$, then $(f(x),f(y)) \in u$ (it exists since $f$ is uniformly continuous) and let $V_{k}$ be an open neighborhood of $y_{k}$ that is $w$-small. Let $V = \bigcup_{k = 1}^{m} V_{k}$ and $W=X-V$ (note that $W$ is compact). 

By the last lemma, there exists compact sets $W^{1,1},\dots,W^{m,m}$, some of these possibly empty, such that $W = \bigcup_{k,l = 1}^{m} W^{k,l}$ and there exists $n_{0} \in \N$ such that for every $n \in \N$, with $|n| > n_{0}$, $f^{n}(W^{k,l}) \subseteq V_{k}$ and $f^{-n}(W^{k,l}) \subseteq V_{l}$, for every $k,l \in \{1,\dots,m\}$. Since the maps $f^{-n_{0}},\dots, f^{n_0}$ are uniformly continuous, there exists $v\in \mathcal{U}$ such that if $(x,y) \in v$ then, for all $-n_{0}\leqslant n \leqslant n_0$, $(f^{n}(x),f^{n}(y)) \in w$. We take a cover of $W^{k,l}$ by $W^{k,l}_{1},\dots, W^{k,l}_{p_{k,l}}$ $v$-small sets and we get that, for every $|j| \leqslant n_{0}$, $f^{j}(W^{k,l}_{i})$ is $u$-small. Then, if $W^{k,l}_{i} \neq \emptyset$, we pick a point $x^{k,l}_i \in W^{k,l}_{i}$ and we get that $W^{k,l}_{i}\subseteq \mathcal{B}(x^{k,l}_{i},n,u)$ for all $n \leqslant n_0$. For every $|j| \geqslant n_{0}$, $f^{j}(W^{k,l}_{i}) \subseteq V_{k}$ and $f^{-j}(W^{k,l}_{i}) \subseteq V_{l}$, which are $u$-small. So, if $n > n_{0}$, and $n_{0} < j \leqslant n$, then, for every $y \in W^{k,l}_{i}$, $(f^{j}(y),f^{j}(x^{k,l}_{i})) \in u$ and we also get that $W^{k,l}_i\subseteq \mathcal{B}(x^{k,l}_i,n,u)$.

Now, we consider the sets $\{f^{j}(W^{k,l}_i): 1\leqslant i\leqslant p_{k,l}\text{ and } -n\leqslant j \leqslant n\}$. Since, for every $t \in \N$, $f^{j+t}(W^{k,l}_{i})$ is a $u$-small set, then $f^{j}(W^{k,l}_i)\subseteq \mathcal{B}(f^j(x^{k,l}_i),n,u)$, and we have a cover for $$Z = \bigcup \{f^{j}(W^{k,l}_i): 1\leqslant i\leqslant p_{k,l}, \ 1 \leqslant k \leqslant m, \ 1 \leqslant l \leqslant m, \ -n\leqslant j \leqslant n\}$$ by at most $(2n+1)p$ $u$-dynamical neighborhoods, where $p = \sum_{k,l = 1}^{m} p_{k,l}$.

Finally, observe that, for each $s \in \{1,\dots,m\}$, $V_{s}' =V_{s}-Z$ is covered by $\mathcal{B}(y_s,n,u)$, since, $V_{s}$ is $w$-small and for every $i \in \{0,\dots,n-1\}$, $f^{i}(V_{s}') \subseteq V_{s}$. In a fact, if $x \in V_{s}'$, then $(y_{s},f(x)) \in u$, which implies that, for every $s' \neq s$, $(y_{s'},f(x)) \notin u$ and then $f(x) \in V_{s}$ or $f(x) \in W$ (this one is not possible by the definition of $V_{s}'$). So $V' =V-Z$ is covered by $\mathcal{B}(y_1,n,u),\dots,\mathcal{B}(y_m,n,u)$.

Taking all the elements of the previous construction, we end up with a cover of $X$ that have at most $(2n + 1)p+m = 2np+p+m$ $u$-dynamical neighborhoods, which proves that

$$g_{f,u}(n)\leqslant 2np + p + m \Rightarrow [g_{f,u}(n)]\leqslant [n].$$

Since the only restriction for $u \in \mathcal{U}$ is to be small enough such that  for every $k \neq l$, $\mathcal{B}(y_k,u) \cap \mathcal{B}(y_l,u) = \emptyset$, then we can conclude, by \textbf{Proposition \ref{basenaomudaentropia}}, that $o(f)\leqslant [n]$, as we wanted.

If $f$ does not fix each point of $F$, then there exists $r \in \N$ such that $f^{r}$ fixes all points of $F$ (since $F$ is finite). It is easy to see that $f^{r}$ also has generalized parabolic dynamics with parabolic set $F$. So $o(f^{r}) = [n]$. By \textbf{Proposition \ref{periodichaszeroentropy}}, it follows that $o(f) \leqslant o(f^{r})$, which implies that $o(f) \leqslant [n]$.
\end{proof}

As special cases, we have the following:

\begin{cor}\label{northsouthhaslinearentropy}  Let $X$ be a Hausdorff compact space, $f: X \rightarrow X$ a homeomorphism with parabolic dynamics or north-south dynamics. Then $o(f) = [n]$.
\end{cor}

\begin{obs} Marco showed that $C^{1}$ north-south dynamics on the interval $[0,1]$ have polynomial entropy equals to $1$ (Proposition 5 of \cite{Mar}), which agrees with our last corollary.
\end{obs}

For generalized parabolic dynamics with infinite parabolic set, the entropy depends on the entropy of the non-wandering set. However, we at least  have a criterion to distinguish if a homeomorphism has generalized parabolic dynamics:

\begin{prop}\label{nongeneralizedparabolichashigherentropy} Let $X$ be a Hausdorff compact connected space, $f: X \rightarrow X$ a homeomorphism and $F$ a closed invariant set such that $\Omega(f) \subseteq F$ and $X-F$ is connected and locally connected. If $f$ does not have generalized parabolic dynamics with parabolic set $F$, then there exist at least two compact mutually singular sets.
\end{prop}

\begin{proof}Suppose that $f$ is not generalized parabolic with parabolic set $F$. Then there exists a compact connected subset $K \subseteq X - F$ such that $\{f^{n}(K)\}_{n \in \N}$ does not converge uniformly to a subset of $F$ or $\{f^{-n}(K)\}_{n \in \N}$ does not converge uniformly to a subset of $F$. Since $X-F$ is connected and locally connected, there exists a connected compact set $K' \subseteq X-F$  such that $K \subseteq K'$ (Proposition 1.8 of \cite{So} gives a proof for this fact). Then $\{f^{n}(K')\}_{n \in \N}$ does not converge uniformly to a subset of $F$ or $\{f^{-n}(K')\}_{n \in \N}$ does not converge uniformly to a subset of $F$. Suppose the first possibility. Then there exists an open neighborhood $U \subseteq X - K'$ of $F$ such that, for every $n \in \N$, there exists $n' > n$ such that $f^{n'}(K') \nsubseteq U$. Let $k \in K'$ and $n_{0} \in \N$. Since $\Omega(f) \subseteq F \subseteq U$, then there exists $n_{1} > n_{0}$ such that for every $n > n_{1}$, $f^{n}(k) \in U$. But there exists $n_{2} > n_{1}$ such that  $f^{n_{2}}(K') \nsubseteq U$, which implies, by the connectedness of $K'$ (and then the connectedness of $f^{n_{2}}(K')$), that $f^{n_{2}}(K')\cap \partial U \neq\emptyset$. So $K'$ and $\partial U$ are mutually singular. For the second possibility, the argument is entirely analogous.    
\end{proof}

\begin{cor}Let $X$ be a Hausdorff compact connected space, $f: X \rightarrow X$ a homeomorphism and $F$ a closed invariant set such that $\Omega(f) \subseteq F$ and $X-F$ is connected and locally connected. If $f$ does not have generalized parabolic dynamics with parabolic set $F$, then $o(f) > [n]$.
\end{cor}

\begin{proof}Immediate from the last proposition and \textbf{Corollary \ref{Entropyviamutuallydisjointsets}}.    
\end{proof}

Now, we give a full characterization of generalized parabolic dynamics with finite parabolic set in terms of generalized entropy. 

\begin{lema}\label{halfparabolic}Let $(X,\mathcal{U})$ be a Hausdorff compact uniform space, $x \in X$ and $C$ a closed subset of $X$ that satisfies the following property: for every neighborhood $U$ of $x$, there exists $n_{0} \in \N$ such that for every $n > n_{0}$, $f^{n}(C) \subseteq U$. Then $o(f,C) = 0$.
\end{lema}

\begin{proof}Let $u \in \mathcal{U}$. Take $w \in \mathcal{U}$ symmetric such that $w^{2} \subseteq u$. There exists $m_{0} \in \N$ such that for every $m > m_{0}$, $f^{m}(C) \subseteq \mathcal{B}(x,w)$. By the equicontinuity of the maps $f^{0},\dots,f^{m_{0}}$, there exists $v \in \mathcal{U}$ such that if $y,z \in X$, with $(y,z) \in v$, then, for every $i \in \{0,\dots,m_{0}\}$, $(f^{i}(y), f^{i}(z)) \in u$. Let $S$ be a finite subset of $C$ such that $C \subseteq \mathcal{B}(S,v)$. Let $c \in C$. Then there exists $s \in S$ such that $(c,s) \in v$, which implies that, for every $i \in \{0,\dots,m_{0}\}$, $(f^{i}(c), f^{i}(s)) \in u$. Also, for every $m > m_{0}$, $f^{m}(c),f^{m}(s) \in \mathcal{B}(x,w)$, which implies that $(f^{m}(c),x)\in w$ and $(x,f^{m}(s))\in w$, which implies that $(f^{m}(c),f^{m}(s)) \in w^{2} \subseteq u$. So, for every $n \in \N$, the set $S$ is a $(n,u)$-generator for $f$, with respect to $C$. Then, for every $n \in \N$, $g_{f,u,C}(n) \leqslant \# S$, which implies that $[g_{f,u,C}(n)] = 0$ and then $o(f,C) = 0$.
\end{proof}

\begin{defi}\label{regularpoints} Let $(X,\mathcal{U})$ be a uniform space, $f: X \rightarrow X$ a homeomorphism and $x \in X$. We say that $x$ is a regular point of $f$ if, for every $u \in \mathcal{U}$, there exists $v \in \mathcal{U}$ such that for every $y \in X$ with $(x,y) \in v$, then, for every $n \in \Z$, $(f^{n}(x),f^{n}(y)) \in u$.
\end{defi}

\begin{lema}(Kinoshita, Lemma 1 of \cite{Ki} for metrizable compact spaces and a singleton as non-wandering set) Let $(X,\mathcal{U})$ be a Hausdorff compact space, $f: X \rightarrow X$ a homeomorphism. Then, the following statements are equivalent:

\begin{enumerate}
    \item For every compact set $K \subseteq X-\Omega(f)$, $f$ and $f^{-1}$ are Lyapunov stable in $K$. 
    \item For every point $x \in X- \Omega(f)$,  $x$ is a regular point of $f$.
\end{enumerate}
\end{lema}

\begin{proof}$1) \Rightarrow 2)$ Let $x \in X- \Omega(f)$ and $u \in \mathcal{U}$. Take an open neighborhood $U$ of $x$ such that $\overline{U} \cap \Omega(f) = \emptyset$. Since $f$ and $f^{-1}$ are Lyapunov stable in $\overline{U}$, then there exists $v \in \mathcal{U}$ such that $\mathcal{B}(x,v) \subseteq \overline{U}$ and, for every $y \in \overline{U}$ with  $(x,y) \in v$ (i.e., $y \in \mathcal{B}(x,v)$), $(f^{n}(x),f^{n}(y)) \in u$, for every $n \in \Z$. Thus $x$ is a regular point of $f$.

$2) \Rightarrow 1)$  Let $K \subseteq X - \Omega(f)$ be a compact set and $u,u' \in \mathcal{U}$ such that $u'$ is a symmetric entourage and $(u')^{2} \subseteq u$. Every point in $K$ is regular, so, for every $k \in K$, there exists $v_{k} \in \mathcal{U}$ such that if $y \in X$ and $(k,y) \in v_{k}$, then, for every $n \in \Z$, $(f^{n}(k),f^{n}(y)) \in u'$. For each $k \in K$, let $v'_{k} \in \mathcal{U}$ such that $v'_{k}$ is open and $v'_{k} \subseteq v_{k}$. We have that $\{\mathcal{B}(k,v'_{k}): k \in K\}$ is an open cover of $K$. Since $K$ is a compact set, there is a finite subcover $\{\mathcal{B}(k_{1},v'_{k_{1}}),\dots, \mathcal{B}(k_{m},v'_{k_{m}})\}$. Let $w \in \mathcal{U}$ such that $w$ is a Lebesgue entourage for the cover $\{X-K,\mathcal{B}(k_{1},v'_{k_{1}}),\dots, \mathcal{B}(k_{m},v'_{k_{m}})\}$. Let $S \subseteq K$ be a $w$-small set. Since $S \nsubseteq X-K$, then there exists $i \in \{1,\dots,m\}$ such that $S \subseteq \mathcal{B}(k_{i},v'_{k_{i}})$. Let $x,y \in S$. Then $(k_{i},x),(k_{i},y) \in v'_{k_{i}}$, which implies that, for every $n \in \Z$, $(f^{n}(k_{i}),f^{n}(x)),(f^{n}(k_{i}),f^{n}(y)) \in u'$ and then $(f^{n}(x),f^{n}(y)) \in (u')^{2} \subseteq u$. Thus $f$ and $f^{-1}$ are Lyapunov stable in $K$. 
\end{proof}

\begin{teo}\label{characterizationofparabolicsviaentropy}Let $(X,\mathcal{U})$ be a Hausdorff compact space, $f: X \rightarrow X$ a homeomorphism such that the non-wandering set $\Omega(f)$ is finite and each point of $\Omega(f)$ is fixed by $f$. Then, the following statements are equivalent:

\begin{enumerate}
    \item $f$ has generalized parabolic dynamics with parabolic set $\Omega(f)$.
    \item The quotient map to the space of orbits $X-\Omega(f) \rightarrow (X-\Omega(f))/\langle f \rangle$ is a covering map.
    \item For every compact set $K \subseteq X-\Omega(f)$, $o(f,K) = 0$ and  $o(f^{-1},K) = 0$.
    \item $o(f|_{X-\Omega(f)}) = 0$ and $o(f^{-1}|_{X-\Omega(f)}) = 0$, where $X-\Omega(f)$ has the subspace uniform structure from $X$.
    \item For every compact set $K \subseteq X-\Omega(f)$, $f$ and $f^{-1}$ are Lyapunov stable in $K$. 
    \item For every point $x \in X- \Omega(f)$,  $x$ is a regular point of $f$.
\end{enumerate}

Moreover, if $X-\Omega(f)$ is connected and locally connected, then there are two more equivalences:

\begin{enumerate}
    \item[7.] $o(f) = [n]$
    \item[8.] There are no non trivial families of mutually singular compact sets.
\end{enumerate}
\end{teo}

\begin{obs}The implications $1) \Leftrightarrow 5 \Leftrightarrow 6)$ are due to Kinoshita (Lemma 1 and Lemma 2 of \cite{Ki}) for the case of $\#\Omega(f) = 1$ and metrizable $X$.
\end{obs}

\begin{proof} $1) \Leftrightarrow 2)$ This is \textbf{Proposition \ref{properlydiscontinuousisparabolic}}.

$3) \Leftrightarrow 4)$ By \textbf{Proposition \ref{entropyandsubspace}}, $o(f|_{X-Y}) = \sup \{o(f,K): K \text{ is a compact subset of } X-Y\}$, which implies that $o(f|_{X-Y}) = 0$ if and only if for every compact set $K \subseteq X-Y$, $o(f,K) = 0$.

$3) \Leftrightarrow 5)$ This is \textbf{Proposition \ref{lyapunovstability}}.

$1) \Rightarrow 3)$ Let $K \subseteq X- \Omega(f)$ be a compact set. By \textbf{Lemma \ref{compactsaredecomposedingeneralizedentropy}} $K = \bigcup_{s \in \Omega(f)} K^{s}$, where $K^{s}$ is a compact set such that for every neighborhood U of $s$, there exists $n_{0}$ such that for every $n > n_{0}$, $f^{n}(K^{s}) \subseteq U$. By \textbf{Lemma \ref{halfparabolic}}, $o(f,K^{s}) = 0$. Since $K = \bigcup_{s \in F} K^{s}$, then, by \textbf{Proposition \ref{subspaces}}, it follows that $o(f,K) = 0$. Analogously, $o(f^{-1},K) = 0$.

$5) \Rightarrow 1)$ Let $K$ be a compact set in $X-\Omega(f)$. Then $f$ and $f^{-1}$ are Lyapunov stable in $K$. Let $u,v \in \mathcal{U}$, with $v$ open in $X\times X$ and $v^{2} \subseteq u$. Then there exists $w \in \mathcal{U}$ such that for every $S \subseteq K$ that is $w$-small, then, for every $n \in \Z$, $f^{n}(S)$ is $v$-small. Since $K$ is compact, there exists $S_{1},\dots, S_{k}$ subsets of $K$ that are $w$-small and $K = \bigcup_{i = 1}^{k}S_{i}$. Let $s_{i} \in S_{i}$. Since $\Omega(f)$ contains all cluster points of $\{f^{n}(s_{i})\}_{n \in \N}$ and $\{f^{-n}(s_{i})\}_{n \in \N}$, then there exists $n_{i} \in \N$ such that for every $|n| > n_{i}$, $f^{n}(s_{i}) \in \mathcal{B}(\Omega(f),v)$ (otherwise these sequences would have cluster points in the compact set $X - \mathcal{B}(\Omega(f),v)$). Since $f^{n}(S_{i})$ is $v$-small and $f^{n}(s_{i}) \in \mathcal{B}(\Omega(f),v)$, then $f^{n}(S_{i}) \subseteq \mathcal{B}(\Omega(f),v^{2})$. Since  $\mathcal{B}(\Omega(f),v^{2}) \subseteq \mathcal{B}(\Omega(f),u)$, then  for every $|n| > n_{i}$, $f^{n}(S_{i}) \subseteq \mathcal{B}(\Omega(f),u)$. If we take $n_{0} = \max\{n_{1},\dots,n_{k}\}$, then, for every $|n| > n_{0}$, $f^{n}(K) \subseteq \mathcal{B}(\Omega(f),u)$. So $f$ has generalized parabolic dynamics.

$5) \Leftrightarrow 6)$ This is the last lemma.

$1) \Rightarrow 7)$ This is \textbf{Theorem \ref{parabolicshavelinearentropy}}.

$7) \Rightarrow 8)$ \textbf{Corollary \ref{Entropyviamutuallydisjointsets}} gives the contrapositive of this.

$8) \Rightarrow 1)$ A special case of \textbf{Proposition \ref{nongeneralizedparabolichashigherentropy}} gives the contrapositive of this.
\end{proof}

The following question leads us to consider the necessity of the theorem's hypotheses:

\begin{quest}Let $f: X \rightarrow X$ be a homeomorphism of a compact Hausdorff space such that the non-wandering set is a finite set of fixed points. Without assuming the connectedness and local connectedness hypothesis for $X-\Omega(f)$, is it possible that $o(f) = [n]$ but $f$ is not generalized parabolic?
\end{quest}

In \cite{CP2}, it is shown that, if $X$ is the $2$-sphere and $\#\Omega(f) = 1$, then there are homeomorphisms whose generalized entropy is between $[n^{2}]$ and $\sup \mathbb{P}$. Also, in \cite{HL} it is shown that there is a gap for homeomorphisms of the $2$-sphere: there is no homeomorphism with $\#\Omega(f) = 1$ such that the polynomial entropy is in the open interval $(1,2)$. The following question is derived from a question given in \cite{CP}:

\begin{quest}Is there any homeomorphism $f: X \rightarrow X$, with $X$ a Hausdorff compact space, such that $\Omega(f)$ is finite, $[n] < o(f)$ and $ o(f) \ngeqslant [n^{2}]$?
\end{quest}

\subsection{Non-wandering set}

If $f: X\rightarrow X$ is a homeomorphism, then some information about the entropy is lost when we decompose our space as $\Omega(f)$ and $X-\Omega(f)$, as the parabolic dynamics shows us.
If $f$ has parabolic dynamics, with $\Omega(f) =  \{p\}$, then, by the last theorem, $o(f|_{X-\{p\}}) = 0$ and $o(f|_{\{p\}}) = 0$. But $o(f) = [n]$. However, this information about the entropy is contained in every neighborhood of $p$:

\begin{prop}Let $(X,\mathcal{U})$ be a Hausdorff compact uniform space and $f: X \rightarrow X$ a homeomorphism with parabolic dynamics and parabolic point $p$. Then, for every open neighborhood $U$ of $p$, $o(f, \overline{U}) = [n]$.
\end{prop}

\begin{proof}Consider the compact set $K = X - U$. By \textbf{Lemma \ref{halfparabolic}}, $o(f,K) = 0$. Since $K \cup \overline{U} = X$, it follows by \textbf{Proposition \ref{subspaces}} that $[n] = o(f) = \sup \{o(f,K), o(f,\overline{U})\} = \sup\{0,o(f,\overline{U})\} = o(f,\overline{U})$.
\end{proof}

\section{Dynamics on compact surfaces}\label{5}

While generalized parabolic dynamics are tightly constrained on compact surfaces — often forcing the space to be a sphere — systems with a finite non-wandering set can still exhibit more complex behavior. In particular, homeomorphisms with a single non-wandering point but non-parabolic dynamics may have generalized entropy that is strictly greater than linear. In this section, we construct such examples on compact surfaces.

\subsection{Brouwer's example}\label{Brouwer}

Proposition 2.3 of \cite{BCGGS} says that a homeomorphism $f: X \rightarrow X$ of a Hausdorff compact space has north-south dynamics with repulsor point $x_{0}$ and attractor point $x_{1}$ if and only if for every point $x \in X-\{x_{0},x_{1}\}$, $\{f^{n}(x)\}_{n \in \N}$ converges to $x_{1}$ and $\{f^{-n}(x)\}_{n \in \N}$ converges to $x_{0}$. 

Maybe we could expect that a homeomorphism $f: X \rightarrow X$ of a Hausdorff compact space has parabolic dynamics with parabolic point $x_{0}$ if and only if for every point $x \in X-\{x_{0}\}$, the sequences $\{f^{n}(x)\}_{n \in \N}$ and $\{f^{-n}(x)\}_{n \in \N}$ converge to $x_{0}$. It is easy to see that if $f$ has parabolic dynamics, then we get the convergence of both sequences. The next example shows that the converse is not true. 
 
Consider the curves $\mathcal{C}_c$, defined by the equation $x = \frac{1}{y(y - 1)} + c$, where $c$ is a real constant, which are pairwise disjoint and fill the strip $0 < y < 1$ in $\mathbb{R}^2$. Let $f : \mathbb{R}^2 \to \mathbb{R}^2 $ be the map defined as:  
$$f(x, y) = \begin{cases} 
(x + 1, y), & \text{if } y \geq 1, \\
(x - 1, y), & \text{if } y \leq 0,
\end{cases}$$  
and by $f(x, y) = (x', y') $ if $ 0 < y < 1 $ and $ (x, y) \in \mathcal{C}_c $, such that $y > y'$ and the arc length of $ \mathcal{C}_c $ from $ (x, y)$ to $ (x', y') $ is equals $1$, \cite{Guil}. We can compactify $\mathbb{R}^2$ to the sphere $S^2$ and extend $f$ to the homeomorphism $\overline{f}: S^2 \to S^2$ whose wandering set is only the fixed point at infinity (let's name it $\infty$).


\begin{figure}[H]
\centering
\begin{subfigure}{.5\textwidth}
  \centering
      \resizebox{6.7cm}{!}{%
      
\tikzset{every picture/.style={line width=0.75pt}} 
     \begin{tikzpicture}[>=Latex, xscale=1,yscale=3]

  \draw[<->] (-8.2,0) -- (8.2,0) node[right, font=\Huge] {$x$};
  \draw[<->] (0,-1.5) -- (0,2.5) node[above,font=\Huge] {$y$};

  \begin{scope}[very thick,decoration={
    markings,
    mark=at position 0.5 with {\arrow[scale=2]{>}},
    mark=at position 0.25 with {\arrow[scale=2]{>}},
    mark=at position 0.75 with {\arrow[scale=2]{>}},
    },
    ] 
    \draw[mycolor2, postaction={decorate}] (-7.2,1)--(7.2,1);
    
\end{scope}

\begin{scope}[very thick,decoration={
    markings,
    mark=at position 0.5 with {\arrow[scale=2]{<}},
    mark=at position 0.25 with {\arrow[scale=2]{<}},
    mark=at position 0.75 with {\arrow[scale=2]{<}},
    },
    ] 
    \draw[mycolor2, postaction={decorate}] (-7.2,0)--(7.2,0);
    
\end{scope}

\begin{scope}[very thick,decoration={
    markings,
    mark=at position 0.5 with {\arrow[scale=2]{>}},
    mark=at position 0.25 with {\arrow[scale=2]{>}},
    mark=at position 0.75 with {\arrow[scale=2]{>}},
    },
    ] 
    \draw[mycolor1, postaction={decorate}] (-7.2,1.33)--(7.2,1.33);
    
\end{scope}
\begin{scope}[very thick,decoration={
    markings,
    mark=at position 0.5 with {\arrow[scale=2]{>}},
    mark=at position 0.25 with {\arrow[scale=2]{>}},
    mark=at position 0.75 with {\arrow[scale=2]{>}},
    },
    ] 
    \draw[mycolor1,postaction={decorate}] (-7.2,1.66)--(7.2,1.66);
    
\end{scope}

\begin{scope}[very thick,decoration={
    markings,
    mark=at position 0.5 with {\arrow[scale=2]{>}},
    mark=at position 0.25 with {\arrow[scale=2]{>}},
    mark=at position 0.75 with {\arrow[scale=2]{>}},
    },
    ] 
    \draw[mycolor1, postaction={decorate}] (-7.2,1.99)--(7.2,1.99);
    
\end{scope}

\begin{scope}[very thick,decoration={
    markings,
    mark=at position 0.5 with {\arrow[scale=2]{<}},
    mark=at position 0.25 with {\arrow[scale=2]{<}},
    mark=at position 0.75 with {\arrow[scale=2]{<}},
    },
    ] 
    \draw[mycolor1, postaction={decorate}] (-7.2,-0.33)--(7.2,-0.33);
    
\end{scope}

\begin{scope}[very thick,decoration={
    markings,
    mark=at position 0.5 with {\arrow[scale=2]{<}},
    mark=at position 0.25 with {\arrow[scale=2]{<}},
    mark=at position 0.75 with {\arrow[scale=2]{<}},
    },
    ] 
    \draw[mycolor1, postaction={decorate}] (-7.2,-0.66)--(7.2,-0.66);
    
\end{scope}

\begin{scope}[very thick,decoration={
    markings,
    mark=at position 0.5 with {\arrow[scale=2]{<}},
    mark=at position 0.25 with {\arrow[scale=2]{<}},
    mark=at position 0.75 with {\arrow[scale=2]{<}},
    },
    ] 
    \draw[mycolor1, postaction={decorate}] (-7.2,-0.99)--(7.2,-0.99);
    
\end{scope}

  \begin{scope}[very thick,decoration={
    markings,
    mark=at position 0.55 with {\arrow[scale=2]{<}},
    },
    ] 
    \draw[postaction={decorate}, domain=0.1:0.9] plot ({1/(\x ) * 1/(\x-1) +4},{\x});
    
\end{scope}

\begin{scope}[very thick,decoration={
    markings,
    mark=at position 0.55 with {\arrow[scale=2]{<}},
    },
    ] 
    \draw[postaction={decorate}, domain=0.076:0.924] plot ({1/(\x ) * 1/(\x-1) +7},{\x});
    
\end{scope}

  \begin{scope}[very thick,decoration={
    markings,
    mark=at position 0.55 with {\arrow[scale=2]{<}},
    },
    ] 
    \draw[postaction={decorate}, domain=0.062:0.938] plot ({1/(\x ) * 1/(\x-1) +10},{\x});
    
\end{scope}

\end{tikzpicture}

}

\end{subfigure}%
\begin{subfigure}{.5\textwidth}
  \centering
\resizebox{5.2cm}{!}
{\tikzset{every picture/.style={line width=0.75pt}} 
\def\centerarc[#1] (#2) (#3:#4:#5)
{ \draw[#1] (#2) ++(#3:#5) arc (#3:#4:#5);
}
    \begin{tikzpicture}[x=0.75pt,y=0.75pt,yscale=-1,xscale=1]
\draw[darkgray] (150.73,135.73) circle (105.73);
\fill[color=mycolor2] (119.47,236.75) circle (1pt);  
\fill[color=mycolor2] (179.47,236.75) circle (1pt);  
\draw[mycolor1]    (51.47,169.75) .. controls (64.47,104.75) and (102.47,52.75) .. (150.73,31) ;
\draw [shift={(87.87,84.16)}, rotate = 124.82] [fill={rgb, 255:red, 74; green, 144; blue, 226 }  ][line width=0.08]  [draw opacity=0] (8.93,-4.29) -- (0,0) -- (8.93,4.29) -- cycle    ;
\draw [color={rgb, 255:red, 208; green, 2; blue, 27 }  ,draw opacity=1 ]   (119.47,236.75) .. controls (87.47,182.75) and (95.47,72.75) .. (150.73,31) ;
\draw [shift={(103.19,123.37)}, rotate = 97.89] [fill={rgb, 255:red, 208; green, 2; blue, 27 }  ,fill opacity=1 ][line width=0.08]  [draw opacity=0] (8.93,-4.29) -- (0,0) -- (8.93,4.29) -- cycle    ;
\draw [color={rgb, 255:red, 208; green, 2; blue, 27 }  ,draw opacity=1 ]   (150.73,31) .. controls (195.47,63.75) and (208.47,166.75) .. (179.47,236.75) ;
\draw [shift={(194.23,133.98)}, rotate = 264.99] [fill={rgb, 255:red, 208; green, 2; blue, 27 }  ,fill opacity=1 ][line width=0.08]  [draw opacity=0] (8.93,-4.29) -- (0,0) -- (8.93,4.29) -- cycle    ;
\draw[mycolor1]    (79.47,211.75) .. controls (68.47,149.75) and (93.47,64.75) .. (150.73,31) ;
\draw [shift={(91.37,105.83)}, rotate = 110.66] [fill={rgb, 255:red, 74; green, 144; blue,  226}  ][line width=0.08]  [draw opacity=0] (8.93,-4.29) -- (0,0) -- (8.93,4.29) -- cycle    ;
\draw[mycolor1]    (150.73,31) .. controls (207.47,63.75) and (230.47,103.75) .. (248.47,169.75) ;
\draw [shift={(218.08,94.04)}, rotate = 237.3] [fill={rgb, 255:red, 74; green, 144; blue, 226 }  ][line width=0.08]  [draw opacity=0] (8.93,-4.29) -- (0,0) -- (8.93,4.29) -- cycle    ;
\draw[mycolor1]    (150.73,31) .. controls (218.47,72.75) and (224.47,191.75) .. (223.47,210.75) ;
\draw [shift={(209.31,117.01)}, rotate = 251.95] [fill={rgb, 255:red, 74; green, 144; blue, 226 }  ][line width=0.08]  [draw opacity=0] (8.93,-4.29) -- (0,0) -- (8.93,4.29) -- cycle    ;
\draw [color={rgb, 255:red, 208; green, 2; blue, 27 }  ,draw opacity=1 ] [dash pattern={on 0.84pt off 2.51pt}]  (119.47,236.75) .. controls (136.47,230.75) and (150.47,82.75) .. (150.73,31) ;
\draw [shift={(142.53,141.67)}, rotate = 276.68] [color={rgb, 255:red, 208; green, 2; blue, 27 }  ,draw opacity=1 ][line width=0.75]    (10.93,-3.29) .. controls (6.95,-1.4) and (3.31,-0.3) .. (0,0) .. controls (3.31,0.3) and (6.95,1.4) .. (10.93,3.29)   ;
\draw [color={rgb, 255:red, 208; green, 2; blue, 27 }  ,draw opacity=1 ] [dash pattern={on 0.84pt off 2.51pt}]  (150.73,31) .. controls (148.47,96.75) and (147.47,216.75) .. (179.47,236.75) ;
\draw [shift={(150.9,131.13)}, rotate = 86.8] [color={rgb, 255:red, 208; green, 2; blue, 27 }  ,draw opacity=1 ][line width=0.75]    (10.93,-3.29) .. controls (6.95,-1.4) and (3.31,-0.3) .. (0,0) .. controls (3.31,0.3) and (6.95,1.4) .. (10.93,3.29)   ;
\draw[mycolor1]  [dash pattern={on 0.84pt off 2.51pt}]  (79.47,211.75) .. controls (119.47,181.75) and (139.47,107.75) .. (150.73,31) ;
\draw[mycolor1] [dash pattern={on 0.84pt off 2.51pt}]  (150.73,31) .. controls (160.47,126.75) and (170.47,173.75) .. (223.47,210.75) ;
\draw[mycolor1]  [dash pattern={on 0.84pt off 2.51pt}]  (51.47,169.75) .. controls (91.47,139.75) and (133.47,88.75) .. (150.73,31) ;
\draw[mycolor1]  [dash pattern={on 0.84pt off 2.51pt}]  (150.73,31) .. controls (177.47,103.75) and (198.47,139.75) .. (248.47,169.75) ;
\draw     ;
\draw  [color={rgb, 255:red, 0; green, 0; blue, 0 }  ,draw opacity=1 ] (150.73,31) .. controls (150.73,31) and (90.47,222.75) .. (146.47,222.75) .. controls (202.47,222.75) and (150.73,31) .. (150.73,31) -- cycle ;
\draw  [color={rgb, 255:red, 0; green, 0; blue, 0 }  ,draw opacity=1 ] (150.73,31) .. controls (150.73,31) and (109.47,190.75) .. (145.47,190.75) .. controls (181.47,190.75) and (150.73,31) .. (150.73,31) -- cycle ;
\draw  [color={rgb, 255:red, 0; green, 0; blue, 0 }  ,draw opacity=1 ] (150.73,31) .. controls (150.73,31) and (121.47,157.75) .. (147.47,157.75) .. controls (173.47,157.75) and (150.73,31) .. (150.73,31) -- cycle ;


\draw[draw=black, fill=black, fill opacity=1] 
(143.09,222.73) -- (149.86,219.01) -- (149.82,226.53) -- cycle;

\draw  [color={rgb, 255:red, 0; green, 0; blue, 0 }  ,draw opacity=1 ][fill={rgb, 255:red, 0; green, 0; blue, 0 }  ,fill opacity=1 ] (142.09,190.73) -- (148.86,187.01) -- (148.82,194.53) -- cycle ;
\draw  [color={rgb, 255:red, 0; green, 0; blue, 0 }  ,draw opacity=1 ][fill={rgb, 255:red, 0; green, 0; blue, 0 }  ,fill opacity=1 ] (144.09,157.73) -- (150.86,154.01) -- (150.82,161.53) -- cycle ;


\fill[color=black] (150.73,30) circle (2pt);  
\end{tikzpicture}
}
\end{subfigure}
\caption{Representation of the maps $f$ and $\overline{f}$.}
\label{fig:test}
\end{figure}
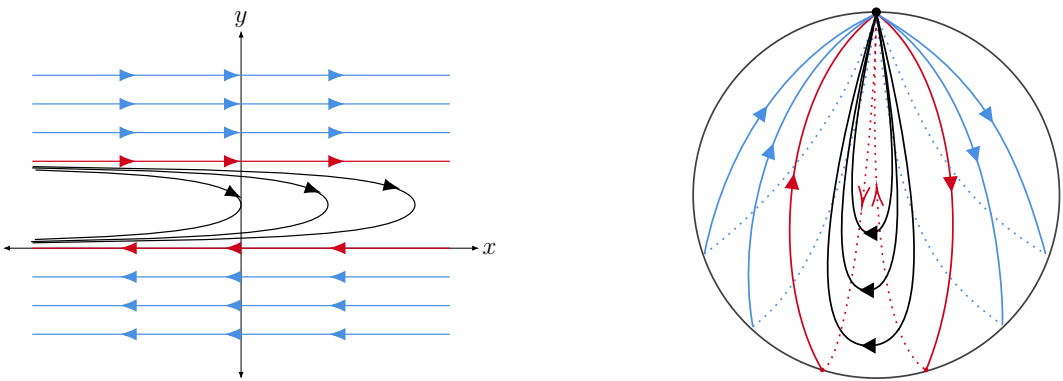



Initially, Brouwer believed that any orientation preserving homeomorphism of $\mathbb{R}^2$ without fixed points was always conjugated to the translation, \cite{bro9}. However, he quickly corrected this assumption, \cite{bro13}, \cite{bro14}, \cite{bro16}, and provided the example above, demonstrating that this is not the case.

By the construction of the map $f$, for every $x \in S^{2}-\{x_{0}\}$, the sequences $\{f^{n}(x)\}_{n \in \N}$ and $\{f^{-n}(x)\}_{n \in \N}$ converge to $\infty$. On the other hand, consider the square $[0,1]^{2} \subseteq \mathbb{R}^{2}$. For every $n \in \Z$, $f^{n}((0,0)) = (-n,0)$, $f^{n}((0,1)) = (n,1)$ and $f^{n}([0,1]^{2}) \subseteq \mathbb{R} \times [0,1]$. Since the set $f^{n}([0,1]^{2})$ is connected, we get that $(\mathbb{R} \times [0,1])- (\{0\}\times [0,1])$ has two connected components and, for $n \neq 0$, the points $f^{n}((0,0))$ and $f^{n}((0,1))$ are in different connected components of $(\mathbb{R} \times [0,1])- (\{0\}\times [0,1])$, then, for $n \neq 0$, $f^{n}([0,1]^{2})\cap(\{0\}\times [0,1])$, which implies that $\{f^{n}([0,1]^{2})\}_{n > 0} = \{\overline{f}^{n}([0,1]^{2})\}_{n > 0}$ do not converge uniformly to $\infty$.  So $\overline{f}$ does not have parabolic dynamics. See \textbf{Figure \ref{fig:intersectionofcompactsonbrouwerexample}} for a picture of $f([0,1]^{2})$ intersecting $[0,1]^{2}$.

\begin{figure}[H]
    \centering
    \includegraphics[width=0.5\linewidth]{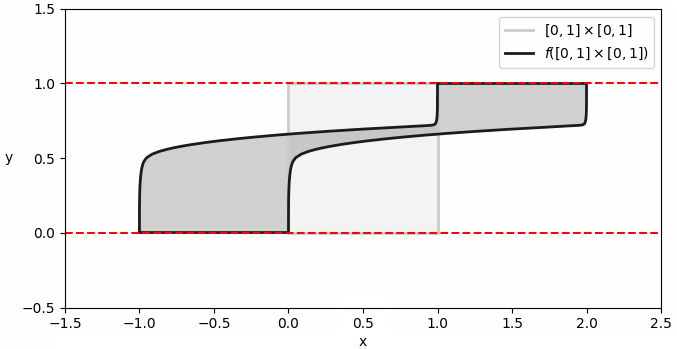}
    \caption{The square $[0,1]^{2}$ and its image $f([0,1]^{2})$.}
    \label{fig:intersectionofcompactsonbrouwerexample}
\end{figure}

By Example 2.5 and Theorem 2 of \cite{CP2}, $o(f) = [n^{2}]$.

\subsection{Surfaces with higher genus}

Theorem A of \cite{HL} and Theorem 6 of \cite{Kau} say that there are strong topological restrictions for the existence of generalized parabolic dynamics in compact manifolds. In particular, there are no parabolic dynamics on compact surfaces without boundary that are different from the sphere. The next examples show that the same does not hold for homeomorphisms of compact surfaces without boundary, where we ask only that the non-wandering set is a singleton (which we saw in the last section that is a weaker condition than to have parabolic dynamics).

There is a construction of a Brouwer homeomorphism on $S^{2}$ given in \cite{CP}, which is a generalization of a construction of \cite{HR}, which we will need next. We will realize this construction inside the hyperbolic plane $\mathbb{H}^{2}$ and take advantage of its structure of semi-planes and its full visualization in the Poincaré disk model, but this is equivalent to the construction given in \cite{CP}. First we need some definitions: a finite closed polygon is the finite intersection of closed semi-planes and we say that this polygon is ideal if any two geodesics that bound the given semi-planes do not intersect in $\mathbb{H}^{2}$ (they can intersect but only in the ideal boundary $\partial\mathbb{H}^{2}$). The construction is the following: 

Take $L \in \N$, with $L \geqslant 2$. Consider the ideal closed polygon $\Gamma$ with $L$ sides in $\mathbb{H}^{2}$ (since we don't want to extend maps to $\partial \mathbb{H}^{2}$, it doesn't matter if this polygon has finite area or not). The sides of $\Gamma$ are geodesics $\gamma_{1},\dots,\gamma_{L}$, which are pairwise disjoint, and $\mathbb{H}^{2} - \Gamma$ is a set of $L$ pairwise disjoint closed semi-planes $\Gamma_{1},\dots,\Gamma_{L}$, with $\Gamma_{i}$ whose boundary is the geodesic $\gamma_{i}$. Let $\mathbb{L}$ be the set of of elements $[a(n)]$ of $\mathbb{O}$ that satisfies the linearly invariant property (i.e., there exist $m \geqslant 2$ such that $[a(n)] = [a(mn)]$) and let $\overline{\mathbb{L}} = \{\sup{\Gamma}: \Gamma \subseteq \mathbb{L}, \# \Gamma \leqslant \aleph_{0}\}$. Then, for every $o \in \overline{\mathbb{L}}$, satisfying $[n^{2}] \leqslant o \leqslant \sup \mathbb{P}$, there exists a homeomorphism $B: \mathbb{H}^{2} \rightarrow \mathbb{H}^{2}$ satisfying:

\begin{enumerate}
    \item $\Gamma$ and each $\Gamma_{i}$ are invariant sets.
    \item For each $i \in \{1,\dots,L\}$, the map $B|_{\Gamma_{i}}$ is a translation, and all directions of translations are oriented counter-clockwise.
    \item If $\hat{B}: \hat{\mathbb{H}}^{2} \rightarrow \hat{\mathbb{H}}^{2}$ is the extension of $B$ to the one-point compactification of $\mathbb{H}^{2}$, then $o(\hat{B}) = o$.
    \item $\Omega(\hat{B}) = \{\infty\}$, where $\{\infty\} =\hat{\mathbb{H}}^{2} - \mathbb{H}^{2}$.
\end{enumerate}

\vspace{-0.5cm}

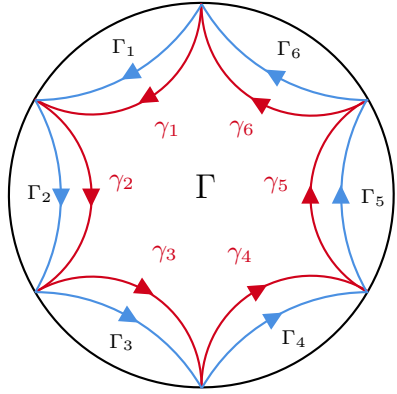
\begin{figure}[H]
    \centering
    \resizebox{9.25cm}{!}{

\tikzset{every picture/.style={line width=0.75pt}} 

\begin{tikzpicture}[x=0.75pt,y=0.75pt,yscale=-1,xscale=1]
    \path[use as bounding box] (150,10) rectangle (460,220);

    \draw   (226.22,164.49) .. controls (201.31,121.87) and (215.66,67.13) .. (258.27,42.22) .. controls (300.89,17.31) and (355.63,31.66) .. (380.54,74.27) .. controls (405.45,116.89) and (391.1,171.63) .. (348.49,196.54) .. controls (305.87,221.46) and (251.13,207.1) .. (226.22,164.49) -- cycle ;

    \draw  [draw opacity=0] (303.28,30.06) .. controls (303.44,48.04) and (294.28,65.52) .. (277.73,75.07) .. controls (261.18,84.62) and (241.47,83.82) .. (225.97,74.69) -- (251.74,30.06) -- cycle ; 
    \draw  [color={rgb, 255:red, 208; green, 2; blue, 27 },draw opacity=1 ] (303.28,30.06) .. controls (303.44,48.04) and (294.28,65.52) .. (277.73,75.07) .. controls (261.18,84.62) and (241.47,83.82) .. (225.97,74.69) ;  

\draw  [draw opacity=0] (226.22,75.23) .. controls (241.87,84.08) and (252.42,100.75) .. (252.42,119.86) .. controls (252.42,138.96) and (241.87,155.64) .. (226.22,164.49) -- (200.45,119.86) -- cycle ; \draw  [color={rgb, 255:red, 208; green, 2; blue, 27 }  ,draw opacity=1 ] (226.22,75.23) .. controls (241.87,84.08) and (252.42,100.75) .. (252.42,119.86) .. controls (252.42,138.96) and (241.87,155.64) .. (226.22,164.49) ;  
\draw  [draw opacity=0] (226.22,164.49) .. controls (241.71,155.36) and (261.43,154.55) .. (277.97,164.11) .. controls (294.52,173.66) and (303.68,191.13) .. (303.52,209.12) -- (251.99,209.12) -- cycle ; \draw  [color={rgb, 255:red, 208; green, 2; blue, 27 }  ,draw opacity=1 ] (226.22,164.49) .. controls (241.71,155.36) and (261.43,154.55) .. (277.97,164.11) .. controls (294.52,173.66) and (303.68,191.13) .. (303.52,209.12) ;  
\draw  [draw opacity=0] (303.48,208.7) .. controls (303.32,190.72) and (312.48,173.24) .. (329.03,163.69) .. controls (345.58,154.14) and (365.29,154.94) .. (380.79,164.07) -- (355.02,208.7) -- cycle ; \draw  [color={rgb, 255:red, 208; green, 2; blue, 27 }  ,draw opacity=1 ] (303.48,208.7) .. controls (303.32,190.72) and (312.48,173.24) .. (329.03,163.69) .. controls (345.58,154.14) and (365.29,154.94) .. (380.79,164.07) ;  
\draw  [draw opacity=0] (380.32,164.34) .. controls (364.66,155.49) and (354.11,138.82) .. (354.11,119.71) .. controls (354.11,100.6) and (364.66,83.93) .. (380.32,75.08) -- (406.09,119.71) -- cycle ; \draw  [color={rgb, 255:red, 208; green, 2; blue, 27 }  ,draw opacity=1 ] (380.32,164.34) .. controls (364.66,155.49) and (354.11,138.82) .. (354.11,119.71) .. controls (354.11,100.6) and (364.66,83.93) .. (380.32,75.08) ;  
\draw  [draw opacity=0] (380.79,75.23) .. controls (365.29,84.36) and (345.58,85.16) .. (329.03,75.61) .. controls (312.48,66.05) and (303.32,48.58) .. (303.48,30.6) -- (355.02,30.6) -- cycle ; \draw  [color={rgb, 255:red, 208; green, 2; blue, 27 }  ,draw opacity=1 ] (380.79,75.23) .. controls (365.29,84.36) and (345.58,85.16) .. (329.03,75.61) .. controls (312.48,66.05) and (303.32,48.58) .. (303.48,30.6) ;  
\draw  [draw opacity=0] (380.32,164.42) .. controls (372.74,151.28) and (368.41,136.02) .. (368.41,119.75) .. controls (368.41,103.48) and (372.74,88.22) .. (380.32,75.08) -- (457.69,119.75) -- cycle ; \draw  [color={rgb, 255:red, 74; green, 144; blue, 226 }  ,draw opacity=1 ] (380.32,164.42) .. controls (372.74,151.28) and (368.41,136.02) .. (368.41,119.75) .. controls (368.41,103.48) and (372.74,88.22) .. (380.32,75.08) ;  
\draw  [draw opacity=0] (303.59,30.55) .. controls (296,43.69) and (284.95,55.07) .. (270.86,63.21) .. controls (256.77,71.34) and (241.39,75.22) .. (226.22,75.23) -- (226.22,-14.12) -- cycle ; \draw  [color={rgb, 255:red, 74; green, 144; blue, 226 }  ,draw opacity=1 ] (303.59,30.55) .. controls (296,43.69) and (284.95,55.07) .. (270.86,63.21) .. controls (256.77,71.34) and (241.39,75.22) .. (226.22,75.23) ;  
\draw  [draw opacity=0] (226.22,75.23) .. controls (233.79,88.37) and (238.13,103.63) .. (238.13,119.9) .. controls (238.13,136.17) and (233.79,151.43) .. (226.22,164.57) -- (148.84,119.9) -- cycle ; \draw  [color={rgb, 255:red, 74; green, 144; blue, 226 }  ,draw opacity=1 ] (226.22,75.23) .. controls (233.79,88.37) and (238.13,103.63) .. (238.13,119.9) .. controls (238.13,136.17) and (233.79,151.43) .. (226.22,164.57) ;  
\draw  [draw opacity=0] (226.22,164.49) .. controls (241.39,164.5) and (256.77,168.37) .. (270.86,176.51) .. controls (284.95,184.64) and (296,196.03) .. (303.59,209.16) -- (226.22,253.83) -- cycle ; \draw  [color={rgb, 255:red, 74; green, 144; blue, 226 }  ,draw opacity=1 ] (226.22,164.49) .. controls (241.39,164.5) and (256.77,168.37) .. (270.86,176.51) .. controls (284.95,184.64) and (296,196.03) .. (303.59,209.16) ;  
\draw  [draw opacity=0] (303.59,209.16) .. controls (311.19,196.03) and (322.23,184.64) .. (336.32,176.51) .. controls (350.41,168.37) and (365.79,164.5) .. (380.97,164.49) -- (380.97,253.83) -- cycle ; \draw  [color={rgb, 255:red, 74; green, 144; blue, 226 }  ,draw opacity=1 ] (303.59,209.16) .. controls (311.19,196.03) and (322.23,184.64) .. (336.32,176.51) .. controls (350.41,168.37) and (365.79,164.5) .. (380.97,164.49) ;  
\draw  [draw opacity=0] (380.97,75.23) .. controls (365.79,75.22) and (350.41,71.34) .. (336.32,63.21) .. controls (322.23,55.07) and (311.19,43.69) .. (303.59,30.55) -- (380.97,-14.12) -- cycle ; \draw  [color={rgb, 255:red, 74; green, 144; blue, 226 }  ,draw opacity=1 ] (380.97,75.23) .. controls (365.79,75.22) and (350.41,71.34) .. (336.32,63.21) .. controls (322.23,55.07) and (311.19,43.69) .. (303.59,30.55) ;  
\draw  [color={rgb, 255:red, 74; green, 144; blue, 226 }  ,draw opacity=1 ][fill={rgb, 255:red, 74; green, 144; blue, 226 }  ,fill opacity=1 ] (368.21,115.82) -- (371.61,122.83) -- (364.81,122.83) -- cycle ;
\draw  [color={rgb, 255:red, 74; green, 144; blue, 226 }  ,draw opacity=1 ][fill={rgb, 255:red, 74; green, 144; blue, 226 }  ,fill opacity=1 ] (340.24,174.57) -- (335.87,181.02) -- (332.47,175.13) -- cycle ;
\draw  [color={rgb, 255:red, 74; green, 144; blue, 226 }  ,draw opacity=1 ][fill={rgb, 255:red, 74; green, 144; blue, 226 }  ,fill opacity=1 ] (273.24,178.08) -- (265.47,177.52) -- (268.87,171.63) -- cycle ;
\draw  [color={rgb, 255:red, 74; green, 144; blue, 226 }  ,draw opacity=1 ][fill={rgb, 255:red, 74; green, 144; blue, 226 }  ,fill opacity=1 ] (238.21,123.83) -- (234.81,116.82) -- (241.61,116.82) -- cycle ;
\draw  [color={rgb, 255:red, 74; green, 144; blue, 226 }  ,draw opacity=1 ][fill={rgb, 255:red, 74; green, 144; blue, 226 }  ,fill opacity=1 ] (267.17,65.08) -- (271.54,58.63) -- (274.94,64.52) -- cycle ;
\draw  [color={rgb, 255:red, 74; green, 144; blue, 226 }  ,draw opacity=1 ][fill={rgb, 255:red, 74; green, 144; blue, 226 }  ,fill opacity=1 ] (334.17,61.57) -- (341.94,62.13) -- (338.54,68.02) -- cycle ;
\draw  [color={rgb, 255:red, 208; green, 2; blue, 27 }  ,draw opacity=1 ][fill={rgb, 255:red, 208; green, 2; blue, 27 }  ,fill opacity=1 ] (354.04,115.82) -- (357.44,122.83) -- (350.64,122.83) -- cycle ;
\draw  [color={rgb, 255:red, 208; green, 2; blue, 27 }  ,draw opacity=1 ][fill={rgb, 255:red, 208; green, 2; blue, 27 }  ,fill opacity=1 ] (333.08,161.57) -- (328.7,168.02) -- (325.3,162.13) -- cycle ;
\draw  [color={rgb, 255:red, 208; green, 2; blue, 27 }  ,draw opacity=1 ][fill={rgb, 255:red, 208; green, 2; blue, 27 }  ,fill opacity=1 ] (280.24,165.08) -- (272.47,164.52) -- (275.87,158.63) -- cycle ;
\draw  [color={rgb, 255:red, 208; green, 2; blue, 27 }  ,draw opacity=1 ][fill={rgb, 255:red, 208; green, 2; blue, 27 }  ,fill opacity=1 ] (252.21,123.83) -- (248.81,116.82) -- (255.61,116.82) -- cycle ;
\draw  [color={rgb, 255:red, 208; green, 2; blue, 27 }  ,draw opacity=1 ][fill={rgb, 255:red, 208; green, 2; blue, 27 }  ,fill opacity=1 ] (275.17,76.08) -- (279.54,69.63) -- (282.94,75.52) -- cycle ;
\draw  [color={rgb, 255:red, 208; green, 2; blue, 27 }  ,draw opacity=1 ][fill={rgb, 255:red, 208; green, 2; blue, 27 }  ,fill opacity=1 ] (327.17,74.57) -- (334.94,75.13) -- (331.54,81.02) -- cycle ;

\draw (280.21,83.73) node [anchor=north west][inner sep=0.75pt]  [font=\small,color={rgb, 255:red, 208; green, 2; blue, 27 }  ,opacity=1 ]  {$\gamma _{1}$};
\draw (259,108.38) node [anchor=north west][inner sep=0.75pt]  [font=\small,color={rgb, 255:red, 208; green, 2; blue, 27 }  ,opacity=1 ]  {$\gamma _{2}$};
\draw (279,141.38) node [anchor=north west][inner sep=0.75pt]  [font=\small,color={rgb, 255:red, 208; green, 2; blue, 27 }  ,opacity=1 ]  {$\gamma _{3}$};
\draw (314,142.38) node [anchor=north west][inner sep=0.75pt]  [font=\small,color={rgb, 255:red, 208; green, 2; blue, 27 }  ,opacity=1 ]  {$\gamma _{4}$};
\draw (331,109.38) node [anchor=north west][inner sep=0.75pt]  [font=\small,color={rgb, 255:red, 208; green, 2; blue, 27 }  ,opacity=1 ]  {$\gamma _{5}$};
\draw (315,83.38) node [anchor=north west][inner sep=0.75pt]  [font=\small,color={rgb, 255:red, 208; green, 2; blue, 27 }  ,opacity=1 ]  {$\gamma _{6}$};
\draw (299,108.38) node [anchor=north west][inner sep=0.75pt]  [font=\large]  {$\Gamma $};
\draw (260.27,45.62) node [anchor=north west][inner sep=0.75pt]  [font=\scriptsize]  {$\Gamma _{1}$};
\draw (220.97,113.51) node [anchor=north west][inner sep=0.75pt]  [font=\scriptsize]  {$\Gamma _{2}$};
\draw (259,182.38) node [anchor=north west][inner sep=0.75pt]  [font=\scriptsize]  {$\Gamma _{3}$};
\draw (339.21,180.73) node [anchor=north west][inner sep=0.75pt]  [font=\scriptsize]  {$\Gamma _{4}$};
\draw (376,115.38) node [anchor=north west][inner sep=0.75pt]  [font=\scriptsize]  {$\Gamma _{5}$};
\draw (337,46.38) node [anchor=north west][inner sep=0.75pt]  [font=\scriptsize]  {$\Gamma _{6}$};

\end{tikzpicture}
    }
    \caption{Representation of the map $B$, where $\Gamma$ is an ideal hexagon bounded by the geodesics in red given by $\gamma_{1}, \gamma_{2}, \gamma_{3}, \gamma_{4}, \gamma_{5}$ and $\gamma_{6}$. The blue lines represent the translations given by the map $B$ to the regions $\Gamma_{1}, \Gamma_{2}, \Gamma_{3}, \Gamma_{4}, \Gamma_{5}$ and $\Gamma_{6}$.}
    \label{fig:enter-label}
\end{figure}


Actually, in \cite{CP} they construct this asking $o \in \overline{\mathbb{L}}$, satisfying $[n^{2}] < o \leqslant \sup \mathbb{P}$, since the case $o = [n^{2}]$ is not necessary for the paper. However, the same construction holds for $o = [n^{2}]$.

Let $\overline{\Gamma}$ and $\overline{\Gamma_{i}}$ be the closures of $\Gamma$ and $\Gamma_{i}$ in $\hat{\mathbb{H}}^{2}$, respectively. We have that $\overline{\Gamma}$ and $\overline{\Gamma_{i}}$ are invariant by $\hat{B}$ and $o(\hat{B}|_{\overline{\Gamma}}) = o$, since $\hat{\mathbb{H}}^{2} = \overline{\Gamma} \cup \bigcup_{i = 1}^{L} \overline{\Gamma_{i}}$, for each $i \in \{1,\dots,L\}$, $o(\hat{B}|_{\overline{\Gamma}_{i}}) = [n]$ (because $\hat{B}|_{\overline{\Gamma}_{i}}$ is the extension of a translation to the one-point compactification of $\Gamma_{i}$) and $o(\hat{B}) = o$.

Gluing some polygons, we get a small modification for this construction that will be useful to us:

\begin{lema}Let $L \geqslant 2$ and $o$ an element in $\overline{\mathbb{L}}$, satisfying $[n^{2}] \leqslant  o \leqslant \sup \mathbb{P}$, there exists a homeomorphism $B_{L,o}: \mathbb{H}^{2} \rightarrow \mathbb{H}^{2}$ and an ideal closed polygon $\Gamma$ with $L$ sides, satisfying:

\begin{enumerate}
    \item If $\hat{B}_{L,o}: \hat{\mathbb{H}}^{2} \rightarrow \hat{\mathbb{H}}^{2}$ is the extension of $B_{L,o}$ to the one-point compactification of $\mathbb{H}^{2}$, then $o(\hat{B}_{L,o}) = o(\hat{B}_{L,o}, \overline{\Gamma}) = o$.
    \item $\Omega(\hat{B}_{L,o}) = \{\infty\}$, where $\{\infty\} =\hat{\mathbb{H}}^{2} - \mathbb{H}^{2}$.
    \item $\Gamma$ and each $\Gamma_{i}$ are invariant sets, where $\Gamma_{1},\dots, \Gamma_{L}$ are the closure of the connected components of $\mathbb{H}^{2} - \Gamma$ in $\mathbb{H}^{2}$.
    \item For each $i \in \{1,\dots,L\}$, the map $B_{L,o}|_{\Gamma_{i}}$ is a translation, and all directions of translations are oriented counter-clockwise.
    \item There exists an invariant closed set $C \subseteq \Gamma$ by the map  $B_{L,o}$, with closure $\overline{C}$ in $\hat{\mathbb{H}}^{2}$ such that $\overline{C}$ does not intersect the sides of $\Gamma$ and $o(\hat{B}_{L,o}|_{\overline{C}}) = o$.
\end{enumerate}
    
\end{lema}

\begin{proof}Consider, as described above, three homeomorphisms $B_{i}: \mathbb{H}^{2} \rightarrow \mathbb{H}^{2}$, $i\in \{1,2,3\}$ and finite polygons $\Gamma^{i}$ such that $\Gamma^{i}$ is invariant by $B_{i}$, $o(\hat{B}_{1})\leqslant o$, $o(\hat{B}_{2}) = o$, $o(\hat{B}_{3}) \leqslant o$, $\Omega(\hat{B}_{1}) = \Omega(\hat{B}_{2}) = \Omega((\hat{B_{3}})) = \{\infty\}$, $\Gamma^{1}$ has $2$ sides, $\Gamma^{3}$ has $L$ sides, $B_{i}$ is a translation when restricted to the closure of the connected components of $\mathbb{H}^{2} - \Gamma^{i}$, the sides of $\Gamma^{2}$ are oriented clockwise by $B_{2}$ and the sides of $\Gamma^{1}$ and $\Gamma^{3}$ are oriented counter-clockwise by $B_{1}$ and $B_{3}$, respectively. Let $\Gamma^{i}_{j}$ be the closure of the $j$-th connected components of $\mathbb{H}^{2} - \Gamma^{i}$, counted counter-clockwise

Then we get a new space gluing the second side of $\Gamma^{1}$ to the first side of $\Gamma^{2}$ and the last side of $\Gamma^{2}$ to the first side of $\Gamma^{3}$. Also, if there is a side of some of these polygons that is not glued to another side, then we glue to it its respective $\Gamma^{i}_{j}$ (these are the sets $\Gamma^{1}_{1}, \Gamma^{2}_{2}, \dots,\Gamma^{2}_{L'-1}, \Gamma^{3}_{2}, \dots,\Gamma^{3}_{L}$, where $L'$ is the number of sides of $\Gamma^{2}$). The result must be a space homeomorphic to $\mathbb{H}^{2}$ (so we identify it with $\mathbb{H}^{2}$). Let $\Gamma$ be the polygon given by union of the classes of $\Gamma^{1}, \Gamma^{2}, \Gamma^{3}, \Gamma^{2}_{2}, \dots,\Gamma^{2}_{L'-1}$. It is illustrated by the figure below:

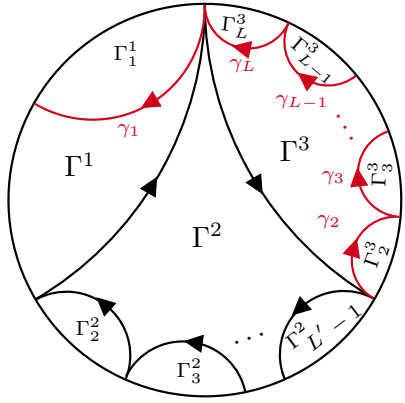
\begin{figure}[H]
    \centering
    \resizebox{9.45cm}{!}{

\tikzset{every picture/.style={line width=0.75pt}} 

\begin{tikzpicture}[x=0.75pt,y=0.75pt,yscale=-1,xscale=1]

\draw   (253.04,156.91) .. controls (228.12,114.29) and (242.47,59.55) .. (285.09,34.64) .. controls (327.71,9.73) and (382.45,24.08) .. (407.36,66.69) .. controls (432.27,109.31) and (417.92,164.05) .. (375.31,188.96) .. controls (332.69,213.88) and (277.95,199.52) .. (253.04,156.91) -- cycle ;
\draw  [draw opacity=0] (407.49,156.91) .. controls (384.85,143.68) and (365.22,124.52) .. (351.14,100.14) .. controls (337.06,75.76) and (330.29,49.17) .. (330.15,22.95) -- (484.83,22.95) -- cycle ; \draw   (407.49,156.91) .. controls (384.85,143.68) and (365.22,124.52) .. (351.14,100.14) .. controls (337.06,75.76) and (330.29,49.17) .. (330.15,22.95) ;  
\draw  [draw opacity=0] (330.26,23.14) .. controls (330.12,49.37) and (323.35,75.95) .. (309.28,100.33) .. controls (295.2,124.72) and (275.56,143.87) .. (252.92,157.1) -- (175.58,23.14) -- cycle ; \draw   (330.26,23.14) .. controls (330.12,49.37) and (323.35,75.95) .. (309.28,100.33) .. controls (295.2,124.72) and (275.56,143.87) .. (252.92,157.1) ;  
\draw  [draw opacity=0] (330.15,23.34) .. controls (330.31,41.32) and (321.15,58.8) .. (304.6,68.35) .. controls (288.06,77.9) and (268.34,77.1) .. (252.85,67.97) -- (278.62,23.34) -- cycle ; \draw  [color={rgb, 255:red, 208; green, 2; blue, 27 }  ,draw opacity=1 ] (330.15,23.34) .. controls (330.31,41.32) and (321.15,58.8) .. (304.6,68.35) .. controls (288.06,77.9) and (268.34,77.1) .. (252.85,67.97) ;  
\draw  [draw opacity=0] (253.09,157.25) .. controls (263.54,151.31) and (277.11,152.46) .. (286.63,161.02) .. controls (296.14,169.59) and (298.7,182.97) .. (293.88,193.98) -- (267.52,182.24) -- cycle ; \draw   (253.09,157.25) .. controls (263.54,151.31) and (277.11,152.46) .. (286.63,161.02) .. controls (296.14,169.59) and (298.7,182.97) .. (293.88,193.98) ;  
\draw  [draw opacity=0] (294.34,192.1) .. controls (299.68,181.34) and (311.57,174.69) .. (324.25,176.48) .. controls (336.92,178.26) and (346.53,187.92) .. (348.69,199.74) -- (320.27,204.75) -- cycle ; \draw   (294.34,192.1) .. controls (299.68,181.34) and (311.57,174.69) .. (324.25,176.48) .. controls (336.92,178.26) and (346.53,187.92) .. (348.69,199.74) ;  
\draw  [draw opacity=0] (366.7,193.64) .. controls (361.89,182.63) and (364.45,169.24) .. (373.96,160.68) .. controls (383.47,152.11) and (397.05,150.97) .. (407.49,156.91) -- (393.07,181.9) -- cycle ; \draw   (366.7,193.64) .. controls (361.89,182.63) and (364.45,169.24) .. (373.96,160.68) .. controls (383.47,152.11) and (397.05,150.97) .. (407.49,156.91) ;  
\draw  [draw opacity=0] (407.19,156.74) .. controls (399.22,152.19) and (395.16,142.62) .. (397.94,133.51) .. controls (400.72,124.4) and (409.44,118.74) .. (418.59,119.43) -- (417.15,139.38) -- cycle ; \draw  [color={rgb, 255:red, 208; green, 2; blue, 27 }  ,draw opacity=1 ] (407.19,156.74) .. controls (399.22,152.19) and (395.16,142.62) .. (397.94,133.51) .. controls (400.72,124.4) and (409.44,118.74) .. (418.59,119.43) ;  
\draw  [draw opacity=0] (417.59,119.33) .. controls (408.48,118.26) and (401,111.04) .. (400,101.57) .. controls (399.01,92.1) and (404.82,83.47) .. (413.51,80.53) -- (419.98,99.47) -- cycle ; \draw  [color={rgb, 255:red, 208; green, 2; blue, 27 }  ,draw opacity=1 ] (417.59,119.33) .. controls (408.48,118.26) and (401,111.04) .. (400,101.57) .. controls (399.01,92.1) and (404.82,83.47) .. (413.51,80.53) ;  
\draw  [draw opacity=0] (368.31,31.06) .. controls (364.48,39.4) and (355.3,44.28) .. (345.98,42.3) .. controls (336.67,40.32) and (330.26,32.12) .. (330.15,22.95) -- (350.16,22.65) -- cycle ; \draw  [color={rgb, 255:red, 208; green, 2; blue, 27 }  ,draw opacity=1 ] (368.31,31.06) .. controls (364.48,39.4) and (355.3,44.28) .. (345.98,42.3) .. controls (336.67,40.32) and (330.26,32.12) .. (330.15,22.95) ;  
\draw  [draw opacity=0] (399.06,55.08) .. controls (391.96,60.89) and (381.56,61.25) .. (374.06,55.39) .. controls (366.56,49.53) and (364.39,39.36) .. (368.31,31.06) -- (386.43,39.56) -- cycle ; \draw  [color={rgb, 255:red, 208; green, 2; blue, 27 }  ,draw opacity=1 ] (399.06,55.08) .. controls (391.96,60.89) and (381.56,61.25) .. (374.06,55.39) .. controls (366.56,49.53) and (364.39,39.36) .. (368.31,31.06) ;  
\draw  [color={rgb, 255:red, 208; green, 2; blue, 27 }  ,draw opacity=1 ][fill={rgb, 255:red, 208; green, 2; blue, 27 }  ,fill opacity=1 ] (399.66,98.1) -- (403.77,104.72) -- (397.01,105.43) -- cycle ;
\draw  [color={rgb, 255:red, 208; green, 2; blue, 27 }  ,draw opacity=1 ][fill={rgb, 255:red, 208; green, 2; blue, 27 }  ,fill opacity=1 ] (373.26,54.43) -- (380.88,56.07) -- (376.7,61.43) -- cycle ;
\draw  [color={rgb, 255:red, 208; green, 2; blue, 27 }  ,draw opacity=1 ][fill={rgb, 255:red, 208; green, 2; blue, 27 }  ,fill opacity=1 ] (343.6,41.86) -- (351.16,39.99) -- (349.75,46.64) -- cycle ;
\draw  [color={rgb, 255:red, 208; green, 2; blue, 27 }  ,draw opacity=1 ][fill={rgb, 255:red, 208; green, 2; blue, 27 }  ,fill opacity=1 ] (399.05,130.24) -- (400.25,137.94) -- (393.75,135.95) -- cycle ;
\draw  [color={rgb, 255:red, 208; green, 2; blue, 27 }  ,draw opacity=1 ][fill={rgb, 255:red, 208; green, 2; blue, 27 }  ,fill opacity=1 ] (301.99,69.34) -- (306.36,62.89) -- (309.76,68.78) -- cycle ;
\draw  [fill={rgb, 255:red, 0; green, 0; blue, 0 }  ,fill opacity=1 ] (321.73,176.2) -- (329.14,173.81) -- (328.2,180.55) -- cycle ;
\draw  [fill={rgb, 255:red, 0; green, 0; blue, 0 }  ,fill opacity=1 ] (370.59,163.61) -- (373.53,156.39) -- (378.08,161.44) -- cycle ;
\draw  [fill={rgb, 255:red, 0; green, 0; blue, 0 }  ,fill opacity=1 ] (281.59,156.8) -- (289.08,158.97) -- (284.53,164.02) -- cycle ;
\draw  [fill={rgb, 255:red, 0; green, 0; blue, 0 }  ,fill opacity=1 ] (355.95,108.19) -- (349.5,103.81) -- (355.39,100.41) -- cycle ;
\draw  [fill={rgb, 255:red, 0; green, 0; blue, 0 }  ,fill opacity=1 ] (307.95,102.11) -- (307.39,109.89) -- (301.5,106.49) -- cycle ;

\draw (396.69,87.65) node [anchor=north west][inner sep=0.75pt]  [color={rgb, 255:red, 208; green, 2; blue, 27 }  ,opacity=1 ,rotate=-240.15]  {$\dots $};
\draw (340.93,175.48) node [anchor=north west][inner sep=0.75pt]  [rotate=-342.22]  {$\dots $};
\draw (264.82,88.64) node [anchor=north west][inner sep=0.75pt]  [font=\normalsize]  {$\Gamma ^{1}$};
\draw (322.82,120.64) node [anchor=north west][inner sep=0.75pt]  [font=\normalsize]  {$\Gamma ^{2}$};
\draw (362.82,81.64) node [anchor=north west][inner sep=0.75pt]  [font=\normalsize]  {$\Gamma ^{3}$};
\draw (288.79,75.83) node [anchor=north west][inner sep=0.75pt]  [font=\scriptsize,color={rgb, 255:red, 208; green, 2; blue, 27 }  ,opacity=1 ]  {$\gamma _{1}$};
\draw (379.79,117.33) node [anchor=north west][inner sep=0.75pt]  [font=\scriptsize,color={rgb, 255:red, 208; green, 2; blue, 27 }  ,opacity=1 ]  {$\gamma _{2}$};
\draw (381.79,97.33) node [anchor=north west][inner sep=0.75pt]  [font=\scriptsize,color={rgb, 255:red, 208; green, 2; blue, 27 }  ,opacity=1 ]  {$\gamma _{3}$};
\draw (360.79,62.33) node [anchor=north west][inner sep=0.75pt]  [font=\scriptsize,color={rgb, 255:red, 208; green, 2; blue, 27 }  ,opacity=1 ]  {$\gamma _{L-1}$};
\draw (339.79,46.33) node [anchor=north west][inner sep=0.75pt]  [font=\scriptsize,color={rgb, 255:red, 208; green, 2; blue, 27 }  ,opacity=1 ]  {$\gamma _{L}$};
\draw (287.09,38.04) node [anchor=north west][inner sep=0.75pt]  [font=\scriptsize]  {$\Gamma _{1}^{1}$};
\draw (269.82,163.64) node [anchor=north west][inner sep=0.75pt]  [font=\scriptsize]  {$\Gamma _{2}^{2}$};
\draw (315.82,182.64) node [anchor=north west][inner sep=0.75pt]  [font=\scriptsize]  {$\Gamma _{3}^{2}$};
\draw (362.88,173.11) node [anchor=north west][inner sep=0.75pt]  [font=\scriptsize,rotate=-322.62]  {$\Gamma _{{\displaystyle L'-1}}^{2}$};
\draw (398.62,141.17) node [anchor=north west][inner sep=0.75pt]  [font=\scriptsize,rotate=-289.14]  {$\Gamma _{2}^{3}$};
\draw (402.53,106.93) node [anchor=north west][inner sep=0.75pt]  [font=\scriptsize,rotate=-272.35]  {$\Gamma _{3}^{3}$};
\draw (374.14,32.99) node [anchor=north west][inner sep=0.75pt]  [font=\scriptsize,rotate=-40.41]  {$\Gamma _{L-1}^{3}$};
\draw (336.96,24.35) node [anchor=north west][inner sep=0.75pt]  [font=\scriptsize,rotate=-6.68]  {$\Gamma _{L}^{3}$};

\end{tikzpicture}
    
    }
    \caption{Gluing of $\Gamma^{1}, \Gamma^{2}, \Gamma^{3},  \Gamma^{1}_{1}, \Gamma^{2}_{2}, \dots,\Gamma^{2}_{L'-1}, \Gamma^{3}_{2}, \dots,\Gamma^{3}_{L}$. In red we mark the sides of the polygon $\Gamma$.}
    \label{fig:enter-label2}
\end{figure}


We do all of these glueings, agreeing with the maps $B_{1}$, $B_{2}$, and $B_{3}$ (it is possible since these maps are translations on the sides of the polygons). So these maps induce a homeomorphism $B_{L,o}: \mathbb{H}^{2} \rightarrow \mathbb{H}^{2}$. So $\Gamma$ is an invariant set for the map $B_{L,o}$ and, by \textbf{Proposition \ref{subspaces}}, we get

$$\begin{array}{rcl}
    o(\hat{B}_{L,o}|_{\overline{\Gamma}}) & =& sup \{o(\hat{B}_{L,o}, \overline{\Gamma}^{i}), o(\hat{B}_{L,o},  \overline{\Gamma}^{2}_{j}): i \in \{1,2,3\}, j \in \{2,\dots,L'-1\}\}\\
     &= & \sup \{o(\hat{B}_{i}, \overline{\Gamma}^{i}), o(\hat{B}_{2},  \overline{\Gamma}^{2}_{j}): i \in \{1,2,3\}, j \in \{2,\dots,L'-1\}\}\\
     & = & \sup \{o(\hat{B}_{1} \overline{\Gamma}^{1}),o, o(\hat{B}_{3}, \overline{\Gamma}^{3}), [n],\dots,[n]\}\\
     & = & o,
\end{array}$$
since $o(\hat{B}_{1}, \overline{\Gamma}^{1}) \leqslant o$, $o(\hat{B}_{3}, \overline{\Gamma}^{3}) \leqslant o$ and $[n] < [n^{2}] \leqslant o$. Analogously, we have that $o(\hat{B}_{L,o}) = o$. Take $C = \Gamma^{2}$, and we have that $\Gamma^{2}$ does not intersect the sides of $\Gamma$, it is invariant by $\hat{B}_{L,o}$ and $o(\hat{B}_{L,o}|_{\overline{C}}) = o$. 
\end{proof}

So, we can establish the following:

\begin{prop}\label{singlenonwanderingonnonorientablesurface}Let $g \geqslant 1$, $P_{g}$ the non-orientable surface of genus $g$ and $o \in \overline{\mathbb{L}}$, with $[n^{2}] \leqslant o \leqslant \sup \mathbb{P}$. Then there exists a homeomorphism $f: P_{g} \rightarrow P_{g}$ such that $\#\Omega(f) = 1$ and $o(f) = o$. 
\end{prop}

\begin{proof}Consider the homeomorphism  $B_{L,o}: \mathbb{H}^{2} \rightarrow \mathbb{H}^{2}$ and an ideal closed polygon $\Gamma$ with $2g$ sides, satisfying the conditions of the last lemma. Consider $\gamma_{1},\dots ,\gamma_{2g}$ the sides of $\Gamma$, with $\overline{\gamma}_{i}$ the closure of $\gamma_{i}$ in $\hat{\mathbb{H}}^{2}$. Then, let us identify $\overline{\gamma}_{1}$ with $\overline{\gamma}_{2}$, $\overline{\gamma}_{3}$ with $\overline{\gamma}_{4}$, $\dots$, and $\overline{\gamma}_{2g-1}$ with $\overline{\gamma}_{2g}$ in $\overline{\Gamma}$, agreeing with the dynamics defined by $\hat{B}_{L,o}$. So, the space defined by $\overline{\Gamma}$ module this identification is the $g$-projective space $P_{g}$. Let $\pi: \overline{\Gamma} \rightarrow P_{g}$ be the quotient map. We have that the map $\hat{B}_{L,o}|_{\overline{\Gamma}}$ induces a homeomorphism $\beta_{g,o}: P_{g} \rightarrow P_{g}$ such that $\pi$ is a semi-conjugation between $\hat{B}_{L,o}|_{\overline{\Gamma}}$ and $\beta_{g,o}$. So, $\Omega(\beta_{g,o}) = \{\pi(\infty)\}$. Also, by \textbf{Proposition \ref{semiconjporrecobrimentocompacto}}, $o(\beta_{g,o}) \leqslant o(\hat{B}_{L,o}|_{\overline{\Gamma}}) = o$. 

On the other hand, $\overline{C}$ is a closed set in $\overline{\Gamma}$ that is invariant by $\hat{B}_{L,o}|_{\overline{\Gamma}}$ and $o(\hat{B}_{L,o}|_{\overline{C}}) = o$. Since $C$ does not intersect the sides of $\Gamma$, $\pi|_{\overline{C}}$ is a homeomorphism over its image, which implies that $o(\beta_{g,o}|_{\pi(\overline{C})}) = o$. So,
$o(\beta_{g,o}) \geqslant o(\beta_{g,o}|_{\pi(\overline{C})}) = o$ and, then, $o(\beta_{g,o}) = o$.

\begin{figure}[H]
    \centering
    \resizebox{6.2cm}{!}{

\tikzset{every picture/.style={line width=0.75pt}} 

\begin{tikzpicture}[x=0.75pt,y=0.75pt,yscale=-1,xscale=1]

\draw   (241,113.46) .. controls (241,64.1) and (281.02,24.08) .. (330.38,24.08) .. controls (379.74,24.08) and (419.76,64.1) .. (419.76,113.46) .. controls (419.76,162.83) and (379.74,202.84) .. (330.38,202.84) .. controls (281.02,202.84) and (241,162.83) .. (241,113.46) -- cycle ;
\draw  [draw opacity=0] (403.85,62.33) .. controls (390.13,71.68) and (371.69,73.31) .. (355.76,65.02) .. controls (339.83,56.72) and (330.59,40.68) .. (330.38,24.08) -- (377.07,24.08) -- cycle ; \draw  [color={rgb, 255:red, 208; green, 2; blue, 27 }  ,draw opacity=1 ] (403.85,62.33) .. controls (390.13,71.68) and (371.69,73.31) .. (355.76,65.02) .. controls (339.83,56.72) and (330.59,40.68) .. (330.38,24.08) ;  
\draw  [draw opacity=0] (414.18,144.81) .. controls (398.6,139.02) and (386.64,124.93) .. (384.3,107.2) .. controls (381.97,89.47) and (389.88,72.78) .. (403.42,63.14) -- (430.06,101.18) -- cycle ; \draw  [color={rgb, 255:red, 208; green, 2; blue, 27 }  ,draw opacity=1 ] (414.18,144.81) .. controls (398.6,139.02) and (386.64,124.93) .. (384.3,107.2) .. controls (381.97,89.47) and (389.88,72.78) .. (403.42,63.14) ;  
\draw  [draw opacity=0] (307.29,199.44) .. controls (311.71,183.41) and (324.7,170.27) .. (342.16,166.4) .. controls (359.62,162.53) and (376.94,168.95) .. (387.72,181.61) -- (352.15,211.46) -- cycle ; \draw  [color={rgb, 255:red, 208; green, 2; blue, 27 }  ,draw opacity=1 ] (307.29,199.44) .. controls (311.71,183.41) and (324.7,170.27) .. (342.16,166.4) .. controls (359.62,162.53) and (376.94,168.95) .. (387.72,181.61) ;  
\draw  [draw opacity=0] (330.38,24.08) .. controls (330.26,40.7) and (321.12,56.76) .. (305.26,65.02) .. controls (289.39,73.27) and (271,71.56) .. (257.31,62.12) -- (283.95,24.08) -- cycle ; \draw  [color={rgb, 255:red, 208; green, 2; blue, 27 }  ,draw opacity=1 ] (330.38,24.08) .. controls (330.26,40.7) and (321.12,56.76) .. (305.26,65.02) .. controls (289.39,73.27) and (271,71.56) .. (257.31,62.12) ;  
\draw  [draw opacity=0] (257.31,62.12) .. controls (270.86,71.75) and (278.77,88.45) .. (276.43,106.18) .. controls (274.1,123.91) and (262.14,137.99) .. (246.56,143.79) -- (230.68,100.16) -- cycle ; \draw  [color={rgb, 255:red, 208; green, 2; blue, 27 }  ,draw opacity=1 ] (257.31,62.12) .. controls (270.86,71.75) and (278.77,88.45) .. (276.43,106.18) .. controls (274.1,123.91) and (262.14,137.99) .. (246.56,143.79) ;  
\draw  [draw opacity=0] (246.56,143.79) .. controls (262.22,138.21) and (280.44,141.32) .. (293.62,153.4) .. controls (306.8,165.48) and (311.48,183.35) .. (307.29,199.44) -- (262.44,187.42) -- cycle ; \draw  [color={rgb, 255:red, 208; green, 2; blue, 27 }  ,draw opacity=1 ] (246.56,143.79) .. controls (262.22,138.21) and (280.44,141.32) .. (293.62,153.4) .. controls (306.8,165.48) and (311.48,183.35) .. (307.29,199.44) ;  
\draw  [color={rgb, 255:red, 208; green, 2; blue, 27 }  ,draw opacity=1 ][fill={rgb, 255:red, 208; green, 2; blue, 27 }  ,fill opacity=1 ] (303.1,66.04) -- (307.75,59.79) -- (310.89,65.82) -- cycle ;
\draw  [color={rgb, 255:red, 208; green, 2; blue, 27 }  ,draw opacity=1 ][fill={rgb, 255:red, 208; green, 2; blue, 27 }  ,fill opacity=1 ] (275.75,110.9) -- (273.29,103.51) -- (280.04,104.39) -- cycle ;
\draw  [color={rgb, 255:red, 208; green, 2; blue, 27 }  ,draw opacity=1 ][fill={rgb, 255:red, 208; green, 2; blue, 27 }  ,fill opacity=1 ] (296.79,155.79) -- (289.33,153.56) -- (293.92,148.55) -- cycle ;
\draw  [color={rgb, 255:red, 208; green, 2; blue, 27 }  ,draw opacity=1 ][fill={rgb, 255:red, 208; green, 2; blue, 27 }  ,fill opacity=1 ] (345.63,165.67) -- (339.52,170.5) -- (338.05,163.86) -- cycle ;
\draw  [color={rgb, 255:red, 208; green, 2; blue, 27 }  ,draw opacity=1 ][fill={rgb, 255:red, 208; green, 2; blue, 27 }  ,fill opacity=1 ] (383.75,102.95) -- (388.04,109.46) -- (381.29,110.35) -- cycle ;
\draw  [color={rgb, 255:red, 208; green, 2; blue, 27 }  ,draw opacity=1 ][fill={rgb, 255:red, 208; green, 2; blue, 27 }  ,fill opacity=1 ] (355.1,64.81) -- (362.89,65.03) -- (359.75,71.06) -- cycle ;

\draw (305.1,69.44) node [anchor=north west][inner sep=0.75pt]  [font=\footnotesize,color={rgb, 255:red, 208; green, 2; blue, 27 }  ,opacity=1 ]  {$\gamma _{1}$};
\draw (285,99.25) node [anchor=north west][inner sep=0.75pt]  [font=\footnotesize,color={rgb, 255:red, 208; green, 2; blue, 27 }  ,opacity=1 ]  {$\gamma _{2}$};
\draw (295,131.25) node [anchor=north west][inner sep=0.75pt]  [font=\footnotesize,color={rgb, 255:red, 208; green, 2; blue, 27 }  ,opacity=1 ]  {$\gamma _{3}$};
\draw (348,102.25) node [anchor=north west][inner sep=0.75pt]  [font=\footnotesize,color={rgb, 255:red, 208; green, 2; blue, 27 }  ,opacity=1 ]  {$\gamma _{2g-1}$};
\draw (344,69.25) node [anchor=north west][inner sep=0.75pt]  [font=\footnotesize,color={rgb, 255:red, 208; green, 2; blue, 27 }  ,opacity=1 ]  {$\gamma _{2g}$};
\draw (402.94,150.73) node [anchor=north west][inner sep=0.75pt]  [font=\footnotesize,rotate=-127.42]  {$\textcolor[rgb]{0.82,0.01,0.11}{\dots }$};
\draw (335,144.31) node [anchor=north west][inner sep=0.75pt]  [font=\footnotesize,color={rgb, 255:red, 208; green, 2; blue, 27 }  ,opacity=1 ]  {$\gamma _{4}$};

\end{tikzpicture}

    }
    \caption{Representation of the polygon $\gamma$ with $2g$ sides.}
    \label{fig:enter-label3}
\end{figure}
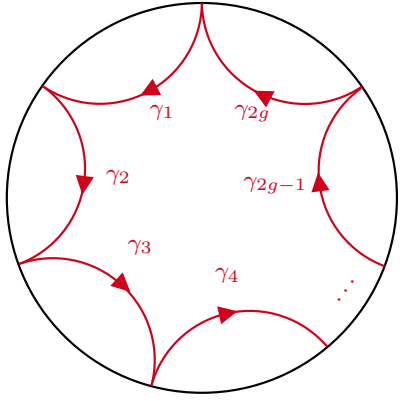

\end{proof}

\begin{prop}\label{singlenonwanderingonorientablesurface} Let $g \geqslant 1$, $S_{g}$ an orientable compact surface of genus $g$ and $o \in \overline{\mathbb{L}}$, with $[n^{2}] \leqslant o \leqslant \sup \mathbb{P}$. Then, there exists a homeomorphism $f: S_{g} \rightarrow S_{g}$ such that $\#\Omega(f) = 1$ and $o(f) = o$.  
\end{prop}

\begin{proof}Consider $g \geqslant 1$. Take $\Gamma^{1},\dots, \Gamma^{2g}$ polygons, where $\Gamma^{1}$ and $\Gamma^{2g}$ have three sides and the other ones have $4$ sides each, and homeomorphisms $B_{i}: \Gamma^{i} \rightarrow \Gamma^{i}$ such that $B_{i}$ is as described in the lemma above. We want that the edges of the polygons to be oriented by $B_{i}$ counter-clockwise if $i$ is odd and clockwise if $i$ is even and $o(\hat{B}_{i}) = o$ for every $i \in \{1,\dots2g\}$. So, we can glue the polygons identifying, equivariantly, the last edge of $\Gamma^{i}$ to the first edge of $\Gamma^{i+1}$, for every $i \in \{1,\dots,2g\}$ (the first edge of the $\Gamma^{1}$ and the last edge of $\Gamma^{2g}$ are not identified here). Then, we get a $4g$-agon $\Gamma$ and an induced homeomorphism $B: \Gamma \rightarrow \Gamma$ with $o(\hat{B}) = o$, $\Omega(\hat{B}) = \{\infty\}$ and the orientation of the edges change from clockwise to counter-clockwise and vice-versa every two vertices. So we can label the edges in sequence counter-clockwise as $a_{1},b_{1},a^{-1}_{2},b^{-1}_{1},\dots,a_{g},b_{g},a^{-1}_{g},b^{-1}_{g}$ and glue each pair $(a_{i},a^{-1}_{i})$ and $(b_{i},b^{-1}_{i})$, respecting the maps. So we get a $g$-torus $S_{g}$ and a map $\alpha_{g,o}$ which has $\#\Omega(\alpha_{g,o}) = 1$ and $o(\alpha_{g,o}) = o$, by an argument that is analogous to that given in the last proposition.


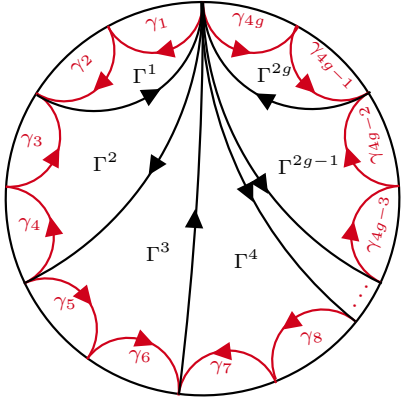
\begin{figure}[H]
    \centering
    \resizebox{8cm}{!}{

\tikzset{every picture/.style={line width=0.75pt}} 

\begin{tikzpicture}[x=0.75pt,y=0.75pt,yscale=-1,xscale=1]
\useasboundingbox (200,13) rectangle (460,213); %
\draw   (253.13,157.78) .. controls (228.22,115.16) and (242.57,60.42) .. (285.18,35.51) .. controls (327.8,10.6) and (382.54,24.95) .. (407.45,67.56) .. controls (432.37,110.18) and (418.02,164.92) .. (375.4,189.83) .. controls (332.78,214.75) and (278.04,200.39) .. (253.13,157.78) -- cycle ;
\draw  [draw opacity=0] (330.48,23.15) .. controls (330.53,33.52) and (323.62,42.98) .. (313.19,45.66) .. controls (302.76,48.34) and (292.15,43.38) .. (287.19,34.27) -- (307.41,23.15) -- cycle ; \draw  [color={rgb, 255:red, 208; green, 2; blue, 27 }  ,draw opacity=1 ] (330.48,23.15) .. controls (330.53,33.52) and (323.62,42.98) .. (313.19,45.66) .. controls (302.76,48.34) and (292.15,43.38) .. (287.19,34.27) ;  
\draw  [draw opacity=0] (287.33,34.13) .. controls (292.37,43.2) and (290.87,54.82) .. (283.02,62.19) .. controls (275.17,69.56) and (263.48,70.33) .. (254.75,64.73) -- (267.11,45.25) -- cycle ; \draw  [color={rgb, 255:red, 208; green, 2; blue, 27 }  ,draw opacity=1 ] (287.33,34.13) .. controls (292.37,43.2) and (290.87,54.82) .. (283.02,62.19) .. controls (275.17,69.56) and (263.48,70.33) .. (254.75,64.73) ;  
\draw  [draw opacity=0] (254.75,64.73) .. controls (263.53,70.25) and (267.82,81.15) .. (264.49,91.39) .. controls (261.16,101.64) and (251.29,107.94) .. (240.94,107.24) -- (242.38,84.21) -- cycle ; \draw  [color={rgb, 255:red, 208; green, 2; blue, 27 }  ,draw opacity=1 ] (254.75,64.73) .. controls (263.53,70.25) and (267.82,81.15) .. (264.49,91.39) .. controls (261.16,101.64) and (251.29,107.94) .. (240.94,107.24) ;  
\draw  [draw opacity=0] (240.94,107.24) .. controls (251.29,107.84) and (260.3,115.33) .. (262.32,125.91) .. controls (264.34,136.49) and (258.72,146.77) .. (249.31,151.15) -- (239.49,130.27) -- cycle ; \draw  [color={rgb, 255:red, 208; green, 2; blue, 27 }  ,draw opacity=1 ] (240.94,107.24) .. controls (251.29,107.84) and (260.3,115.33) .. (262.32,125.91) .. controls (264.34,136.49) and (258.72,146.77) .. (249.31,151.15) ;  
\draw  [draw opacity=0] (249.31,151.15) .. controls (258.68,146.69) and (270.18,148.91) .. (277.04,157.21) .. controls (283.91,165.51) and (283.94,177.22) .. (277.8,185.59) -- (259.14,172.02) -- cycle ; \draw  [color={rgb, 255:red, 208; green, 2; blue, 27 }  ,draw opacity=1 ] (249.31,151.15) .. controls (258.68,146.69) and (270.18,148.91) .. (277.04,157.21) .. controls (283.91,165.51) and (283.94,177.22) .. (277.8,185.59) ;  
\draw  [draw opacity=0] (373.78,34.27) .. controls (368.82,43.38) and (358.21,48.34) .. (347.78,45.66) .. controls (337.34,42.98) and (330.44,33.52) .. (330.48,23.15) -- (353.56,23.15) -- cycle ; \draw  [color={rgb, 255:red, 208; green, 2; blue, 27 }  ,draw opacity=1 ] (373.78,34.27) .. controls (368.82,43.38) and (358.21,48.34) .. (347.78,45.66) .. controls (337.34,42.98) and (330.44,33.52) .. (330.48,23.15) ;  
\draw  [draw opacity=0] (277.8,185.59) .. controls (283.86,177.16) and (295.01,173.57) .. (305.03,177.54) .. controls (315.04,181.5) and (320.71,191.75) .. (319.36,202.04) -- (296.47,199.15) -- cycle ; \draw  [color={rgb, 255:red, 208; green, 2; blue, 27 }  ,draw opacity=1 ] (277.8,185.59) .. controls (283.86,177.16) and (295.01,173.57) .. (305.03,177.54) .. controls (315.04,181.5) and (320.71,191.75) .. (319.36,202.04) ;  
\draw  [draw opacity=0] (319.36,202.04) .. controls (320.62,191.74) and (328.66,183.22) .. (339.34,181.87) .. controls (350.03,180.52) and (359.93,186.77) .. (363.71,196.44) -- (342.25,204.93) -- cycle ; \draw  [color={rgb, 255:red, 208; green, 2; blue, 27 }  ,draw opacity=1 ] (319.36,202.04) .. controls (320.62,191.74) and (328.66,183.22) .. (339.34,181.87) .. controls (350.03,180.52) and (359.93,186.77) .. (363.71,196.44) ;  
\draw  [draw opacity=0] (363.6,194.71) .. controls (359.73,185.08) and (362.68,173.74) .. (371.39,167.41) .. controls (380.1,161.08) and (391.79,161.78) .. (399.76,168.43) -- (385.05,186.21) -- cycle ; \draw  [color={rgb, 255:red, 208; green, 2; blue, 27 }  ,draw opacity=1 ] (363.6,194.71) .. controls (359.73,185.08) and (362.68,173.74) .. (371.39,167.41) .. controls (380.1,161.08) and (391.79,161.78) .. (399.76,168.43) ;  
\draw  [draw opacity=0] (411.25,150.77) .. controls (401.84,146.39) and (396.22,136.11) .. (398.24,125.53) .. controls (400.26,114.95) and (409.27,107.46) .. (419.62,106.86) -- (421.07,129.89) -- cycle ; \draw  [color={rgb, 255:red, 208; green, 2; blue, 27 }  ,draw opacity=1 ] (411.25,150.77) .. controls (401.84,146.39) and (396.22,136.11) .. (398.24,125.53) .. controls (400.26,114.95) and (409.27,107.46) .. (419.62,106.86) ;  
\draw  [draw opacity=0] (419.62,106.86) .. controls (409.27,107.56) and (399.4,101.26) .. (396.07,91.01) .. controls (392.74,80.77) and (397.03,69.87) .. (405.81,64.35) -- (418.18,83.83) -- cycle ; \draw  [color={rgb, 255:red, 208; green, 2; blue, 27 }  ,draw opacity=1 ] (419.62,106.86) .. controls (409.27,107.56) and (399.4,101.26) .. (396.07,91.01) .. controls (392.74,80.77) and (397.03,69.87) .. (405.81,64.35) ;  
\draw  [draw opacity=0] (405.81,64.35) .. controls (397.08,69.95) and (385.39,69.18) .. (377.54,61.81) .. controls (369.69,54.44) and (368.19,42.82) .. (373.23,33.75) -- (393.45,44.87) -- cycle ; \draw  [color={rgb, 255:red, 208; green, 2; blue, 27 }  ,draw opacity=1 ] (405.81,64.35) .. controls (397.08,69.95) and (385.39,69.18) .. (377.54,61.81) .. controls (369.69,54.44) and (368.19,42.82) .. (373.23,33.75) ;  
\draw  [draw opacity=0] (330.29,23.2) .. controls (330.32,40.72) and (321.07,57.7) .. (304.67,66.72) .. controls (288.54,75.58) and (269.61,74.49) .. (254.93,65.49) -- (280.75,23.2) -- cycle ; \draw   (330.29,23.2) .. controls (330.32,40.72) and (321.07,57.7) .. (304.67,66.72) .. controls (288.54,75.58) and (269.61,74.49) .. (254.93,65.49) ;  
\draw  [draw opacity=0] (406.57,64.12) .. controls (391.79,73.54) and (372.5,74.82) .. (356.1,65.8) .. controls (339.97,56.94) and (330.75,40.37) .. (330.48,23.15) -- (380.02,22.28) -- cycle ; \draw   (406.57,64.12) .. controls (391.79,73.54) and (372.5,74.82) .. (356.1,65.8) .. controls (339.97,56.94) and (330.75,40.37) .. (330.48,23.15) ;  
\draw  [color={rgb, 255:red, 208; green, 2; blue, 27 }  ,draw opacity=1 ][fill={rgb, 255:red, 208; green, 2; blue, 27 }  ,fill opacity=1 ] (309.72,46.33) -- (315.67,41.29) -- (317.36,47.88) -- cycle ;
\draw  [color={rgb, 255:red, 208; green, 2; blue, 27 }  ,draw opacity=1 ][fill={rgb, 255:red, 208; green, 2; blue, 27 }  ,fill opacity=1 ] (281.56,62.86) -- (284.35,55.58) -- (289,60.54) -- cycle ;
\draw  [color={rgb, 255:red, 208; green, 2; blue, 27 }  ,draw opacity=1 ][fill={rgb, 255:red, 208; green, 2; blue, 27 }  ,fill opacity=1 ] (264.77,90.36) -- (264.48,98.15) -- (258.47,94.96) -- cycle ;
\draw  [color={rgb, 255:red, 208; green, 2; blue, 27 }  ,draw opacity=1 ][fill={rgb, 255:red, 208; green, 2; blue, 27 }  ,fill opacity=1 ] (261.46,122.01) -- (266.12,128.27) -- (259.44,129.54) -- cycle ;
\draw  [color={rgb, 255:red, 208; green, 2; blue, 27 }  ,draw opacity=1 ][fill={rgb, 255:red, 208; green, 2; blue, 27 }  ,fill opacity=1 ] (279.36,160.16) -- (272.27,156.92) -- (277.5,152.59) -- cycle ;
\draw  [color={rgb, 255:red, 208; green, 2; blue, 27 }  ,draw opacity=1 ][fill={rgb, 255:red, 208; green, 2; blue, 27 }  ,fill opacity=1 ] (307.38,178.75) -- (299.61,179.33) -- (302.11,173.01) -- cycle ;
\draw  [color={rgb, 255:red, 208; green, 2; blue, 27 }  ,draw opacity=1 ][fill={rgb, 255:red, 208; green, 2; blue, 27 }  ,fill opacity=1 ] (335.64,182.9) -- (342.17,178.65) -- (343.02,185.39) -- cycle ;
\draw  [color={rgb, 255:red, 208; green, 2; blue, 27 }  ,draw opacity=1 ][fill={rgb, 255:red, 208; green, 2; blue, 27 }  ,fill opacity=1 ] (367.28,171.52) -- (370.96,164.65) -- (374.96,170.15) -- cycle ;
\draw  [color={rgb, 255:red, 208; green, 2; blue, 27 }  ,draw opacity=1 ][fill={rgb, 255:red, 208; green, 2; blue, 27 }  ,fill opacity=1 ] (398.78,124.01) -- (400.8,131.54) -- (394.12,130.27) -- cycle ;
\draw  [color={rgb, 255:red, 208; green, 2; blue, 27 }  ,draw opacity=1 ][fill={rgb, 255:red, 208; green, 2; blue, 27 }  ,fill opacity=1 ] (396.04,90.12) -- (401.44,95.74) -- (394.97,97.84) -- cycle ;
\draw  [color={rgb, 255:red, 208; green, 2; blue, 27 }  ,draw opacity=1 ][fill={rgb, 255:red, 208; green, 2; blue, 27 }  ,fill opacity=1 ] (380.68,63.86) -- (373.24,61.54) -- (377.89,56.58) -- cycle ;
\draw  [color={rgb, 255:red, 208; green, 2; blue, 27 }  ,draw opacity=1 ][fill={rgb, 255:red, 208; green, 2; blue, 27 }  ,fill opacity=1 ] (351.52,46.33) -- (343.88,47.88) -- (345.57,41.29) -- cycle ;
\draw  [fill={rgb, 255:red, 0; green, 0; blue, 0 }  ,fill opacity=1 ] (310.36,63.33) -- (305.86,69.69) -- (302.58,63.73) -- cycle ;
\draw  [fill={rgb, 255:red, 0; green, 0; blue, 0 }  ,fill opacity=1 ] (355.22,65.33) -- (363,65.73) -- (359.73,71.69) -- cycle ;
\draw  [draw opacity=0] (330.29,23.2) .. controls (330.19,49.11) and (323.01,75.35) .. (308.07,98.88) .. controls (293.13,122.42) and (272.46,140.09) .. (249.05,151.22) -- (188.81,23.2) -- cycle ; \draw   (330.29,23.2) .. controls (330.19,49.11) and (323.01,75.35) .. (308.07,98.88) .. controls (293.13,122.42) and (272.46,140.09) .. (249.05,151.22) ;  
\draw  [draw opacity=0] (411.73,151.17) .. controls (388.32,140.04) and (367.64,122.37) .. (352.7,98.84) .. controls (337.77,75.3) and (330.58,49.07) .. (330.48,23.15) -- (471.97,23.15) -- cycle ; \draw   (411.73,151.17) .. controls (388.32,140.04) and (367.64,122.37) .. (352.7,98.84) .. controls (337.77,75.3) and (330.58,49.07) .. (330.48,23.15) ;  
\draw  [draw opacity=0] (330.48,23.15) .. controls (330.51,52.8) and (329.61,82.63) .. (327.78,112.63) .. controls (325.94,142.63) and (323.19,172.35) .. (319.56,201.77) -- (-1094.56,25.64) -- cycle ; \draw   (330.48,23.15) .. controls (330.51,52.8) and (329.61,82.63) .. (327.78,112.63) .. controls (325.94,142.63) and (323.19,172.35) .. (319.56,201.77) ;  
\draw  [draw opacity=0] (399.76,168.43) .. controls (378.85,151.68) and (361.22,130.15) .. (348.7,104.48) .. controls (335.97,78.38) and (329.87,50.78) .. (329.72,23.56) -- (517.3,22.25) -- cycle ; \draw   (399.76,168.43) .. controls (378.85,151.68) and (361.22,130.15) .. (348.7,104.48) .. controls (335.97,78.38) and (329.87,50.78) .. (329.72,23.56) ;  
\draw  [fill={rgb, 255:red, 0; green, 0; blue, 0 }  ,fill opacity=1 ] (306.41,100.98) -- (307.3,93.24) -- (313.04,96.88) -- cycle ;
\draw  [fill={rgb, 255:red, 0; green, 0; blue, 0 }  ,fill opacity=1 ] (326.9,118.46) -- (329.86,125.67) -- (323.08,125.25) -- cycle ;
\draw  [fill={rgb, 255:red, 0; green, 0; blue, 0 }  ,fill opacity=1 ] (352.83,113.17) -- (346.7,108.36) -- (352.81,105.38) -- cycle ;
\draw  [fill={rgb, 255:red, 0; green, 0; blue, 0 }  ,fill opacity=1 ] (359.44,109.48) -- (352.81,105.38) -- (358.55,101.73) -- cycle ;

\draw (408.19,151.16) node [anchor=north west][inner sep=0.75pt]  [font=\footnotesize,rotate=-122.4]  {$\textcolor[rgb]{0.82,0.01,0.11}{\dots }$};
\draw (302.08,30.73) node [anchor=north west][inner sep=0.75pt]  [font=\scriptsize,color={rgb, 255:red, 208; green, 2; blue, 27 }  ,opacity=1 ,rotate=-345.83]  {$\gamma _{1}$};
\draw (267.52,52.83) node [anchor=north west][inner sep=0.75pt]  [font=\scriptsize,color={rgb, 255:red, 208; green, 2; blue, 27 }  ,opacity=1 ,rotate=-316.71]  {$\gamma _{2}$};
\draw (246.12,81.86) node [anchor=north west][inner sep=0.75pt]  [font=\scriptsize,color={rgb, 255:red, 208; green, 2; blue, 27 }  ,opacity=1 ]  {$\gamma _{3}$};
\draw (245.12,119.86) node [anchor=north west][inner sep=0.75pt]  [font=\scriptsize,color={rgb, 255:red, 208; green, 2; blue, 27 }  ,opacity=1 ]  {$\gamma _{4}$};
\draw (261.12,155.86) node [anchor=north west][inner sep=0.75pt]  [font=\scriptsize,color={rgb, 255:red, 208; green, 2; blue, 27 }  ,opacity=1 ]  {$\gamma _{5}$};
\draw (295.12,179.86) node [anchor=north west][inner sep=0.75pt]  [font=\scriptsize,color={rgb, 255:red, 208; green, 2; blue, 27 }  ,opacity=1 ]  {$\gamma _{6}$};
\draw (332.12,184.86) node [anchor=north west][inner sep=0.75pt]  [font=\scriptsize,color={rgb, 255:red, 208; green, 2; blue, 27 }  ,opacity=1 ]  {$\gamma _{7}$};
\draw (372.63,172.5) node [anchor=north west][inner sep=0.75pt]  [font=\scriptsize,color={rgb, 255:red, 208; green, 2; blue, 27 }  ,opacity=1 ,rotate=-350.96]  {$\gamma _{8}$};
\draw (402.98,138.2) node [anchor=north west][inner sep=0.75pt]  [font=\scriptsize,color={rgb, 255:red, 208; green, 2; blue, 27 }  ,opacity=1 ,rotate=-281.26]  {$\gamma _{4g-3}$};
\draw (406.32,99.33) node [anchor=north west][inner sep=0.75pt]  [font=\scriptsize,color={rgb, 255:red, 208; green, 2; blue, 27 }  ,opacity=1 ,rotate=-254.49]  {$\gamma _{4g-2}$};
\draw (381.18,38.43) node [anchor=north west][inner sep=0.75pt]  [font=\scriptsize,color={rgb, 255:red, 208; green, 2; blue, 27 }  ,opacity=1 ,rotate=-44.65]  {$\gamma _{4g-1}$};
\draw (342.12,26.86) node [anchor=north west][inner sep=0.75pt]  [font=\scriptsize,color={rgb, 255:red, 208; green, 2; blue, 27 }  ,opacity=1 ]  {$\gamma _{4g}$};
\draw (297,50.32) node [anchor=north west][inner sep=0.75pt]  [font=\scriptsize]  {$\Gamma ^{1}$};
\draw (279,90.32) node [anchor=north west][inner sep=0.75pt]  [font=\scriptsize]  {$\Gamma ^{2}$};
\draw (303,131.32) node [anchor=north west][inner sep=0.75pt]  [font=\scriptsize]  {$\Gamma ^{3}$};
\draw (343,134.32) node [anchor=north west][inner sep=0.75pt]  [font=\scriptsize]  {$\Gamma ^{4}$};
\draw (363,91.32) node [anchor=north west][inner sep=0.75pt]  [font=\scriptsize]  {$\Gamma ^{2g-1}$};
\draw (353.52,49.73) node [anchor=north west][inner sep=0.75pt]  [font=\scriptsize]  {$\Gamma ^{2g}$};

\end{tikzpicture} }
    \caption{Representation of the polygon $\Gamma$ with $4g$ sides (in red). It is the gluing of the polygons $\Gamma^{1},\dots, \Gamma^{2g}$, also represented in the figure.}
    \label{fig:enter-label4}
\end{figure}
\end{proof}

Summarizing \textbf{Propositions \ref{singlenonwanderingonnonorientablesurface} and \ref{singlenonwanderingonorientablesurface}}, we establish:

\begin{teo}\label{singlenonwanderingonallsurfaces}Let $S$ be a compact surface without boundary and let $o \in \overline{\mathbb{L}}$, with $[n^{2}] \leqslant o \leqslant \sup \mathbb{P}$. Then there exists a homeomorphism $f: S \rightarrow S$ such that $\#\Omega(f) = 1$ and $o(f) = o$.
\end{teo}

Note that maybe we can mimic this process of gluing polygons and identifying edges to obtain different identifications of polygons. This would lead to other dynamics on compact surfaces with finite non-wandering set. The cardinality of this set would correspond to the number of vertex cycles in the polygon identification, which is $1$ in the examples shown above.

\section{Some other examples of parabolic dynamics}\label{6}

To illustrate the breadth of parabolic and north-south dynamics, we conclude with a collection of examples across various topological settings. These include constructions in non-metrizable spaces such as the double arrow space and suspensions and more geometric settings like convergence actions. These examples highlight the natural occurrence of parabolic behavior beyond the classical metric context.

\subsection{Double arrow space}

Consider the space $[0,1]\times \{0,1\}$ with the topology defined by the lexicographical order. Then the set $\{((x,y)\times \{0\})\cup([x,y)\times \{1\}): y > x\}$ is a system of neighborhoods of the point $(x,1)$ and the set $\{((y,x]\times \{0\})\cup((y,x)\times \{1\}): y < x\}$ is a system of neighborhoods of the point $(x,0)$. We have that the points $(0,0)$ and $(1,1)$ are isolated points. So, we can remove these points and get a topological space $X$. This is the Double arrow space. It is Hausdorff, compact, and non-metrizable since it is not second countable (all of these properties are given in Counterexample $\#95$ of \cite{StS}, with the name of weak parallel line topology). Let $g: [0,1] \rightarrow [0,1]$ be a homeomorphism with north-south dynamics with repulsor point $0$ and attractor point $1$. Then we define $f: X \rightarrow X$ as $f(x,y) = (g(x),y)$. The map $f$ is an increasing map, which implies that it is a homeomorphism. It is easy to see that $f$ has north-south dynamics, with repulsor point $(0,1)$ and attractor point $(1,0)$. By \textbf{Corollary \ref{northsouthhaslinearentropy}}, we get that $o(f) = [n]$.

\subsection{Suspensions}

Let $X$ be a Hausdorff compact space and let $f: [0,1] \rightarrow [0,1]$ a north-south dynamic with repulsor point $0$ and attractor point $1$. Then, $id_{X} \times f: X \times [0,1] \rightarrow X \times [0,1]$ has as its fixed points the set $(X \times \{0\}) \cup (X \times \{1\})$. Collapsing each of the sets $X \times \{0\}$ and $X \times \{1\}$, we get the suspension $\Sigma X$. It is easy to see that the map $id_{X} \times f$ induces a homeomorphism $f': \Sigma X \rightarrow \Sigma X$, which has north-south dynamics. Then, by \textbf{Corollary \ref{northsouthhaslinearentropy}}, $o(f') = [n]$. Note that $X$ is metrizable if and only if $\Sigma X$ is metrizable. So we get a family of examples of north-south dynamics on non-metrizable spaces. 

\subsection{One-point compactifications}

Let $X$ be a Hausdorff locally compact space, $f: X \rightarrow X$ a homeomorphism, $\hat{X}$ the one-point compactification of $X$ and $\hat{f}: \hat{X} \rightarrow \hat{X}$ the extension that fixes the point at infinite. If $f$ is properly discontinuous, then the induced map $\hat{f}$ is parabolic, by \textbf{Proposition \ref{properlydiscontinuousisparabolic}}. So \textbf{Theorem \ref{parabolicshavelinearentropy}} says that $o(\hat{f}) = [n]$.

As a special case, let $\{X_{\alpha}\}_{\alpha \in \Gamma}$ be a family of Hausdorff locally compact spaces, and take $X$ as the disjoint union $\dot{\bigcup}_{\alpha \in \Gamma}X_{\alpha}$ with the disjoint union topology. If $\Gamma = \N$ and each $X_{\alpha}$ is homeomorphic to $\mathbb{R}$, then we get that $\hat{X}$ is the Hawaiian Earring and if each $X_{\alpha}$ is homeomorphic to $\mathbb{R}$ but $\Gamma$ is arbitrary, then we get that $\hat{X}$ is a big Hawaiian Earring (see \cite{CC}). Note that $X$ is metrizable if and only if $\Gamma$ is countable and for every $\alpha \in \Gamma$, $X_{\alpha}$ is countable. For $\alpha \in \Gamma$, take $f_{\alpha}: X_{\alpha} \rightarrow X_{\alpha}$ a properly discontinuous homeomorphism. This family of maps induces a properly discontinuous homeomorphism $f: X \rightarrow X$, and then we get a homeomorphism $\hat{f}$ with parabolic dynamics (which has linear generalized entropy, by \textbf{Theorem \ref{parabolicshavelinearentropy}}).

\subsection{Convergence actions}\label{convergence}

Another source of examples comes from Geometric Group Theory.

\begin{defi}Let $G$ be a discrete group and $X$ a Hausdorff compact space. An action by homeomorphisms of $G \curvearrowright X$ has the convergence property if for every infinite set $S \subseteq G$, there exists an infinite subset $S' \subseteq S$ and there exists $x_{0},x_{1} \in X$ such that for every compact $K \subseteq X-\{x_{0}\}$ and every compact $K' \subseteq X-\{x_{1}\}$, the set $\{g \in S': g K \cap K' \neq \emptyset\}$ is finite. 
\end{defi}

Examples of such actions are the actions of a finitely generated group on its space of ends, a hyperbolic group on its Gromov boundary, and a relatively hyperbolic group on its Bowditch boundaries  (for more information about this, see \cite{Bo2}).

If there is a convergence action $G \curvearrowright X$, with $\# X > 2$, and $g \in G$, then we say that:

\begin{enumerate}
    \item $g$ is elliptic if it has finite order.
    \item $g$ is parabolic if it has infinite order and it has only one fixed point in $X$.
    \item $g$ is loxodromic if it has infinite order and it has only two fixed points in $X$.
\end{enumerate}

Lemma 3.1 of \cite{Bo2} says that these are all the possibilities for an element in $G$. If $g$ has infinite order, then $\langle g \rangle$ acts on $X-Fix(g)$ properly discontinuously (page 34 of \cite{Bo2}). So, if $g$ is a parabolic element, then  it has parabolic dynamics and if $g$ is a loxodromic element, then it has north-south dynamics. So we get:

\begin{prop}Let $G$ be a finitely generated group, $G \curvearrowright X$ a convergence action and $g \in G$. Then

\begin{enumerate}
    \item If $g$ is elliptic, then $o(g) = 0$.
    \item If $g$ is parabolic or loxodromic, then $o(g) = [n]$.
\end{enumerate}
\end{prop}

\begin{proof}If $g$ is an elliptic element, then it has finite order. So $o(g) = 0$, by \textbf{Proposition \ref{periodichaszeroentropy}}. If $g$ is parabolic or loxodromic, then it has, respectively, a parabolic or north-south dynamics, which implies, by \textbf{Theorem \ref{parabolicshavelinearentropy}} and \textbf{Corollary \ref{northsouthhaslinearentropy}}, that $o(g) =  [n]$.
\end{proof}

\end{document}